\documentclass[reqno,10pt]{amsart}

\usepackage[noadjust]{cite}
\usepackage{amssymb,amsmath,amsthm,}
\usepackage{amscd}
\usepackage{amsfonts}
\usepackage{amssymb}
\usepackage{latexsym}
\usepackage{color}
\usepackage{esint}
\usepackage{graphicx}
\usepackage{caption}
\usepackage{subcaption}
\usepackage[normalem]{ulem}
\usepackage{amsthm}
\usepackage{verbatim}
\usepackage{MnSymbol} 
\usepackage[colorlinks]{hyperref}

\usepackage[makeroom]{cancel}
\usepackage{mathtools}
\usepackage{pgfplots}

\let\hide\iffalse

\newtheorem{theorem}{Theorem}[section]

\newtheorem{corollary}[theorem]{Corollary}

\newtheorem{lemma}[theorem]{Lemma}
\newtheorem{proposition}[theorem]{Proposition}
\newtheorem{remark}[theorem]{Remark}

\let\e=\varepsilon

\let\p=\partial

\newcommand{\R}{\mathbb{R}}
\newcommand{\T}{\mathbb{T}}

\renewcommand{\S}{\mathbb{S}}

\newcommand{\be}{\begin{equation}}
	\newcommand{\bm}{\begin{multline}}
		\newcommand{\ee}{\end{equation}}

	\newcommand{\Bes}{\begin{eqnarray*}}
		\newcommand{\Ees}{\end{eqnarray*}}
	\newcommand{\Be}{\begin{equation}}
		\newcommand{\Ee}{\end{equation}}
	
	\def\p{\partial}

	\def\R{\mathbb{R}}
	
	\def\B{\begin{equation}}
		\def\E{\end{equation}}
	\def\BN{\begin{eqnarray*}}
		\def\EN{\end{eqnarray*}}
	
	\numberwithin{equation}{section}
	
	\usepackage{tikz-cd}
	\usepackage{ragged2e}
	\usepackage{etoolbox}
	\usepackage{lipsum}
	\usepackage[normalem]{ulem}
	\usepackage{cancel}
	\usepackage{color}

	\allowdisplaybreaks

\begin{document}
	\title[Large-amplitude Multi-species Boltzmann equation]{Global Well-posedness for the Multi-Species Boltzmann Equation with Large Amplitude Initial Data}
	
	\author{Gyounghun Ko}
\address{Academy of Mathematics and Systems Science, Chinese Academy of Sciences, Beijing, 100190, China}
\email{gyeonghungo@amss.ac.cn}

	\author{Myeong-Su Lee} 
	\address{Research Institute of Mathematics, Seoul National University}
 	\email{msl3573@snu.ac.kr}
	
	\author{Sung-Jun Son}
	\address{Center for Mathematical Machine Learning and its Applications (CM2LA), Department of Mathematics, Pohang University of Science and Technology} 
	\email{sungjun129@postech.ac.kr}
	\begin{abstract}	
This paper establishes the global well-posedness of the multi-species Boltzmann equation with large-amplitude initial data in the periodic domain $\mathbb{T}^3$. In contrast to the single-species case, the multi-species mixture model lacks structural symmetry in its collision operators due to the distinct masses of different species. This asymmetry makes it difficult to obtain pointwise estimates for the nonlinear collision terms. Although the Carleman representation for the mixture model introduced in \cite{BD2016} provides a useful reduction of the collision integral, it does not directly yield the desired estimate. To overcome this difficulty, we identify an additional algebraic cancellation structure which leads to the pointwise estimates for the nonlinear terms. By applying this refined approach, we derive the necessary velocity-weighted $L^\infty$ estimates for the nonlinear terms. Furthermore, under the smallness assumption on the initial relative entropy, we establish a uniform lower bound for the nonlinear collision frequency and prove that the large-amplitude solutions exist globally in time and decay exponentially to the global equilibrium.
	\end{abstract}
	\maketitle
	\tableofcontents
	\section{Introduction} 
	The Boltzmann equation is one of the fundamental kinetic models describing collisional particle systems in a statistical manner. In the context of dilute gases, it describes the evolution of the molecular distribution function over time. However, many physical systems of practical interest consist of mixtures of several interacting species rather than a single gas component. In such cases, the kinetic description requires a system of Boltzmann equations, where each species is characterized by its own distribution function and interacts with other species through cross-collision operators. The presence of multiple species with different molecular masses and collision mechanisms leads to a substantially more complex mathematical structure, particularly with respect to conservation laws and entropy dissipation. As a consequence, analytical techniques developed for the monatomic Boltzmann equation do not directly extend to the setting of gas mixtures. From a mathematical perspective, the analysis of multi-species Boltzmann systems is further complicated by the strong coupling induced by cross-collision operators and the asymmetry of the particle masses.
	
	Motivated by this issue, we study the initial value problem for the multi-species Boltzmann equation. We consider a gas mixture in the periodic domain $\mathbb{T}^3$ with velocity variable $v \in \mathbb{R}^3$ and time $t \ge 0$. The multi-species Boltzmann equation for $N$ species is given by 
	\begin{equation}\label{mixed bol}
		\partial_{t}F_{i} + v \cdot \nabla_{x}F_{i} = \sum_{j=1}^{N} Q_{ij}(F_{i}, F_{j}), \quad \forall i \in \{1,\cdots, N\},
	\end{equation}
	where $F_i=F_i(t, x, v) \ge 0$ denotes the distribution function of the $i$-th species. The collision operator $Q_{ij}(F_i, F_j)$ is defined as the integral:
	\begin{equation*}
		Q_{ij}(F_i, F_j)(v) = \int_{\R^3 \times \S^2} B_{ij}(v-v_*, \sigma) [F_i(v')F_j(v'_*) - F_i(v)F_j(v_*)] d\sigma dv_*, 
	\end{equation*}
	where $v, v_*$ are the pre-collision velocities and $v', v'_*$ are the post-collisional velocities. For a binary collision between a particle of species $i$ with mass $m_i$ and a particle of species $j$ with mass $m_j$, the post-collisional velocities $(v', v'_*)$ are determined by the conservation of momentum and energy:
	\begin{align*}
		m_i v + m_j v_* &= m_i v' + m_j v'_*, \\
		\frac{1}{2}m_i |v|^2 + \frac{1}{2}m_j |v_*|^2 &= \frac{1}{2}m_i |v'|^2 + \frac{1}{2}m_j |v'_*|^2.
	\end{align*}
	More explicitly, they can be written as: 
	\begin{align} \label{post-collision v}
		v' &= \frac{1}{m_i+m_j}(m_iv + m_j v_* +m_j|v-v_*|\sigma) \quad \textrm{and} \quad 
		v'_* &= \frac{1}{m_i+m_j}(m_iv + m_j v_* -m_i|v-v_*|\sigma),
	\end{align}
	where $\sigma \in \S^2$ is the unit scattering vector. Following the paper \cite{BD2016}, we assume that the collision kernel $B_{ij}$ depends only on $|v-v_*|$ and the scattering angle $\theta$. We consider collision kernels satisfying Grad's angular cutoff assumption of the form
	\begin{equation*}
		B_{ij}(v-v_*,\sigma)=B_{ij}(|v-v_*|, \cos \theta) = \Phi_{ij}(|v-v_*|) b_{ij}(\cos \theta),
	\end{equation*}
	where the scattering angle $\theta \in [0, \pi]$ is defined through
	\begin{equation*}
		\cos \theta = \left\langle \frac{v-v_*}{|v-v_*|}, \sigma \right\rangle.
	\end{equation*}
	Moreover, we assume the following conditions for all $i,j \in \{1, \cdots , N\}$:
	\begin{align} \label{kernel}
		&(1) \Phi_{ij}(|z|) = C^\Phi_{ij} |z|^\gamma, \quad C_{ij}^{\Phi} >0, \;\gamma \in [0, 1]\;  \textrm{(hard potentials)}.\cr
		&(2) b_{ij}(\cos \theta) \textrm{ is a nonnegative measurable function satisfying } 0 \le b_{ij}(\cos \theta) \le C^b |\cos \theta|.\cr
		&(3) \textrm{The kernels satisfy symmetry condition} \; B_{ij} = B_{ji}.
	\end{align}
	The collision operator $Q_{ij}(F_i, F_j)$ can be decomposed into the gain and loss parts as $Q_{ij} = Q_{ij}^+ - Q_{ij}^-$, where
	\begin{align*}
		Q_{ij}^+(F_i, F_j) &= \int_{\R^3 \times \S^2} B_{ij}(v-v_*, \sigma) F_i(v') F_j(v'_*) d\sigma dv_*, \\
		Q_{ij}^-(F_i, F_j) &= \int_{\R^3 \times \S^2} B_{ij}(v-v_*, \sigma) F_i(v) F_j(v_*) d\sigma  dv_*.
	\end{align*}
	
\subsection{Notation}
	To avoid ambiguity in the multi-species setting, we summarize below several conventions that will be used throughout the paper.
	\begin{itemize}
		\item We denote a vector-valued function by $\pmb{f} = (f_1, \dots, f_N)$.	
		\item All species masses are assumed to be positive and fixed, so that any dependence on the masses can be absorbed into generic constants.
		\item  We use the notation $A(\pmb{f})$ to denote a scalar-valued operator depending on $\pmb{f} = (f_1,\dots,f_N)$, that is, its value depend on all components of $\pmb{f}$.		
		\item Generic positive constants, denoted by $C$, may vary from line to line. Their dependence on parameters will be indicated when relevant (e.g., $C_\delta$). 		
		\item For vector-valued functions $\pmb{f}=(f_1,\cdots, f_N)$ and $\pmb{g}=(g_1,\cdots, g_N)$, we define 
		\[
		\pmb{f}(v)\cdot \pmb{g}(v) := \sum_{j=1}^N f_j(v)\,g_j(v)
		\]
		and 
		\[
		\langle \pmb{f}, \pmb{g} \rangle_{L^2_v}
		:= \int_{\mathbb{R}^3} \pmb{f}(v)\cdot \pmb{g}(v)\,dv
		= \sum_{j=1}^N \int_{\mathbb{R}^3} f_j(v)\,g_j(v)\,dv.
		\]
		\item For $\pmb{F}=(F_1,\cdots, F_N)$, we define the $L^\infty$ norm by
\[
\|\pmb{F}\|_{L^\infty} := \sum_{i=1}^N \|F_i\|_{L^\infty},
\]
where 
\[
\|F_i\|_{L^\infty} := \operatorname*{ess\,sup}_{x \in \mathbb{T}^3,\, v \in \mathbb{R}^3} |F_i(x,v)|.
\]
More generally, for $1 \le p \le \infty$, the $L^p$ norms are taken with respect to $(x,v)$ unless otherwise specified. 
When the domain of integration is restricted, we indicate it explicitly; for instance, 
$\|\cdot\|_{L^p_v}$ denotes the $L^p$ norm taken only with respect to the velocity variable $v$.		
\item We define the velocity weight function 
\begin{align} \label{v_weight} 
w(v) := \langle v \rangle^q, \quad \textrm{for} \;  q>4,
\end{align}
where $\langle v \rangle := (1+|v|^2)^{1/2}$ denotes the Japanese bracket. 
For $1 \le p \le \infty$, we define the weighted norm by
\[
\|f\|_{L^p(w)} := \| w f \|_{L^p}.
\]
	\end{itemize}

\subsection{Conservation laws and H-theorem}
The structure of the multi-species Boltzmann equation is determined by the collision operator $Q_{ij}$ and the associated entropy dissipation. For any test functions $\pmb{\psi}(v) = (\psi_1, \dots, \psi_N)$, we have the following weak formulation:
\begin{align*}
	&\sum_{i,j=1}^N \int_{\mathbb{R}^3} Q_{ij}(F_i, F_j) \psi_i(v) \, dv \\
	&= \frac{1}{4} \sum_{i,j=1}^N \int_{ \mathbb{R}^3  \times \mathbb{R}^3\times \S^2} B_{ij}(v-v_*,\sigma) [F_i(v') F_j(v'_*) - F_i(v) F_j(v_*)] [\psi_i + \psi_{j,*} - \psi_i' - \psi_{j,*}'] \, d\sigma dv_* dv,
\end{align*}
where $\psi_i=\psi_i(v), \psi_{j,*}=\psi_{j}(v_*), \psi_i'=\psi_i(v'),$ and $\psi_{j,*}'=\psi_{j}(v'_*)$. By choosing $\pmb{\psi}$ such that 
\begin{align*}
\psi_i + \psi_{j,*} = \psi_i' + \psi_{j,*}', \quad \forall i,j \in \{1, \cdots, N \},
\end{align*}
we obtain the collision invariants, which lead to the following conservation laws:
\begin{align*}
	 \pmb{\psi} = (\delta_{1k}, \dots, \delta_{Nk}) &\rightarrow \frac{d}{dt} \int_{\mathbb{T}^3 \times \mathbb{R}^3} F_k \, dvdx= 0, \quad \forall k \in \{1, \dots, N\}, \\
	\pmb{\psi} = (m_1 v, \dots, m_N v) &\rightarrow \frac{d}{dt} \sum_{i=1}^N \int_{\mathbb{T}^3 \times \mathbb{R}^3} m_i v F_i \, dvdx = 0, \\ 
	 \pmb{\psi} = (\frac{1}{2} m_1 |v|^2, \dots, \frac{1}{2} m_N |v|^2) &\rightarrow \frac{d}{dt} \sum_{i=1}^N \int_{\mathbb{T}^3 \times \mathbb{R}^3} \frac{1}{2} m_i |v|^2 F_i \, dvdx = 0.
\end{align*}
This yields conservation of the total number density $n_{\infty,i}$ of each species, of the total momentum $\rho_{\infty}u_{\infty}$, and its total energy $\frac{3}{2} \rho_{\infty} T_{\infty}$: for all $t \geq0$, 
\begin{align*}
	n_{\infty,i} &= \int_{\mathbb{T}^3 \times \mathbb{R}^3} F_{i}(t,x,v) \, dvdx, \\
	\rho_\infty u_\infty &= \sum_{i=1}^N \int_{\mathbb{T}^3 \times \mathbb{R}^3} m_i v F_{i}(t,x,v)\, dvdx, \\
	3\rho_\infty T_\infty &= \sum_{i=1}^N \int_{\mathbb{T}^3 \times \mathbb{R}^3} m_i |v - u_\infty|^2 F_{i}(t,x,v) \, dvdx,
\end{align*}
where $\rho_\infty= \sum_{i=1}^N m_i n_{\infty,i}$ denotes the global density of the gas.\\
\indent The entropy structure plays a crucial role in the relaxation of the multi-species Boltzmann system. The total entropy $E(\pmb{F})$ is defined by
\begin{equation} \label{entropy}
	E(\pmb{F}) := \sum_{i=1}^{N} \int_{\mathbb{T}^3 \times \mathbb{R}^3} F_i \log F_i \, dvdx.
\end{equation}
The H-theorem for mixtures states that the entropy production is non-positive:
\begin{equation*}
	\frac{d}{dt}E(\pmb{F}) = D(\pmb{F}) \le 0.
\end{equation*}
Moreover, the entropy production $D(\pmb{F})$ vanishes if and only if each $F_i$ takes the form of a local Maxwellian: 
\begin{equation*}
	M_i(t,x,v) = n_i(t,x) \left( \frac{m_i}{2\pi k_B T(t,x)} \right)^{3/2} \exp \left( -\frac{m_i |v - u(t,x)|^2}{2 k_B T(t,x)} \right),
\end{equation*}
where $n_i$ denotes the local number density of species $i$, while $u$ and $T$ are the common local bulk velocity and temperature, respectively. In the periodic box $\T^3$, the multi-species H-theorem ensures that the stationary solution to \eqref{mixed bol} associated with the initial data is uniquely given by a global Maxwellian $\pmb{\mu} = (\mu_1, \dots, \mu_N)$. The general form of the global Maxwellian is given by
\begin{equation*}
	\mu_i(v) = n_{\infty,i} \left( \frac{m_i}{2\pi k_B T_\infty} \right)^{3/2} \exp \left( -\frac{m_i |v - u_\infty|^2}{2 k_B T_\infty} \right).
\end{equation*}
Without loss of generality, we normalize the global equilibrium such that the gas is at rest and the temperature is unity. Specifically, we set the global momentum $u_\infty = 0$ and the temperature $k_B T_\infty = 1$. Under these physical constraints, the global Maxwellian reduces to
\begin{equation} \label{global maxwellian}
	\mu_i(v) = n_{\infty, i} \left( \frac{m_i}{2\pi} \right)^{3/2} \exp \left( -m_i \frac{|v|^2}{2} \right).
\end{equation}
This centered Maxwellian $\mu_i$ serves as the reference state for our perturbation analysis, satisfying $Q_{ij}(\mu_i, \mu_j) = 0$ for all $i,j \in \{1, \cdots, N \}$. To study the convergence of solutions toward the global equilibrium $\pmb{\mu}$, we introduce the relative entropy (or Kullback-Leibler divergence), defined as: 
\begin{align} \label{relative entropy}
 \mathcal{E}(\pmb{F})=\mathcal{E}(\pmb{F})(t) := E(\pmb{F}) - E(\pmb{\mu}) = \sum_{i=1}^N \int_{\T^3 \times \R^3} F_i \log F_i -\mu_i \log \mu_i dvdx, 
\end{align} 
where $\mu_i$ is defined in \eqref{global maxwellian}.

\subsection{Perturbation regime}
Under the perturbative regime, we define the solution as $F_i = \mu_i + \sqrt{\mu_i} f_i$, where $\mu_i$ is the global Maxwellian \eqref{global maxwellian} for each species $i$. The perturbed multi-species Boltzmann equation can be expressed as
\begin{equation}\label{refor Bol}
	\partial_{t}f_i + v \cdot \nabla_{x}f_i = L_i(\pmb{f}) + \Gamma_i(\pmb{f}), \quad \forall i \in \{1,\cdots, N\}, 
\end{equation}
where $L_i(\pmb{f})$ is the linearized collision operator and $\Gamma_i(\pmb{f})$ is the nonlinear part. The linearized operator $\pmb{L} = (L_1, \dots, L_N)$ is defined by

\begin{equation*}
	L_i(\pmb{f}) = \sum_{j=1}^N L_{ij}(\pmb{f}) = \sum_{j=1}^N \frac{1}{\sqrt{\mu_i}} \left[ Q_{ij}(\mu_i, \sqrt{\mu_j}f_j) + Q_{ij}(\sqrt{\mu_i}f_i, \mu_j) \right] = -\nu_i(v)f_i + K_i(\pmb{f}).
\end{equation*}
Specifically, the linearized collision operator $L_i(\pmb{f})$ is decomposed into collision frequency term $-\nu_i(v)f_i$ and the integral operator $K_i(\pmb{f})$. The collision frequency for each species $i$ is defined as 
\begin{equation} \label{nu_def}
	\nu_i(v) = \sum_{j=1}^N \nu_{ij}(v)= \sum_{j=1}^N \frac{1}{\sqrt{\mu_i}} Q_{ij}^{-}(\sqrt{\mu_i}, \mu_j) = \sum_{j=1}^N \int_{\R^3 \times \S^2} B_{ij}(v-v_*, \sigma) \mu_j(v_*) \, d\sigma dv_* .
\end{equation}
The integral operator $K_i(\pmb{f})$ is defined by
\begin{align} \label{def_K}
K_i(\pmb{f}) := \sum_{j=1}^N 	K_{ij}(\pmb{f})(v),
\end{align}
where for each $i,j \in \{1,\dots,N\}$, the operator $K_{ij}$ is given by 
\begin{align*}
		K_{ij}(\pmb{f})(v)
		&:= \frac{1}{\sqrt{\mu_i(v)}}
		\left[
		Q_{ij}(\mu_i, \sqrt{\mu_j} f_j)(v)
		+ Q_{ij}^{+}(\sqrt{\mu_i} f_i, \mu_j)(v)
		\right]\cr
		&= \int_{\R^3 \times \S^2}B_{ij}(v-v_*, \sigma)\sqrt{\mu_j(v_*)}\left[\sqrt{\mu_i(v')}f_j(v'_*)+\sqrt{\mu_j(v'_*)}f_i(v')\right]d\sigma dv_* \cr
		&\quad - \int_{\R^3 \times \S^2}B_{ij}(v-v_*, \sigma)\sqrt{\mu_j(v_*)}\sqrt{\mu_i(v)}f_j(v_*)d\sigma dv_*.
\end{align*}
The nonlinear collision term is given by $\Gamma_i(\pmb{f}) = \sum_{j=1}^{N} \Gamma(f_i, f_j)$, where 
\begin{equation*} 
	\Gamma(f_i, f_j) = \frac{1}{\sqrt{\mu_i}} Q_{ij} (\sqrt{\mu_i} f_i , \sqrt{\mu_j}f_j).
\end{equation*}
The nonlinear collision term is also split into gain and loss terms $\Gamma_i= \Gamma^+_{i} - \Gamma^-_{i}$:
\begin{align*}
	\Gamma^+_{i}(\pmb{f})= \sum_{j=1}^N\Gamma^+(f_i, f_j) &= \sum_{j=1}^N\int_{\R^3 \times \S^2} B_{ij}(v-v_*, \sigma) \sqrt{\mu_j(v_*)} f_i(v') f_j(v'_*) d\sigma dv_*, \\
	\Gamma^-_{i}(\pmb{f})=\sum_{j=1}^N\Gamma^-(f_i, f_j) &= \sum_{j=1}^N\int_{\R^3 \times \S^2} B_{ij}(v-v_*, \sigma) \sqrt{\mu_j(v_*)} f_i(v) f_j(v_*) d\sigma dv_*.
\end{align*}

\subsection{History and related works}
For the classical Boltzmann equation describing dynamics of colliding monatomic gases, extensive progress has been made in the study of its well-posedness. DiPerna-Lions \cite{Diperna-Lion} established the global existence of renormalized solutions for general initial data possessing finite mass, energy, and entropy. Despite this fundamental result, the issue of uniqueness for such solutions remains open. In addition to existence theories, the investigation of long-time behavior constitutes a central theme in the analysis of the Boltzmann equation, largely motivated by the H-theorem. Since entropy dissipation vanishes at the global Maxwellian, it is natural to expect convergence toward this equilibrium state. Under global-in-time a priori regularity assumptions and a Gaussian lower bound on solutions, Desvillettes and Villani \cite{DV} proved convergence to the global Maxwellian with an almost exponential rate.

Alongside these developments, well-posedness and asymptotic behavior have also been studied within perturbative frameworks. Ukai \cite{Ukai} first constructed global-in-time solutions in a periodic domain $\mathbb{T}^3$, assuming the initial data are sufficiently close to a global Maxwellian. Subsequently, Guo \cite{GuoVPB,GuoVMB} employed energy methods in high Sobolev spaces to obtain classical solutions in the absence of physical boundaries. In bounded domains, however, high-order regularity cannot generally be expected; see \cite{GKTT2017,Kim2011,KLKRM}. To address this difficulty, Guo \cite{Guo10} introduced an $L^2$-$L^\infty$ bootstrap framework, which enables the construction of solutions with lower regularity under suitable boundary conditions. This approach was later refined by Kim and Lee \cite{KimLee}, who replaced the analyticity assumption in \cite{Guo10} with general $C^3$ uniformly convex bounded domains under the specular reflection boundary condition. For further developments concerning boundary value problems, we refer to \cite{CKLVPB,KimLeeNonconvex,KL_holder,KKL2025}. For the system of polyatomic gases, related results can be found in \cite{DL2023}.
	
However, the aforementioned works based on perturbative frameworks rely heavily on the assumption that the initial perturbation is sufficiently small in the $L^\infty$ norm. Relaxing this constraint to allow for a general class of large initial data remains a challenging open problem. Nevertheless, some progress has been made in this direction. For instance, \cite{DHWY2017} replaced the $L^\infty_{x,v}$ smallness assumption with smallness in $L_x^1 L_v^\infty$ and initial relative entropy, thereby allowing for large amplitudes in the $L^\infty_{x,v}$ norm. Then, \cite{DW2019,KLP2022,DKL2023} established large-amplitude results by utilizing the smallness of the initial relative entropy rather than the $L_x^1 L_v^\infty$ norm. For polyatomic gases, related results can be found in \cite{GS2025}. More recently, large-amplitude problems have been studied in the $L^p_v L^\infty_x$ framework rather than in the $L^\infty_{x,v}$ setting; see, for instance, \cite{KK2026,DKL2025}.
	
Early works on kinetic models for gas mixtures focused on constructing thermodynamically consistent descriptions. Desvillettes et al. introduce a kinetic model for gaseous mixtures with chemical reactions, which preserves the appropriate conservation laws and satisfies fundamental thermodynamic principles \cite{DMS2005}. The diffusive behavior of kinetic mixture models was later investigated by Boudin, Grec, Pavić, and Salvarani, who studied the diffusion asymptotics of a kinetic model for gaseous mixtures \cite{BGPS2013}. In a subsequent work, Boudin, Grec, and Salvarani rigorously derived the Maxwell–Stefan diffusion limit from a kinetic model of gas mixtures \cite{BGS2015}. More recently, Daus et al. analyzed hypocoercivity and exponential convergence to equilibrium for a linearized multi-species Boltzmann system \cite{DJMZ2016}, while related spectral gap stability results were established in \cite{BBBG2020}. Our study is partly motivated by the work of Briant and Daus \cite{BD2016}. In parallel, inelastic kinetic models for gas mixtures have also been studied; see, for instance, \cite{AL2023,RT2025}. Recently, classical solutions for multi-species Boltzmann equation with unequal molecular masses were established in the soft potential regime \cite{WWW2026}. However, the existing results for the multi-species Boltzmann equation are restricted to the small-amplitude setting. To the best of our knowledge, the global well-posedness of the multi-species Boltzmann equation for large-amplitude initial data has not been fully understood. Finally, we refer to  \cite{KLY2021_1,BKLY_BGK,BKPY2021,KLY2021_2,BKPY2023,BKLY2024,HLY2024} for various results on relaxation models of the Boltzmann equation in the context of gas mixtures.

\subsection{Strategy and Novelty}

As discussed in the previous subsection, significant progress has been made for the multi-species Boltzmann equation in the small-amplitude regime, covering both hard potentials \cite{BD2016} and, more recently, soft potentials \cite{WWW2026}. However, the question of global well-posedness and convergence to equilibrium for large-amplitude initial data remains entirely open, particularly for gas mixtures. The existing theories, including \cite{BD2016}, fundamentally rely on the assumption that the initial perturbation is sufficiently small in the $L^\infty_{x,v}$ norm to control the nonlinear collision operator $\Gamma$. In the large-amplitude setting, this strategy breaks down as the $L^\infty_{x,v}$ smallness assumption is no longer valid, rendering the direct extension of standard perturbative arguments impossible.

To overcome this, we adopt a strategy inspired by recent developments in the single-species Boltzmann equation \cite{KLP2022} and polyatomic models \cite{GS2025}, which successfully handle large initial perturbations by exploiting the smallness of relative entropy rather than the $L^\infty_{x,v}$ norm. However, extending this framework to gas mixtures with distinct masses is highly non-trivial and presents unique analytical challenges not present in the single-species case.

To implement this strategy, a central problem is to control the nonlinear collision term—specifically the gain term $\Gamma^+$ without relying on a priori smallness in $L^{\infty}_{x,v}$. This necessitates a more refined point-wise estimate of the nonlinear collision operator. We emphasize that the multi-species setting with distinct masses ($m_i \neq m_j$) further complicates this problem. The mass disparity breaks the intrinsic symmetry properties of the collision operator found in single-species models. Consequently, the standard Carleman representation cannot be directly applied; it fails for mixtures with distinct masses because the post-collisional variables exhibit a complex dependence that prevents the decoupling of parallel and orthogonal components.

In the remainder of this subsection, we outline the key novelties and technical tools developed to address these structural challenges, proceeding in the following three steps\\

\textbf{Step 1: Estimate of the nonlinear gain term $\Gamma^+(f, f)$ via Carleman representation.} The classical pointwise estimate for the nonlinear collision operator,
\begin{align*}
	|w\Gamma(f, f)| \leq C \nu(v) \|wf\|_{L^\infty}^2,
\end{align*}
is no longer applicable in large-amplitude regime, because it relies crucially on the smallness of the weighted $L^\infty_{x,v}$ norm. Consequently, we require an alternative estimate to control the nonlinear term $\Gamma$ without assuming any smallness of $\Vert wf \Vert_{L^\infty}$. \\
\indent To address this, we decompose the nonlinear term into gain and loss parts, $\Gamma=\Gamma^+-\Gamma^-$, and treat them separately. The estimate for the loss term $\Gamma^-$ is deferred to Step 2, and we focus here on the gain term $\Gamma^+$. In the context of single-species gases, the following refined point-wise estimate is typically employed:
\begin{align*}
	|w(v) \Gamma^+(f)(t,x,v)| \leq C \frac{ \Vert wf(t) \Vert_{L^\infty}}{1+|v|} \left( \int_{\R^3} (1+ |\eta|)^{-\ell }|wf(\eta)|^p d\eta \right)^{1/p},
\end{align*} 
for some $p>1$, where $w(v) = \langle v \rangle^q$. This shows that the nonlinear gain term can be controlled even for large-amplitude data, provided that the initial relative entropy is sufficiently small.\\
\indent We note that this estimate can be derived through the Carleman representation. For the single-species case, the Carleman representation enjoys strong geometric simplifications: the post-collision velocities $v'$ and $v_*'$ can be decomposed cleanly into components parallel $z_{\parallel}:= (z\cdot \omega) \omega$  and orthogonal $z_{\perp}:= z-z_{\parallel}$ to the relative velocity $z:=v_*-v$. This orthogonal–parallel decomposition is essential for revealing the symmetry structure of the collision gain operator, thereby enabling the derivation of the nonlinear estimates stated above. However, in the multi-species setting with distinct particle masses $m_i \neq m_j$, this structure breaks down. Specifically, while the decomposition for $v'$ remains manageable, in the sense that 
\begin{align*}
	v' = v'(v,z_{\parallel})= v+ \frac{2m_j}{m_i+m_j} z_{\parallel}, 
\end{align*}
the post-collision velocity $v_*'$ possesses a more complicated structure:
\begin{align*}
	v_*' =v_*'(v,z_{\parallel},z_{\perp})= v + \left(1 - \frac{2m_i}{m_i+m_j} \right)z_{\parallel} + z_{\perp}. 
\end{align*}
The presence of the parallel component in $v_*'$, which is absent in the single-species case, reflects the mass asymmetry of the collision and prevents a clean orthogonal–parallel decoupling (see the proof of Lemma \ref{gain est_new}).
This prevents a direct extension of the single-species Carleman analysis and creates new difficulties in controlling the gain term.\\
\indent To resolve the geometric difficulty caused by the mass asymmetry, we use the Carleman-type representation introduced in \cite{BD2016} (Lemma \ref{carleman 2}), which expresses the gain term through an integration over a sphere $\tilde{E}^{ij}_{vv'_*}$ of radius $R_{vv'_*}$ centered at $O_{vv'_*}$. This representation, however, only performs the geometric reduction of the collision integral, and the velocity behavior of the resulting exponential factors still has to be analyzed. A direct expansion of the exponential factors produced by the change of variables yields the product
	\begin{align*}
		e^{-\frac{m_j}{4}|v'_*|^2}\, e^{\frac{m_i}{4}|v|^2}\, e^{-\frac{m_i}{4}R^2_{vv'_*}}\, e^{-\frac{m_i}{4}|O_{vv'_*}|^2}\, e^{\frac{m_i}{2}R_{vv'_*}|O_{vv'_*}|},
	\end{align*}
	in which both growing and decaying contributions in the velocity variable are present, with the cross-term $e^{\frac{m_i}{2}R_{vv'_*}|O_{vv'_*}|}$ in particular preventing a direct sign analysis. It is not immediately clear whether these exponential factors retain sufficient decay in the velocity variable to derive the desired pointwise gain estimate. \\
	\indent The first contribution of this work is to reveal an exact algebraic cancellation among these factors. We prove that, after substituting the explicit forms of $R_{vv'_*}$ and $O_{vv'_*}$ from \eqref{RO}, the five exponential factors above combine into the single negative semi-definite quadratic form:
	\begin{align*}
		-\frac{m_j}{4|m_i-m_j|^2}\bigl(|m_i v - m_j v'_*| - m_i|v-v'_*|\bigr)^2.
	\end{align*}
	This cancellation, which to our knowledge has not previously been observed in the literature on multi-species kinetic equations, yields the polynomial decay in the velocity variable required for the pointwise gain estimate of Lemma \ref{gain est_new}.\\

\textbf{Step 2: Nonlinear solution operator and exponential weight integral.} 
To make use of the gain estimate for $\Gamma^+$ derived in Step 1, we reformulate the multi-species Boltzmann equation:
\begin{equation*}
\partial_t f_i  + v \cdot \nabla_x f_i  + \mathcal{R}_i(\pmb{f}) f_i = K_i(\pmb{f}) +\Gamma^+(\pmb{f}), \quad \forall i \in \{1,\cdots, N\}, 
\end{equation*}
where $\mathcal{R}_i(\pmb{f})$ is defined as the sum of the linear collision frequency $\nu_i(v)f_i $ and the nonlinear loss term $\Gamma^-(\pmb{f})$. Accordingly, the nonlinear solution operator related with the damping term $\mathcal{R}_i(\pmb{f})$ is represented by 
\begin{equation*}
\mathcal{A}_i(t,x,v)
:= \exp\!\left(
-\int_0^t \mathcal{R}_i(\pmb f)\bigl(s,\,x-v(t-s),\,v\bigr)\,ds
\right).
\end{equation*}
Under the hard potential assumption, the linear collision frequency $\nu_i(v)$ admits a positive lower bound (Lemma \ref{CFE}). The damping rate generated by $\mathcal{R}_i(\pmb{f})$, however, depends on $\pmb{f}$ itself through the loss term, and it is not clear whether a uniform lower bound holds when $\|w\pmb{f}\|_{L^\infty}$ is large. In the single-species case \cite{DW2019,DKL2023}, the corresponding lower bound is obtained by controlling a single Gaussian-weighted velocity integral through the smallness of the initial relative entropy. In the multi-species setting, the cross-collision terms involve $\mu_j(v_*)$ for every $j \neq i$, and the analysis requires controlling $N$ distinct Gaussian-weighted velocity integrals simultaneously, each associated with a different species. Furthermore, the loss term $\Gamma_i^-(\pmb{f})$ has a polynomial growth of order $(1+|v|)^\gamma$, and the analysis requires showing that this contribution does not dominate the damping induced by $\nu_i(v)$ for each species. \\
\indent To establish such a positive lower bound for $\mathcal{R}_i(\pmb f)$, we therefore
perform a weighted velocity integral estimate.
More precisely, we control the exponentially weighted quantity
\[
\int_{\mathbb R^3} e^{-\frac{|v|^2}{4}} |h_i(t,x,v)|\,dv \leq C,
\quad \forall i \in \{1, \cdots, N\} ,
\]
which is essential since the estimate for the nonlinear loss term
$\Gamma_i^-(\pmb f)$ produces a polynomial velocity growth of order
$(1+|v|)^\gamma$.
The Gaussian weight $e^{-\frac{|v|^2}{4}}$ provides sufficient decay to absorb this
growth, while the smallness of the initial relative entropy guarantees that the
weighted integral remains uniformly bound after a finite time $\tilde{t}$.
As a consequence, the contribution of the nonlinear loss term can be controlled and
made sufficiently small, which allows us to recover the lower bound
\begin{align*}
	\mathcal{R}_i(\pmb f)(t,x,v) \ge \frac12 \nu_i(v) \gtrsim 1, \quad \forall (t,x,v) \in [\tilde{t},\infty) \times \T^3 \times \R^3. 
\end{align*}
\indent Thanks to this positive lower bound, the mild formulation associated with the
nonlinear solution operator exhibits an effective damping structure. \\ 

\textbf{Step 3: $L^\infty(\langle v \rangle^q)$ estimate via Gr\"{o}nwall inequality.} Using the positive lower bound for the nonlinear solution operator established in Step 2 together with the refined nonlinear gain estimate derived in Step 1, we
rewrite the equation in a new mild form:
\begin{equation*}
f_i(t)
\sim  e^{-\nu_i t} f_{0,i}
+ \int_0^t e^{-\nu_i(t-s)} \bigl[ K_i (\pmb{f})(s) + \Gamma^+_{i}(\pmb{f})(s) \bigr]\,ds.
\end{equation*}
The core of the argument is an $L^2$--$L^\infty$ bootstrap scheme in the spirit of Guo’s approach in \cite{Guo10}. Although the $L^\infty_{x,v}$ norm of $wf_0$ may be large, we assume that the initial relative entropy. In particular, our paper provides the control of the $L^1_{x,v}$ and $L^2_{x,v}$ norms by the initial relative entropy in Lemma \ref{RE}. A double Duhamel iteration yields the $L^1_{x,v}$ or $L^2_{x,v}$ type norms, which allows the contributions from both linear and nonlinear terms to be controlled through the initial relative entropy. 

More precisely, combining the nonlinear gain estimate from Step~1 with the new mild
formulation, we derive a weighted $L^\infty$ estimate in which the $L^\infty_{x,v}$ norm of
the solution is bounded by a functional involving $L^1_{x,v}$ or $L^2_{x,v}$ type norms generated through the double Duhamel iteration and controlled by the small initial relative entropy. This leads to the following Gr\"{o}nwall inequality 
\begin{align*}
\|w\pmb{f}(t)\|_{L^\infty} \lesssim e^{-\lambda t} \| w\pmb{f}_{0} \|_{L^\infty}\left(1+\int_{0}^{t}\|w\pmb{f}(s)\|_{L^\infty}ds\right)+\mathcal{P}\left(\sup_{0\leq s \leq t} \|w\pmb{f}(s)\|_{L^\infty}\right) \mathcal{E}(\pmb{F}_0).
\end{align*}
where $\mathcal{P}$ is a polynomial function. As a consequence, there exists a finite time $T_0>0$ such that the solution enters
the small-amplitude regime for all $t\ge T_0$. Thereafter, the small-amplitude theory applies, yielding global
well-posedness and convergence to equilibrium.

\subsection{Organization of this paper}
In Section \ref{Sec2}, we state global well-posedness results for the multi-species Boltzmann equation, including both small-amplitude and large-amplitude initial data. In particular, we present a new lower bound estimate for a nonlinear collision frequency, which plays a crucial role in the analysis of large-amplitude solutions. In Section \ref{Sec3}, we introduce the preliminary tools used throughout this paper: entropy estimates for the multi-species Boltzmann equation, estimates for the linearized operator developed in \cite{BD2016}, and pointwise estimates for the gain part of the nonlinear collision operator. Section \ref{sec:small_regime} is devoted to proving the main theorem in the small-amplitude regime. In Section \ref{sec:LAPP}, we first establish the lower bound for the nonlinear collision frequency stated in Section \ref{Sec2}, and then derive a weighted $L^\infty$ Gr\"onwall-type estimate. Finally, in Section \ref{sec:large_regime}, we combine the results of Sections \ref{sec:small_regime} and \ref{sec:LAPP} to complete the proof of the main theorem for large-amplitude perturbations.
\section{Main results} \label{Sec2}
We begin with the perturbative regime where the initial data are sufficiently small in a weighted $L^\infty_{x,v}$ norm. 
In this setting, the solution remains globally small and exhibits exponential decay. The proof of this theorem is given in Section~\ref{sec:small_regime}.
\begin{theorem} \label{thm:main2}Let $w(v)$ be the velocity weight function defined in \eqref{v_weight}. Let the initial data be given by 
\begin{align*}
F_{0,i}(x,v) = \mu_i(v) + \sqrt{\mu_i(v)} f_{0,i}(x,v) \ge 0,
\end{align*}
and assume that the conservation laws of mass, momentum, and energy are satisfied, that is,
\begin{align} \label{conserv}
\begin{cases} 
    \displaystyle \int_{\mathbb{T}^3 \times \mathbb{R}^3} \sqrt{\mu_i(v)}f_{0,i} (x,v) \, dvdx =0, & \forall i \in \{1,\cdots,N\}, \\[10pt]
    \displaystyle \sum_{i=1}^N \int_{\mathbb{T}^3 \times \mathbb{R}^3} m_i v \sqrt{\mu_i(v)} f_{0,i}(x,v) \, dvdx =0, \\[10pt]
    \displaystyle \sum_{i=1}^N \int_{\mathbb{T}^3 \times \mathbb{R}^3} m_i |v|^2 \sqrt{\mu_i(v)} f_{0,i}(x,v) \, dvdx=0,
\end{cases}
\end{align}
There exists a sufficiently small constant $\kappa>0$ such that if 
		\begin{align*}
			\Vert w\pmb{f}_0 \Vert_{L^\infty} \leq \kappa,
		\end{align*} 
		then the multi-species Boltzmann equation admits a unique global solution 
		\begin{align*}
			F_i(t,x,v)=\mu_i(v) + \sqrt{\mu_i(v)} f_i(t,x,v) \geq 0, \quad \forall i \in \{1,\cdots, N\},
		\end{align*}
		satisfying  
		\begin{align*}
			\Vert w\pmb{f}(t) \Vert_{L^\infty} \leq C e^{-\lambda_1 t} \Vert w\pmb{f}_0 \Vert_{L^\infty}, \quad \forall t \geq 0, 
		\end{align*}
		for some constants $C>0$ and $\lambda_1 >0$.
\end{theorem}
	We establish a key estimate for the solution operator $\mathcal{R}_i(\pmb{f})$ in the large-amplitude regime. This estimate serves as a fundamental ingredient for applying the nonlinear estimate in the proof of Theorem \ref{thm:main}. The proof of this proposition is given in Section~\ref{sec:LAPP}.
	\begin{proposition} \label{Rf est}
	Let $w(v)$ denote the velocity weight function defined in \eqref{v_weight}. Let $f_i(t,x,v)$ be a solution to \eqref{refor Bol} with initial data 
${f}_{0,i}(x,v)$ satisfying \eqref{conserv}, and define $h_i(t,x,v)=w(v) f_i(t,x,v)$ for $1\leq i \leq N$. Assume that 
	\begin{align*}
		\sup_{0\leq s\leq t}\|h_i(s)\|_{L^\infty}\leq \bar{M}<\infty, \quad\forall i\in \{1,\cdots, N\}.
	\end{align*}
	Then, there exists a  small positive constant $\varepsilon_1>0$ such that if the initial relative entropy satisfies $\mathcal{E}(\pmb F_0)\leq \varepsilon_1$, then
	\begin{align*}
		\mathcal{R}_i(\pmb{f})(t,x,v)\geq \frac{1}{2} \nu_i(v)>0, \quad \forall (t, x, v) \in [\tilde{t}, T_0] \times \mathbb{T}^3 \times \mathbb{R}^3,
	\end{align*} 
	where 
		\begin{align*}
			\mathcal{R}_i(\pmb{f})(t,x,v):= \sum_{j=1}^N \int_{\R^3 \times \S^2} B_{ij}(v-v_*,\sigma) [\mu_j(v_*)+\sqrt{\mu_j(v_*)} f_j(v_*)] d\sigma dv_*
		\end{align*}
and $\nu_i(v)$ is the collision frequency defined in \eqref{nu_def}. Here, $\tilde t>0$ is a constant to be determined in the proof.
\end{proposition}
	Finally, We now state our main result in the large-amplitude regime, where the initial perturbation may be arbitrarily large in the weighted $L^\infty_{x,v}$ norm, provided that the initial relative entropy is sufficiently small. The proof is given in Section~\ref{sec:large_regime}.		
	\begin{theorem}\label{thm:main}
	Let $w(v)$ be as defined in \eqref{v_weight}. 
The initial data are given by 
\[
F_{0,i}(x,v) = \mu_i(v) + \sqrt{\mu_i(v)} f_{0,i}(x,v) \ge 0.
\]
Assume that the initial perturbation ${f}_{0,i}$ satisfies \eqref{conserv} and
\[
\| w \pmb{f}_0 \|_{L^\infty} \le M_0
\]
for some fixed (possibly large) constant $M_0>0$. Then there exists a sufficiently small constant $\varepsilon_2 >0$ such that, if
\[
\mathcal{E}(\pmb{F}_0) \le \varepsilon_2,
\]
the multi-species Boltzmann equation~\eqref{mixed bol}
admits a unique global-in-time mild solution
\begin{align*}
F_i(t,x,v)=\mu_i(v) + \sqrt{\mu_i(v)} f_i(t,x,v) \geq 0, 
\quad \forall i \in \{1,\cdots, N\},
\end{align*}
satisfying 
\[
\| w \pmb{f}(t) \|_{L^\infty}
\le
 C M_0^2 \exp \left\{ \frac{C M_0^2}{\lambda_1}\right\} e^{-\vartheta t},
\qquad
\forall t \ge 0,
\]
for some constants $C>0$ and $\vartheta>0$, where $\lambda_1$ is defined in Theorem \ref{thm:main2}.
\end{theorem}

	\section{Preliminary Estimates}\label{Sec3}
	In this section, we establish the preliminary estimates will be used throughout this paper.
	\subsection{Entropy Theory in the Multi-species Boltzmann Equation}
This subsection is devoted to providing a rigorous analysis of the entropy evolution within the multi-species framework. Based on the multi-species Boltzmann equation \eqref{mixed bol}, the time derivative of the entropy of each individual species $i$ is formulated is given by 
\begin{align*}
	&\frac{d}{dt} \int_{\mathbb{T}^3 \times \mathbb{R}^3} F_i \log F_i \, dvdx \\
	&= \sum_{j=1}^N \int_{\mathbb{T}^3 \times \mathbb{R}^3} \int_{\R^3 \times \S^2} B_{ij}(v-v_*,\sigma) [F_i(v') F_j(v'_*) - F_i(v) F_j(v_*)] \log F_i(v) \, d\sigma dv_* dvdx.
\end{align*}
Recalling the entropy $E(\pmb{F})$ introduced in \eqref{entropy}, we sum the above expression over all species $i=1, \dots, N$. By using the symmetry of the collision kernel $B_{ij}$ and the weak form of the collision operator, we derive the H-theorem for the mixture:
\begin{align}
	&\frac{d}{dt} E(\pmb{F}) \nonumber \\
	&= \frac{1}{4} \sum_{i,j=1}^N \int_{\mathbb{T}^3 \times \mathbb{R}^3} \int_{\R^3 \times \S^2}B_{ij}(v-v_*,\sigma) [F_i(v') F_j(v'_*) - F_i(v) F_j(v_*)] \log \frac{F_i(v) F_j(v_*)}{F_i(v') F_j(v'_*)} \, d\sigma dv_* dvdx \le 0. \label{H-theorem}
\end{align}
The inequality \eqref{H-theorem} demonstrates the fundamental dissipative nature of the multi-species system, where the equality holds if and only if the distribution $\pmb{F}$ attains the local Maxwellian state. Furthermore, as $t \to \infty$, the solution is expected to relax towards the global Maxwellian $\pmb{\mu}$, which serves as the unique equilibrium state determined by the conserved quantities.\\
	\indent A fundamental property of the relative entropy \eqref{relative entropy} is its monotonicity; since $\log \mu_i$ is a collision invariant and mass is conserved, the H-theorem implies that $\mathcal{E}(\pmb{F})(t) \leq \mathcal{E}(\pmb{F}_0)$ for all $t \ge 0$. The following lemma provides crucial functional estimates that allow us to control the perturbation in both $L^1_{x,v}$ and $L^2_{x,v}$ norms by the initial relative entropy. 
	\begin{lemma}\label{RE}
		Assume that $F_i(t,x,v)$ solves the multi-species Boltzmann equation \eqref{mixed bol} for each $i \in \{1, \cdots, N\}$. For any $t \geq0$, we have 
		\begin{align*}
			\sum_{i=1}^N \left(\int_{\T^3 \times \R^3} \frac{1}{4\mu_i} |F_i - \mu_i|^2 \chi_{\{|F_i-\mu_i|\leq \mu_i\}} dvdx +\int_{\T^3 \times \R^3} \frac{1}{4} |F_i - \mu_i| \chi_{\{|F_i-\mu_i|\geq \mu_i\}}dvdx \right) \leq \mathcal{E}(\pmb{F}_0). 
		\end{align*}
		Furthermore, one obtains the estimate in terms of $f_i:=(F_i-\mu_i) / \sqrt{\mu_i}$:
		\begin{align*}
			\frac{1}{4}\sum_{i=1}^N \left(\int_{\T^3 \times \R^3} |f_i(t,x,v)|^2 \chi_{\{|f_i|\leq \sqrt{\mu_i}\}} dvdx +\int_{\T^3 \times \R^3}\sqrt{\mu_i(v)}|f_i(t,x,v)| \chi_{\{|f_i|\geq \sqrt{\mu_i}\}}dvdx \right) \leq \mathcal{E}(\pmb{F}_0). 
		\end{align*}
	\end{lemma}
	\begin{proof}
		It follows from Taylor's expansion that 
		\begin{align*}
			F_i \log F_i -\mu_i \log \mu_i = (1+ \log \mu_i) (F_i- \mu_i) +\frac{1}{2\tilde{F_i}} |F_i - \mu_i|^2, 
		\end{align*}
		where $\tilde{F_i}$ is between $F_i$ and $\mu_i$. Then, we compute 
		\begin{align*}
			\frac{1}{2\tilde{F_i}} |F_i - \mu_i|^2 = 	F_i \log F_i -\mu_i \log \mu_i -  (1+ \log \mu_i) (F_i- \mu_i)  =: \psi\left(\frac{F_i}{\mu_i}\right) \mu_i, 
		\end{align*}
		where $\psi(x):=x\log x - x +1$. Thus, 
		\begin{align*}
			\int_{\T^3 \times \R^3}  \frac{1}{2\tilde{F_i}} |F_i - \mu_i |^2 dvdx = \int_{\T^3 \times \R^3} \psi \left(\frac{F_i}{\mu_i}\right) \mu_i dvdx. 
		\end{align*}
		For the left-hand side, we split the integration domain as 
		\begin{align*}
			1= 1 \chi_{\{|F_i-\mu_i| \leq \mu_i\}} + 1 \chi_{\{|F_i-\mu_i|>\mu_i\}}.
		\end{align*}
		Since $F_i$ is non-negative, the condition $|F_i - \mu_i| > \mu_i$ implies $F_i > 2\mu_i$. Thus, on  $\{|F_i - \mu_i| > \mu_i\}$,  we have 
		\begin{align*}
			\frac{|F_i -\mu_i|}{\tilde{F_i}} = \frac{F_i-\mu_i}{\tilde{F_i}} > \frac{F- \frac{1}{2}F_i}{F_i} = 1/2.
		\end{align*}
		 On the other hand, over $\{|F_i - \mu_i| \leq \mu_i\}$, we have $0\leq F_i \leq 2\mu_i$. This implies that 
		\begin{align*}
			\frac{1}{\tilde{F_i}} \geq \frac{1}{2\mu_i}.
		\end{align*}
		Thus, we get
		\begin{align*}
			\int_{\T^3 \times \R^3} \frac{1}{4\mu_i} |F_i - \mu_i|^2 \chi_{\{|F_i-\mu_i|\leq \mu_i\}} dvdx +\int_{\T^3 \times \R^3} \frac{1}{4} |F_i - \mu_i| \chi_{\{|F_i-\mu_i|\> \mu_i\}}dvdx  \leq \int_{\T^3 \times \R^3} \psi\left(\frac{F_i}{\mu_i}\right) \mu_i dvdx. 
		\end{align*}
		Due to $\psi'(x) =\log x$, we can deduce from the multi-species Boltzmann equation \eqref{mixed bol} that 
		\begin{align*}
			\p_t \left[\sum_{i=1}^N \psi(\frac{F_i}{\mu_i}) \mu_i\right] +\nabla_x \cdot \left[\sum_{i=1}^N\psi(\frac{F_i}{\mu_i}) \mu_i v \right] =\sum_{i,j=1}^N Q_{ij}(F_i,F_j) \log \frac{F_i}{\mu_i}.
		\end{align*}
		Integrating  over $(x,v) \in \T^3 \times \R^3$, it follows from the collision invariance that 
		\begin{align*}
			\frac{d}{dt} \left[\sum_{i=1}^N \int_{\T^3 \times \R^3} \psi(\frac{F_i}{\mu_i}) \mu_i dvdx \right] = \sum_{i,j=1}^N\int_{\T^3 \times \R^3} Q_{ij}(F_i,F_j) \log F_i dvdx \leq 0.
		\end{align*}
		Hence, we conclude that 
		\begin{align*}
			\sum_{i=1}^N \left(\int_{\T^3 \times \R^3} \frac{1}{4\mu_i} |F_i - \mu_i|^2 \chi_{\{|F_i-\mu_i|\leq \mu_i\}} dvdx +\int_{\T^3 \times \R^3} \frac{1}{4} |F_i - \mu_i| \chi_{\{|F_i-\mu_i|\> \mu_i\}}dvdx \right) \leq \mathcal{E}(\pmb{F}_0). 
		\end{align*}
		Moreover, by substituting the perturbation $F_i(t,x,v)= \mu_i(v) + \sqrt{\mu_i(v)} f_i(t,x,v)$, we can derive  
		\begin{align*}
		\frac{1}{4}\sum_{i=1}^N \left(\int_{\T^3 \times \R^3}  |f_i(t,x,v)|^2 \chi_{\{|f_i|\leq \sqrt{\mu_i}\}} dvdx +\int_{\T^3 \times \R^3}  \sqrt{\mu_i(v)}|f_i(t,x,v)| \chi_{\{|f_i|\geq \sqrt{\mu_i}\}}dvdx \right) \leq \mathcal{E}(\pmb{F}_0). 
		\end{align*}
	\end{proof}
	\subsection{Estimates for the Linear Operator}
	In this subsection, we analyze the properties of the linearized collision operator $\pmb{L} = (L_1, \dots, L_N)$. We recall that the linearized operator can be decomposed into the collision frequency $\nu_i(v)$ and an integral operator $K_i$ defined in \eqref{nu_def} and \eqref{def_K}, respectively
	\begin{equation*}
		L_i(\pmb{f}) = -\nu_i(v)f_i + K_i(\pmb{f}).
	\end{equation*}
	 To establish the global well-posedness, it is crucial to understand bounds for $\nu_i$ and the integrability of the operator $K_i$. The following lemma provides estimates for $\nu_i$, showing that they are comparable across different species.
	\begin{lemma}\label{CFE}
		\cite{BD2016}
		For all $i \in \{1, \dots, N\}$, there exist constants $\nu_i^{(0)}, \nu_i^{(1)} > 0$ such that
		\begin{align*}
			\nu_0 \le \nu_i^{(0)} (1 + |v|)^\gamma \le \nu_i(v) \le \nu_i^{(1)} (1 + |v|)^\gamma, \quad \forall v \in \mathbb{R}^3,
		\end{align*}
		where $\nu_0 = \min\{\nu_1^{(0)}, \dots, \nu_N^{(0)}\}$.
		As a consequence, the collision frequencies satisfy
		\begin{align*}
			\nu_i(v) \le \beta \nu_{ii}(v), \quad \forall v \in \mathbb{R}^3,
		\end{align*}
		for some constant $\beta > 0$.
	\end{lemma}
	To further understand the behavior of the linearized operator, we investigate the pointwise decay properties of the kernels $k_{ij}$ associated with the operator $K_i$. The following lemma shows that these kernels exhibit fast decay under suitable exponential and polynomial weights, which are essential for obtaining estimates in weighted functional spaces.
	\begin{lemma} \cite{BD2016} \label{K est}
		For all $i \in \{1,\cdots, N\}$, there exists $\pmb{k}_i=(k_{ij})_{1\leq j \leq N}$ such that 
		\begin{align*}
			K_i(\pmb{f})(v) =\int_{\R^3} (\pmb{k}_i \cdot \pmb{f} ) dv_* =  \sum_{j=1}^N \int_{\R^3} k_{ij} (v,v_*) f_{j}(v_*) dv_*.  
		\end{align*}
		Let $\beta>0$ and $\theta \in [0,\frac{1}{4})$. Then, there exist $C_{\theta,\beta}>0$ and $\varepsilon_{\theta,\beta}>0$ such that for all $i,j \in \{1,\cdots , N\}$ and $\varepsilon \in [0,\varepsilon_{\theta,\beta})$,
		\begin{align*}
			\int_{\R^3} |k_{ij} (v,v_*)| e^{\varepsilon m |v-v_*|^2 +\varepsilon m \frac{||v|^2-|v_*|^2|^2}{|v-v_*|^2}}\frac{(1+|v|)^{\beta}e^{\theta|v|^2}}{(1+|v_*|)^{\beta}e^{\theta|v_*|^2}}dv_*\leq \frac{C_{\beta,\theta}}{1+|v|}. 
		\end{align*} 
	\end{lemma}
	
	\subsection{Pointwise Gain Estimates via Carleman Representation}
To obtain global well-posedness for large-amplitude data, it is imperative to establish precise pointwise estimates for the nonlinear collision operator $(\Gamma^+_{i}(\pmb{f}))_{1\leq i \leq N}$. To this end, in this subsection, we derive a weighted pointwise estimate for
\begin{align*}
	\sum_{j=1}^{N}|w(v)\Gamma^+(f_i,f_j)(v)|.
\end{align*}
By using the energy conservation
	\begin{align*}
		\frac{m_i}{2}|v|^2+\frac{m_j}{2}|v_*|^2=\frac{m_i}{2}|v'|^2+\frac{m_j}{2}|v'_*|^2,
	\end{align*}
	we get 
	\begin{align*}
		\frac{m_i}{2}|v|^2\leq\frac{m_i}{2}|v'|^2+\frac{m_j}{2}|v'_*|^2.
	\end{align*}
	This implies that 
	\begin{align*}
		\frac{m_i}{4} |v|^2 \leq \frac{m_i}{2}|v'|^2 \quad \text{or} \quad \frac{m_i}{4} |v|^2 \leq \frac{m_j}{2} |v_*'|^2. 
	\end{align*}
	Therefore, it follows that
	\begin{align*}
		w(v) \leq 2^{q} w(v') \quad \textrm{or} \quad  w(v) \leq C w(v'_*).
	\end{align*}
	Then we have 
	\begin{align} \label{gain_split}
		|w(v)\Gamma^+(f_i,f_j)| &\leq w(v) \int_{\R^3 \times \S^2}B(v-v_*,\sigma)\sqrt{\mu_j(v_*)}|f_i(v')f_j(v'_*)|d\sigma dv_*\cr
		&\leq C_q  \int_{\R^3 \times \S^2}B_{ij}(v-v_*,\sigma)\sqrt{\mu_j(v_*)}|w(v')f_i(v')f_j(v'_*)|d\sigma dv_*\cr
		&\quad + C_q  \int_{\R^3 \times \S^2}B_{ij}(v-v_*,\sigma)\sqrt{\mu_j(v_*)}|f_i(v')w(v'_*)f_j(v'_*)|d\sigma dv_*\cr
		& \leq C_q\|wf_i(t)\|_{L^\infty}  \int_{\R^3 \times \S^2}B_{ij}(v-v_*,\sigma)\frac{\sqrt{\mu_i(v')}\sqrt{\mu_j(v'_*)}}{\sqrt{\mu_i(v)}}|f_j(v'_*)|d\sigma dv_*\cr
		&\quad + C_q\|wf_j(t)\|_{L^\infty}  \int_{\R^3 \times \S^2}B_{ij}(v-v_*,\sigma)\sqrt{\mu_j(v_*)}|f_i(v')|d\sigma dv_*\cr
		&:=I_1+I_2.
	\end{align}
	
	The two terms $I_1$ and $I_2$ require different treatments due to the mass asymmetry $m_i \neq m_j$. The post-collisional velocity $v'$ depends only on the component of the relative velocity parallel to the collisional direction, regardless of the masses. Hence the term $I_2$ involving $f_i(v')$ admits the orthogonal--parallel decoupling and can be estimated as in the single-species case. In contrast, the post-collisional velocity $v'_*$ depends on both the parallel and orthogonal components, so the term $I_1$ involving $f_j(v'_*)$ does not admit a clean orthogonal--parallel decoupling. This constitutes the genuine multi-species difficulty.
	
	To overcome this difficulty in $I_1$, we recall the Carleman representation for gas mixtures established in \cite{BD2016}, which transforms the collision integral into a form amenable to pointwise analysis by integrating over geometric structures (spheres or planes) in the velocity space.
\begin{lemma}\cite{BD2016}(Carleman representation for the multi-species Boltzmann equation) \label{carleman 2}
	For $i,j \in \{1,\dots,N\}$, the following representation holds:
	\begin{align*}
		&\int_{\R^3 \times \S^2} B_{ij}(v-v_*,\sigma) f(v') g(v'_*) \, d\sigma dv_* \\ 
		&= C\int_{\mathbb{R}^3} \frac{1}{| v-v'_*|} \left( \int_{\tilde{E}^{ij}_{vv'_*}} \frac{B_{ij}\left(v-V(v',v'_*), \frac{v'_*-v'}{|v'_*-v'|}\right)}{|v'_*-v'|} f(v') \, dE(v') \right) g(v'_*) \, dv'_*,        
	\end{align*}
	where 
	\begin{align*}
	V(v',v'_*) := v_*' + \frac{m_i} {m_j} (v' - v). 
	\end{align*}
	The integration domain $\tilde{E}^{ij}_{vv'_*}$ coincides with the hyperplane $E^{ij}_{vv'_*}$ when $m_i = m_j$. In the case $m_i \neq m_j$, it is the sphere centered at $O_{vv'_*}$ with radius $R_{vv'_*}$, defined by
	\begin{align}\label{RO}
		R_{vv'_*} := \frac{m_j}{|m_i - m_j|} |v - v'_*|, \quad O_{vv'_*} := \frac{m_i}{m_i - m_j} v - \frac{m_j}{m_i - m_j} v'_*.
	\end{align}
	Here, $dE$ is the Lebesgue measure on the set $\tilde{E}^{ij}_{vv'_*}$, and $C$ depends only on the masses $m_i$ and $m_j$.
\end{lemma}
Before deriving pointwise bounds based on the Carleman representation, we provide two technical estimates that will be used later.
\begin{lemma} \label{exp_cal}
		For fixed $x \in \R^3$ and $k>0$, 
		\begin{align*}
			 \int_{\S^2} e^{-k \sigma \cdot x} d\sigma =4\pi \frac{\sinh(k |x|)}{k |x|}.
		\end{align*}
	\end{lemma}
	\begin{proof}
		Since the integral is rotationally invariant, we assume that the fixed vector $x \in \R^3$ is in the $z$-axis. By using the spherical coordinate, we have 
		\begin{align*}
			\int_{\S^2} e^{-k \sigma \cdot x} d\sigma = \int_0^{2\pi} \int_0^{\pi} e^{-k |x| \cos \theta} \sin \theta d\theta d \phi 
			= 2\pi \int_{-1}^{1} e^{-k|x|y} dy 
			= 2\pi \frac{e^{k|x|} - e^{-k|x|}}{k|x|} 
			= 4\pi \frac{\sinh (k|x|)}{k|x|},
		\end{align*}
		where the equality comes from the change of variables $y:= \cos \theta$. 
	\end{proof} 
	\begin{lemma}\label{I1,I2}
		For any $v \in \mathbb{R}^3$ and indices $i, j \in \{1, \dots, N\}$, consider 
		\[
		I(v) := \int_{\mathbb{R}^3} \frac{1}{|m_i v - m_j v'_*|^{5/2}} \frac{1}{1 + |v'_*|} \, dv'_*.
		\]
		Then, there exists a positive constant $C>0$ such that
		\begin{align*}
			I(v) \leq \frac{C}{(1 + |v|)^{1/2}}.
		\end{align*}
	\begin{proof}
		Let 
		\begin{align*}
			m_{\min} := \min_{1 \le k \le N} m_k \quad \textrm{and} \quad m_{\max} :=\max_{1 \le k \le N} m_k.
		\end{align*}
		We define the constants
		\[
		a := \frac{m_{\min}}{2m_{\max}} \quad \textrm{and} \quad A := \frac{2m_{\max}}{m_{\min}}.
		\]
		To prove the estimate for $|v| \ge 1$, we decompose the integration domain into three regions:
		\[
		A_1 := \{v'_* \in \R^3 : |v_*'| \le a|v|\}, \quad
		A_2 := \{v'_* \in \R^3: a|v| < |v_*'| < A|v|\}, \quad
		A_3 := \{v'_* \in \R^3: |v_*'| \ge A|v|\}.
		\]
		Then, we decompose the integral $I(v)$ as $I(v) = I_1(v) + I_2(v) + I_3(v)$, where $I_k$ corresponds to the integral over $A_k$.
		\medskip
		
		\noindent\textbf{Case 1. Estimate on $A_1$} 
		For $v'_* \in A_1$, we have
		\[
		|v'_*| \le a|v| = \frac{m_{\min}}{2m_{\max}}|v|.
		\]
		Since $m_j \le m_{\max}$ and $m_i \ge m_{\min}$, it follows that
		\[
		m_j |v'_*|
		\le m_{\max} \frac{m_{\min}}{2m_{\max}} |v|
		= \frac{m_{\min}}{2}|v|
		\le \frac{m_i}{2}|v|.
		\]
		Hence, by the triangle inequality,
		\[
		|m_i v - m_j v'_*|
		\ge m_i|v| - m_j|v'_*|
		\ge \frac{m_i}{2}|v|
		\ge \frac{m_{\min}}{2}|v|,
		\]
		which implies $|m_i v - m_j v_*'|^{-5/2} \le C|v|^{-5/2}$. Then, we obtain the estimate for $I_1$ as follows   
		\[
		I_1(v) \le C|v|^{-5/2} \int_{|v_*'| \le a|v|} \frac{1}{1+|v_*'|} \, dv_*' \le C|v|^{-5/2} \int_{0}^{a|v|} \frac{r^2}{1+r} \, dr \le C|v|^{-1/2}.
		\]
		This yields the bound  $I_1(v) \leq C(1+|v|)^{-1/2}$ for $|v| \ge 1$.
		
		\medskip
		\noindent\textbf{Case 2. Estimate on $A_2$} 
		For $v'_* \in A_2$, we have 
		\[
		1 + |v'_*| \ge 1 + a|v| \ge C(1+|v|).
		\]
		Using the change of variables $w = m_i v - m_j v'_*$, we have $dv'_* = m_j^{-3} \, dw$. 
		Since $a|v| < |v'_*| < A|v|$, it follows that $|w| = |m_i v - m_j v'_*| \le C|v|$. 
		Therefore,
		\[
		I_2(v) \le \frac{C}{1+|v|} \int_{|w| \le C|v|} \frac{1}{|w|^{5/2}} \, dw 
		= \frac{C}{1+|v|} \int_{0}^{C|v|} r^{2-5/2} \, dr 
		\le \frac{C|v|^{1/2}}{1+|v|} 
		\le C(1+|v|)^{-1/2}.
		\]
		\medskip
		\noindent\textbf{Case 3. Estimate on $A_3$} 
		For $v'_* \in A_3$, we have 
		\[
		|v'_*| \ge A|v| = \frac{2m_{\max}}{m_{\min}}|v|.
		\]
		Since $m_j \ge m_{\min}$ and $m_i \le m_{\max}$, it follows that
		\[
		m_j |v'_*| \ge m_j A|v| \ge 2 m_i |v|.
		\]
		By the triangle inequality,
		\[
		|m_i v - m_j v'_*|
		\ge m_j|v'_*| - m_i|v|
		\ge \frac{m_j}{2}|v'_*|.
		\]
		Therefore,
		\[
		I_3(v) \le C \int_{|v'_*| \ge A|v|} |v'_*|^{-5/2} (1+|v'_*|)^{-1} \, dv'_*
		\le C \int_{A|v|}^{\infty} r^{-3/2} \, dr 
		\le C|v|^{-1/2} 
		\le C(1+|v|)^{-1/2},
		\]
		for $|v| \geq 1$. \\
		\medskip 
		\indent For $|v|<1$, we apply the change of variables $u = m_i v - m_j v'_*$ to obtain
		\[ 
		I(v) = C \int_{\mathbb{R}^3} \frac{1}{|u|^{5/2}} \frac{1}{1 + |m_j^{-1}(m_i v - u)|} \, du. 
		\]
		To show that $I(v)$ is uniformly bounded for $|v|<1$, we split the integral into two parts: $\{|u|\le R\}$ and $\{|u|>R\}$ for sufficiently large $R>0$. For $|u|\le R$, we have $1 + |m_j^{-1}(m_i v - u)| \ge 1$. Thus,
		\[
		\int_{|u| \le R} \frac{1}{|u|^{5/2}} \frac{1}{1 + |m_j^{-1}(m_i v - u)|} \, du 
		\le \int_{|u| \le R} \frac{1}{|u|^{5/2}} \, du 
		= 4\pi \int_0^R r^{2-5/2} \, dr 
		= 8\pi R^{1/2} < \infty.
		\]
		For $|u|>R$ and $|v|<1$, we have
		\[
		|m_j^{-1}(m_i v - u)| \ge m_j^{-1}(|u| - m_i|v|) \ge m_j^{-1}(|u| - m_{\max}).
		\]
		Choosing $R$ sufficiently large so that $|u| - m_{\max} \ge \frac{1}{2}|u|$, we obtain
		\[
		\int_{|u| > R} \frac{1}{|u|^{5/2}} \frac{1}{1 + |m_j^{-1}(m_i v - u)|} \, du 
		\le C \int_{|u| > R} \frac{1}{|u|^{5/2}} \frac{1}{|u|} \, du 
		= C \int_R^\infty r^{2-7/2} \, dr 
		= C R^{-1/2} < \infty.
		\]
		Since both terms are bounded independently of $v$, we obtain $\sup_{|v|<1} I(v) \le C$. 
		Combining this with the case $|v|\ge 1$, we conclude that
		\[
		I(v) \leq \frac{C}{(1+|v|)^{1/2}}, \quad \forall v \in \mathbb{R}^3.
		\]
	\end{proof}
	\end{lemma}
We now state the main estimate of this subsection. When the Carleman representation of Lemma \ref{carleman 2} is  applied directly to the gain integral, five exponential factors appear involving $|v|^2$, $|v'_*|^2$, $R^2_{vv'_*}$, $|O_{vv'_*}|^2$, and the cross-term $R_{vv'_*}|O_{vv'_*}|$. Some of these factors decay while others grow in the velocity variable. In particular, the cross-term in particular obstructs a direct sign analysis. The lemma below shows that, after substituting the explicit forms of $R_{vv'_*}$ and $O_{vv'_*}$ from \eqref{RO}, all five factors combine into a single negative semi-definite quadratic form. This structure provides sufficient decay in the velocity variable to establish the pointwise gain estimate. To our knowledge, this cancellation has not previously appeared in the literature on multi-species kinetic equations.
\begin{lemma}\label{carleman 3}
		Let $w(v)$ denote the velocity weight function defined in \eqref{v_weight}. For all $i,j \in \{1,\cdots, N\}$, we have the following estimate 
		\begin{align} \label{main carleman}
				& \int_{\R^3 \times \S^2}|v-v_*|^\gamma b_{ij}(\cos\theta)\frac{\sqrt{\mu_i(v')}\sqrt{\mu_j(v'_*)}}{\sqrt{\mu_i(v)}}|f_j(v'_*)|d\sigma dv_* \cr
				&\leq  C\int_{\R^3} \frac{1}{|m_iv - m_j v'_*|}\frac{1}{|v-v'_*|^{1-\gamma}} (1+|v'_*|)^{-q} |wf_j(v'_*)| dv'_*,
		\end{align}
		for some constant $C>0$. 
\end{lemma}
\begin{proof}
	  We use the Carleman representation in Lemma \ref{carleman 2} to estimate the left-hand side of \eqref{main carleman}  
	\begin{align*}
		 &\int_{\R^3 \times \S^2}|v-v_*|^\gamma b_{ij}(\cos\theta)\frac{\sqrt{\mu_i(v')}\sqrt{\mu_j(v'_*)}}{\sqrt{\mu_i(v)}}|f_j(v'_*)|d\sigma dv_* \cr 
		 & \leq C\int_{\R^3}\frac{1}{|v-v'_*|}\frac{\sqrt{\mu_j(v'_*)}}{\sqrt{\mu_i(v)}}\int_{\tilde{E}^{ij}_{vv'_*}} \frac{1}{|v'_*-v'|^{1-\gamma}} \sqrt{\mu_i(v')} dE(v')|f_j(v'_*)| dv'_*.
		\end{align*}
	Since $\tilde{E}^{ij}_{vv'_*}$ is the sphere defined in \eqref{RO}, we make a change of variables to end up on $\S^2$: 
	\begin{align*}
			\frac{1}{|v-v'_*|} \int_{\tilde{E}^{ij}_{vv'_*}} \frac{1}{|v'_*-v'|^{1-\gamma}} \sqrt{\mu_i(v')} dE(v') &\leq C|v-v'_*|^\gamma\int_{\S^2} e^{-\frac{m_i}{4} |R_{vv'_*}\sigma + O_{vv'_*}|^2} d\sigma\cr
			&=C|v-v'_*|^\gamma e^{-\frac{m_i}{4} R^2_{vv'_*}}  e^{-\frac{m_i}{4} |O_{vv'_*}|^2} \int_{\S^2}  e^{-\frac{m_i}{2} R_{vv'_*} \sigma \cdot O_{vv'_*}} d\sigma,
		\end{align*}
		where we used $|v-v'_*|^{1-\gamma} \leq |v'-v'_*|^{1-\gamma}$ for $0\leq \gamma \leq 1 $. Using Lemma \ref{exp_cal}, we can further estimate 
	\begin{align*}
		& \frac{1}{|v-v'_*|} \frac{\sqrt{\mu_j(v'_*)}}{\sqrt{\mu_i(v)}}\int_{\tilde{E}^{ij}_{vv'_*}} \frac{1}{|v'_*-v'|^{1-\gamma}}\sqrt{\mu_i(v')} dE(v') \\
		&\leq C |v-v'_*|^{\gamma}e^{-\frac{m_j}{4} |v'_*|^2} e^{\frac{m_i}{4} |v|^2} e^{-\frac{m_i}{4} R^2_{vv'_*}}  e^{-\frac{m_i}{4} |O_{vv'_*}|^2}\frac{e^{\frac{m_i}{2} R_{vv'_*} |O_{vv'_*}|}}{R_{vv'_*} |O_{vv'_*}|}\\
		&= \frac{C}{|m_i v - m_j v'_*| }\frac{1}{|v-v'_*|^{1-\gamma}}e^{-\frac{m_j}{4} |v'_*|^2} e^{\frac{m_i}{4} |v|^2} e^{-\frac{m_i}{4} R^2_{vv'_*}}  e^{-\frac{m_i}{4} |O_{vv'_*}|^2}e^{\frac{m_i}{2} R_{vv'_*} |O_{vv'_*}|}.
	\end{align*}
	Using the definition of the radius $R_{vv'_*}$ and center $O_{vv'_*}$ of the sphere $\tilde{E}_{vv'_*}^{ij}$, we rewrite the exponent in the upper bound above: 
	\begin{align*}
		&-\frac{m_j}{4} |v'_*|^2 + \frac{m_i}{4} |v|^2 -\frac{m_im_j^2}{4|m_i - m_j|^2} |v-v'_*|^2 \\
		&\quad -\frac{m_i}{4|m_i-m_j|^2} |m_i v - m_j v'_*|^2 + \frac{m_im_j}{2|m_i-m_j|^2} |v-v'_*| |m_i v - m_j v'_*|\\
		&= 	-\frac{m_j}{4} |v'_*|^2 + \frac{m_i}{4} |v|^2 - \frac{m_im_j^2}{4|m_i-m_j|^2} |v|^2 - \frac{m_i m_j^2}{4|m_i-m_j|^2} |v'_*|^2 + \frac{m_i m_j^2}{2|m_i-m_j|^2} v \cdot v'_* \\
		&\quad - \frac{m_i^3}{4|m_i-m_j|^2} |v|^2 + \frac{m_i^2 m_j}{2|m_i-m_j|^2} v \cdot v'_* -\frac{m_i m_j^2}{4|m_i-m_j|^2} |v'_*|^2 + \frac{m_im_j}{2|m_i-m_j|^2} |v-v'_*| |m_i v - m_j v'_*| \\
		&= -\left(\frac{m_j}{4}+\frac{m_im_j^2}{2|m_i-m_j|^2}\right) |v'_*|^2 - \left(\frac{m_i (m_i^2 +m_j^2)}{4|m_i-m_j|^2} - \frac{m_i}{4}\right) |v|^2 \\
		&\quad + \frac{m_im_j(m_i+m_j)}{2|m_i-m_j|^2} v \cdot v'_* + \frac{m_im_j}{2|m_i-m_j|^2} |v-v'_*| |m_i v - m_j v'_*| \\ 
		&= - \left( \frac{m_j (m_i^2+m_j^2)}{4|m_i - m_j|^2}\right) |v'_*|^2 - \left(\frac{2m_i^2 m_j}{4|m_i-m_j|^2}\right) |v|^2 \\
		&\quad + \frac{m_im_j(m_i+m_j)}{2|m_i-m_j|^2} v \cdot v'_*+ \frac{m_im_j}{2|m_i-m_j|^2} |v-v'_*| |m_i v - m_j v'_*|  \\
		&= -\frac{m_j}{4|m_i-m_j|^2} \left((m_i^2 +m_j^2) |v'_*|^2 +2m_i^2 |v|^2 - 2m_i (m_i+m_j) v\cdot v'_*\right)+ \frac{m_im_j}{2|m_i-m_j|^2} |v-v'_*| |m_i v - m_j v'_*| \\ 
		&=-\frac{m_j}{4|m_i-m_j|^2} \left(|m_i v -m_j v'_*| ^2 + m_i^2 |v'_*|^2 +m_i^2 |v|^2 -2m_i^2 v \cdot v'_*\right)+ \frac{m_im_j}{2|m_i-m_j|^2} |v-v'_*| |m_i v - m_j v'_*| \\
		&=-\frac{m_j}{4|m_i-m_j|^2} \left( |m_iv - m_j v'_*|^2 + m_i^2 |v-v'_*|^2\right)+ \frac{m_im_j}{2|m_i-m_j|^2} |v-v'_*| |m_i v - m_j v'_*| \\
		&= -\frac{m_j}{4|m_i-m_j|^2} \left(|m_i v -m_j v'_*|- m_i |v-v'_*|\right)^2.
	\end{align*}
	Thus, one obtains
	\begin{align*}
			&\frac{C}{|m_i v - m_j v'_*| }\frac{1}{|v-v'_*|^{1-\gamma}}e^{-\frac{m_j}{4} |v'_*|^2} e^{\frac{m_i}{4} |v|^2} e^{-\frac{m_i}{4} R^2_{vv'_*}}  e^{-\frac{m_i}{4} |O_{vv'_*}|^2}e^{\frac{m_i}{2} R_{vv'_*} |O_{vv'_*}|}\cr
			&=  \frac{C}{|m_iv - m_j v'_*|} \frac{1}{|v-v'_*|^{1-\gamma}}e^{ -\frac{m_j}{4|m_i-m_j|^2} \left(|m_i v -m_j v'_*|- m_i |v-v'_*|\right)^2}.
		\end{align*}
		and consequently
	\begin{align*}
	&\int_{\R^3} \frac{1}{|m_iv - m_j v'_*|}\frac{1}{|v-v'_*|^{1-\gamma}} e^{ -\frac{m_j}{4|m_i-m_j|^2} \left(|m_i v -m_j v'_*|- m_i |v-v'_*|\right)^2} |f_j (v'_*)| dv'_* \cr
		&  \leq C\int_{\R^3} \frac{1}{|m_iv - m_j v'_*|}\frac{1}{|v-v'_*|^{1-\gamma}} (1+|v'_*|)^{-q} |wf_j(v'_*)| dv'_*.
	\end{align*}
\end{proof}
Combining Lemma~\ref{carleman 3} with Lemma~\ref{I1,I2}, we derive pointwise bounds for the gain term $\Gamma^+_{i}(\pmb{f})$. The following lemma provides the key estimate controlling the nonlinear gain term in terms of the weighted $L^\infty$ and $L^p$ norms of the perturbations.
		\begin{lemma} \label{gain est_new}
		Let $w(v)$ be defined as in \eqref{v_weight}. Assume that $\gamma \in [0,1]$. Then there exists a positive constant $C_{*,1}=C_{*,1}(q,\gamma)>0$ depending on $q$ and  $\gamma$ such that 
			\begin{align*}
					\sum_{j=1}^{N}|w(v)\Gamma^+(f_i,f_j)(v)| &\leq C_{*,1}\sum_{j=1}^{N}\frac{\|wf_j(t)\|_{L^\infty}}{1+|v|}\left(\int_{\R^3} (1+|\eta|)^{-2q+4}|wf_i(\eta)|^2d\eta\right)^\frac{1}{2}\cr
				&\quad +C_{*,1}  \frac{\Vert wf_i(t)\Vert_{L^\infty}}{(1+|v|)^{\frac{6}{5}-\gamma}}\sum_{j=1}^{N}\left( \int_{\R^3} (1+|\eta|)^{-5q + 10 } |wf_j(\eta)|^5 d\eta\right)^\frac{1}{5},
			\end{align*}
			where $\Gamma^+(f_i,f_j)(v)= \int_{\R^3 \times \S^2} B_{ij} (v-v_*,\sigma) \sqrt{\mu_j(v_*)} f_i(v') f_j(v_*')d\sigma dv_*$. 
		\end{lemma}
		\begin{proof} Recalling the decomposition $|w(v)\Gamma^+(f_i,f_j)(v)| \leq I_1 + I_2$ in \eqref{gain_split}, we estimate the two terms separately.
			
			\medskip
			\noindent\textbf{Estimate of $I_2$:} We adopt the $\omega$-parameterization instead of $\sigma$-representation, where the post-collisional velocities $v'$ and $v'_*$ are defined by:
			\begin{align}\label{v'v'*}
					v' &= v + \frac{2m_j}{m_i+m_j} \left( (v_*-v)\cdot \omega \right) \omega, \cr
					v'_* &= v_* - \frac{2m_i}{m_i+m_j} \left( (v_*-v)\cdot \omega \right) \omega.
			\end{align}
			From the representations \eqref{post-collision v} and \eqref{v'v'*} for the post-collisional velocities, we derive the following relation 
			\begin{align*}
			\sigma = \hat{u} -2(\hat{u} \cdot \omega)\omega \quad \textrm{and} \quad |\hat{u} \cdot \omega| = \sin (\theta/2),
			\end{align*}
			where $\hat{u}= \frac{v-v_*}{|v-v_*|}$ and $\cos \theta=  \hat{u} \cdot \sigma$. 
			The Jacobian of the map $\sigma \in \mathbb{S}^2 \mapsto \omega \in \mathbb{S}^2$ becomes 
			\begin{align*}
				d\sigma = 4 |\hat{u} \cdot \omega| d\omega.
			\end{align*}
			By the change of variables from $\sigma$ to $\omega$ by using the relation  above, $I_2$ becomes
			\begin{align*}
				I_2 &= C_q\|wf_j(t)\|_{L^\infty} \int_{\R^3 \times \S^2} B_{ij}(v-v_*,\sigma) \sqrt{\mu_j(v_*)} |f_i(v')| \, d\sigma dv_* \cr
				&= 4 C_q\|wf_j(t)\|_{L^\infty} \int_{\R^3 \times \S^2} B_{ij}(v-v_*,\omega) \left| \frac{v-v_*}{|v-v_*|} \cdot \omega \right| \sqrt{\mu_j(v_*)} |f_i(v')| \, d\omega dv_*.
			\end{align*}
			We denote $v'_*=v+z-\frac{2m_i}{m_i+m_j}z_\parallel$, $v'=v+\frac{2m_j}{m_i+m_j}z_\parallel$ with $z=v_*-v$, $z_\parallel=(z\cdot\omega)\omega$, $z_\perp=z-(z\cdot\omega)\omega$, and using the change of variables $v_* \rightarrow z$, then we have 
			\begin{align*}
			&\int_{\R^3 \times \S^2} B_{ij}(v-v_*,\omega) \left| \frac{v-v_*}{|v-v_*|} \cdot \omega \right| \sqrt{\mu_j(v_*)} |f_i(v')| \, d\omega dv_*\\
			&\leq C\int_{\R^3 \times \S^2}\frac{|z_\parallel|}{(|z_\parallel|^2+|z_\perp|^2)^{\frac{1-\gamma}{2}}}e^{-\frac{m_j}{4}|v+z|^2}\left|f_i\left(v+\frac{2m_j}{m_i+m_j}z_\parallel\right)\right|d\omega dz.	 
				\end{align*}
			Here, we have used \eqref{kernel} and the following estimate
			\begin{align*}
					B_{ij}(v-v_*,\omega)\left| \frac{v-v_*}{|v-v_*|} \cdot \omega \right| 
					\leq C^{\Phi}_{ij}|v-v_*|^\gamma\left|\frac{v-v_*}{|v-v_*|}\cdot\omega\right|
					\leq C_{\Phi}|z_\parallel|\cdot|z|^{\gamma-1}
					=C_{\Phi}|z_\parallel|\cdot\left(|z_\perp|^2+|z_\parallel|^2\right)^\frac{\gamma-1}{2},
			\end{align*}
			where $C_{\Phi}:= \max_{1\leq i,j \leq N} C_{ij}^{\Phi}>0$. 
			For fixed $\omega \in \S^2$, the change of variables $z\rightarrow (z_{\parallel},z_\perp)$ is a rotation with unit Jacobian. Thus, 
			\begin{align*}
				d\omega dz = d\omega 2 dz_\perp d|z_\parallel| = 2 dz_{\perp} \frac{|z_{\parallel}|^2 d|z_{\parallel}| d\omega}{|z_{\parallel}|^2} = \frac{2dz_{\perp}dz_{\parallel}}{|z_{\parallel}|^2}.
			\end{align*}
			Then, 
			\begin{align}\label{I-es}
				&\int_{\R^3 \times \S^2}|z_\parallel|(|z_\parallel|^2+|z_\perp|^2)^{\frac{\gamma-1}{2}}e^{-\frac{m_j}{4}|v+z|^2}|f_i(v+\frac{2m_j}{m_i+m_j}z_\parallel)|dzd\omega \cr
				&= 	\int_{\R^3}\frac{1}{|z_{\parallel}|}\int_{z_{\parallel}\cdot z_\perp =0}(|z_\parallel|^2+|z_\perp|^2)^{\frac{\gamma-1}{2}}e^{-\frac{m_j}{4}|v+z_{\parallel}+z_\perp|^2}|f_i(v+\frac{2m_j}{m_i+m_j}z_\parallel)|dz_{\perp}dz_{\parallel}\cr
				&\leq 	\int_{\R^3}\frac{1}{|z_{\parallel}|}\int_{z_{\parallel}\cdot z_\perp =0}|z_\perp|^{\gamma-1}e^{-\frac{m_j}{4}|v+z_{\parallel}+z_\perp|^2}|f_i(v+\frac{2m_j}{m_i+m_j}z_\parallel)|dz_{\perp}dz_{\parallel}\cr
				&= C \int_{\R^3} \frac{1}{|\eta - v|} \int_{(\eta-v)\cdot z_{\perp}=0} |z_\perp|^{\gamma-1} e^{-\frac{m_j}{4}|\frac{m_i+m_j}{2m_j}\eta+\frac{m_j-m_i}{2m_j}v+z_\perp|^2}|f_i(\eta)| dz_\perp d\eta,
			\end{align}
			where the last equality comes from the change of variables $\eta:= v+ \frac{2m_j}{m_i+m_j}z_{\parallel}$. For the integral with respect to $dz_\perp$, by $0\leq \gamma\leq 1$, it holds that
			\begin{align}\label{zperp}
					&\int_{(\eta-v)\cdot z_\perp=0}|z_\perp|^{\gamma-1}e^{-\frac{m_j}{4}|\frac{m_i+m_j}{2m_j}\eta+\frac{m_j-m_i}{2m_j}v+z_\perp|^2}dz_\perp \cr
					&=	 	 
					\left(\int_{|z_\perp| \leq 1} + \int_{|z_\perp| \geq 1} \right)
					|z_\perp|^{\gamma-1} e^{-\frac{m_j}{4}|\frac{m_i+m_j}{2m_j}\eta+\frac{m_j-m_i}{2m_j}v+z_\perp|^2} dz_\perp\cr
					&\leq \int_{|z_\perp| \leq 1} |z_\perp|^{\gamma-1} dz_\perp +
					\int_{|z_\perp| \geq 1} e^{-\frac{m_j}{4}|\frac{m_i+m_j}{2m_j}\eta+\frac{m_j-m_i}{2m_j}v+z_\perp|^2} dz_\perp\cr
					&\leq C.
			\end{align}
			Combining \eqref{I-es} and\eqref{zperp} and using H\"older's inequality, we have 
			\begin{align*}
			&\int_{\R^3 \times \S^2}|z_\parallel|(|z_\parallel|^2+|z_\perp|^2)^{\frac{\gamma-1}{2}}e^{-\frac{m_j}{4}|v+z|^2}|f_i(v+\frac{2m_j}{m_i+m_j}z_\parallel)|dzd\omega \\ 
			&\leq C \int_{\R^3} \frac{|f_i(\eta)|}{|\eta-v|} d\eta \\ &\leq C\left(\int_{\R^3}\frac{1}{|\eta-v|^2\cdot(1+|\eta|)^4}d\eta\right)^{\frac{1}{2}}\left(\int_{\R^3}(1+|\eta|)^{4}|f_i(\eta)|^2 d\eta\right)^\frac{1}{2}\cr
				&\leq  \frac{C}{1+|v|}\left(\int_{\R^3}(1+|\eta|)^{4}|f_i(\eta)|^2d\eta\right)^\frac{1}{2}.
			\end{align*}
			In conclusion, we get a desired result for the $I_2$ term as follows 
			\begin{align} \label{I2}
				I_2\leq C_{q} \frac{\|wf_j(t)\|_{L^\infty}}{1+|v|}\left(\int_{\R^3} (1+|\eta|)^{-2q+4}|wf_i(\eta)|^2d\eta\right)^\frac{1}{2}.
			\end{align}
 	\medskip
 	\noindent\textbf{Estimate of $I_1$:} We apply Lemma \ref{carleman 3} to obtain
		\begin{align*}
			I_1 &=C_q\|wf_i(t)\|_{L^\infty}  \int_{\R^3 \times \S^2}B(v-v_*,\sigma)\frac{\sqrt{\mu_i(v')}\sqrt{\mu_j(v'_*)}}{\sqrt{\mu_i(v)}}|f_j(v'_*)|d\sigma dv_*\cr
			&\leq    C_q \Vert w f_i (t) \Vert_{L^\infty} \int_{\R^3} \frac{1}{|m_iv - m_j v'_*|}\frac{1}{|v-v'_*|^{1-\gamma}} (1+|v'_*|)^{-q} |wf_j(v'_*)| dv'_*.
		\end{align*}
		Using Hölder's inequality with conjugate exponents $p=\frac{5}{4}$ and $q=5$, we obtain		
		\begin{align*}
			&\int_{\R^3} \frac{1}{|m_iv - m_j v'_*|}\frac{1}{|v-v'_*|^{1-\gamma}} (1+|v'_*|)^{-q} |wf_j(v'_*)| dv'_*\cr
			&\leq\left(\int_{\R^3} \frac{1}{|m_iv -m_j v'_*|^\frac{5}{4}}\frac{1}{|v-v'_*|^{\frac{5}{4}(1-\gamma)}}\frac{1}{(1+|v'_*|)^\frac{5}{2}} dv'_*\right)^{4/5} \left( \int_{\R^3} (1+|v'_*|)^{-5q + 10 } |wf_j(v'_*)|^5 dv'_*\right)^{1/5}.
		\end{align*}
		To complete the estimate, we analyze the decay in $v$ of the first factor arising from Hölder's inequality.
		Applying Hölder's inequality once more, we obtain
\begin{align*}
&\int_{\mathbb{R}^3} \frac{1}{|m_i v - m_j v'_*|^{\frac{5}{4}}}
\frac{1}{|v-v'_*|^{\frac{5}{4}(1-\gamma)}}
\frac{1}{(1+|v'_*|)^{\frac{5}{2}}} \, dv'_* \\
&\le 
\left(\int_{\mathbb{R}^3}\frac{1}{|v-v'_*|^{\frac{5}{2}(1-\gamma)}(1+|v'_*|)^4}\, dv'_*\right)^{\frac{1}{2}}
\left(\int_{\mathbb{R}^3}\frac{1}{|m_i v - m_j v'_*|^{\frac{5}{2}}}\frac{1}{(1+|v'_*|)}\, dv'_*\right)^{\frac{1}{2}}.
\end{align*}
The first term is estimated by standard convolution bounds, while the second term is controlled by Lemma \ref{I1,I2}.
Consequently, we obtain
\begin{align*}
\int_{\mathbb{R}^3} \frac{1}{|m_i v - m_j v'_*|^{\frac{5}{4}}}
\frac{1}{|v-v'_*|^{\frac{5}{4}(1-\gamma)}}
\frac{1}{(1+|v'_*|)^{\frac{5}{2}}} \, dv'_*
\le \frac{C_{\gamma}}{(1+|v|)^{\frac{3}{2}-\frac{5}{4}\gamma}},
\end{align*}
where the constant $C_{\gamma}>0$ depends on $\gamma$. 
		Combining the above estimates, we obtain
		\begin{align} \label{I1} 
			I_1&\leq  C_q \Vert w f_i (t) \Vert_{L^\infty} \int_{\R^3} \frac{1}{|m_iv - m_j v'_*|}\frac{1}{|v-v'_*|^{1-\gamma}} (1+|v'_*|)^{-q} |w(v'_*)f_j(v'_*)| dv'_*\cr
			&\leq C_{q,\gamma}\frac{\Vert wf_i(t)\Vert_{L^\infty}}{(1+|v|)^{\frac{6}{5}-\gamma}}\left( \int_{\R^3} (1+|v'_*|)^{-5q + 10 } |wf_j(v'_*)|^5 dv'_*\right)^{1/5}.
		\end{align}
		Hence, from \eqref{gain_split},\eqref{I2}, and \eqref{I1}, there exists a positive constant $C_{*,1}= C_{*,1}(q, \gamma)>0$ depending on $q$ and $\gamma$ such that 
		\begin{align*}
			|w(v)\Gamma^+(f_i,f_j)(v)| &\leq C_{*,1}\frac{\|wf_j(t)\|_{L^\infty}}{1+|v|}\left(\int_{\R^3} (1+|\eta|)^{-2q+4}|wf_i(\eta)|^2d\eta\right)^\frac{1}{2}\cr
			&\quad +C_{*,1}\frac{\Vert wf_i(t)\Vert_{L^\infty}}{(1+|v|)^{\frac{6}{5}-\gamma}}\left( \int_{\R^3} (1+|\eta|)^{-5q + 10 } |wf_j(\eta)|^5 d\eta\right)^\frac{1}{5}
		\end{align*}
		and
		\begin{align*}
			\sum_{j=1}^{N}|w(v)\Gamma^+(f_i,f_j)(v)| &\leq C_{*,1}\sum_{j=1}^{N}\frac{\|wf_j(t)\|_{L^\infty}}{1+|v|}\left(\int_{\R^3} (1+|\eta|)^{-2q+4}|wf_i(\eta)|^2d\eta\right)^\frac{1}{2}\cr
		&\quad +C_{*,1}\frac{\Vert wf_i(t)\Vert_{L^\infty}}{(1+|v|)^{\frac{6}{5}-\gamma}}\sum_{j=1}^{N}\left( \int_{\R^3} (1+|\eta|)^{-5q + 10 } |wf_j(\eta)|^5 d\eta\right)^\frac{1}{5}.
		\end{align*}
			\end{proof}
		As a direct application of Lemma~\ref{gain est_new}, we obtain $L^\infty$-type bounds for the full nonlinear collision operator $\Gamma_i(\pmb{f})$.
		\begin{corollary} \label{nonlinear est} 
			For $\gamma \in [0,1]$, there exists a constant $C>0$ such that 
			\begin{align*}
				&|w(v)\Gamma^+_{i}(\pmb{f})(v)|\leq \sum_{j=1}^N |w(v) \Gamma^+(f_i,f_j)(v)| \leq C \Vert w \pmb{f}(t)\Vert_{L^{\infty}}^2,\\
				&|w(v) \Gamma^-_{i} (\pmb{f})(v)| \leq \sum_{j=1} ^N |w(v) \Gamma^-(f_i,f_j)(v)| \leq C \langle v \rangle^{\gamma} \Vert wf_i(t) \Vert_{L^\infty}\Vert w\pmb{f}(t) \Vert_{L^\infty},
			\end{align*}
			and consequently, 
			\begin{align*}
				|w(v) \Gamma_{i} (\pmb{f})(v)| \leq C\langle v \rangle^{\gamma} \Vert w \pmb{f}(t) \Vert_{L^\infty}^2. 
			\end{align*}
		\end{corollary}
		\begin{proof}
			For the part $w(v) \Gamma^+_{i}(\pmb{f})(v)$, we can directly deduce the estimate from Lemma \ref{gain est_new}. By definition of $\Gamma^-_{i}$ and \eqref{kernel}, we have 
			\begin{align*}
				|w(v) \Gamma^-_{i}(\pmb{f})(v)|  \leq \sum_{j=1} ^N |w(v) \Gamma^-(f_i,f_j)(v)| &\leq \Vert wf_i(t) \Vert_{L^\infty} \sum_{j=1}^N \int_{\R^3} \int_{\S^2} B_{ij}(v-v_*,\sigma) \sqrt{\mu_j(v_*)}|f_j(v_*)| d\sigma dv_*\\
				&\leq C \Vert w f_i(t) \Vert_{L^\infty} \sum_{j=1}^N \int_{\R^3} |v-v_*|^{\gamma} \sqrt{\mu_j(v_*)}w^{-1}(v_*)|wf_j(v_*)|  dv_*\\
				&\leq C \langle v \rangle^{\gamma} \Vert wf_i(t) \Vert_{L^\infty} \Vert w \pmb{f}(t) \Vert_{L^\infty}.
			\end{align*}
			Combining the above estimates yields the desired result.
		\end{proof}
		
		\section{Small-amplitude perturbation problem}\label{sec:small_regime}
		In this section, we establish the global well-posedness of the multi-species Boltzmann equation in the small-amplitude regime, thereby proving Theorem~\ref{thm:main2}. These results will also provide the final ingredient in the large-amplitude analysis of Section~\ref{sec:large_regime}, where the solution is shown to enter the small-amplitude regime after a finite time.
		\subsection {Linear $L^2$ decay theory}
		In this subsection, we consider the following linear system for $\pmb{f}$: 
		\begin{align} \label{linear BE}
			\p_t \pmb{f} + v \cdot \nabla_x \pmb{f} = \pmb{L}(\pmb{f}) + \pmb{q}.
		\end{align}
		Here, $\pmb{f} = (f_1, \dots, f_N)$, and the system is understood componentwise for each species.
		Recall the definition of linear multi-species Boltzmann operator 
		\begin{align*}
			\pmb{L}(\pmb{f}) = (L_i(\pmb{f}))_{1\leq i \leq N} =\left( \sum_{j=1}^N L_{ij}(\pmb{f})\right)_{1\leq i \leq N},
		\end{align*}
		where 
		\begin{align*}
			L_{ij}(\pmb{f}) = \frac{1}{\sqrt{\mu_i}} \left[Q_{ij}(\mu_i, \sqrt{\mu_j}f_j)+ Q_{ij}(\sqrt{\mu_i}f_i,\mu_j)\right].
		\end{align*}
		The operator $\pmb{L}$ is self-adjoint in $L^2_v$. Moreover, $\langle \pmb{f}, \pmb{L}(\pmb{f})\rangle_{L^2_v}=0$ if and only if $\pmb{f}\in \mathrm{Ker}(\pmb{L})$, where
		\begin{align*}
			\mathrm{Ker}(\pmb{L})=\mathrm{Span}\{\pmb{\phi}_1,\pmb{\phi}_2,\dots,\pmb{\phi}_{N+4}\}.
		\end{align*}
		Here, the orthonormal basis $\{\pmb{\phi}_i\}_{1\leq i \leq N+4}$ is given by
		\begin{align} \label{phi} 
			\begin{cases}
				\pmb{\phi}_i = \pmb{\phi}_i(v)= \frac{1}{\sqrt{n_{\infty,i}}} \sqrt{\mu_i(v)} \pmb{e}_i, \quad 1\leq i \leq N,\\
				\pmb{\phi}_{N+i} = \pmb{\phi}_{i+N}(v) =\frac{v_i}{\left(\sum_{k=1}^N m_k n_{\infty,k}\right)^{1/2}} \sum_{j=1}^N \left(m_j \sqrt{\mu_j(v)}\right) \pmb{e}_j, \quad 1 \leq i \leq 3,\\
				\pmb{\phi}_{N+4}  =  \pmb{\phi}_{N+4} (v) = \frac{1}{\left(\sum_{k=1}^N  n_{\infty,k}\right)^{1/2}}  \sum_{j=1}^N \left(\frac{|v|^2 -3m_j^{-1}}{\sqrt{6}}m_j \sqrt{\mu_j(v)}\right)\pmb{e}_j,
			\end{cases}
		\end{align}
		where $\pmb{e}_j$ is the $j$-th unit vector in $\R^N$. To distinguish between the pointwise and $L^2_v$ inner products, we define
		\[
		\pmb{f}(v)\cdot \pmb{g}(v)=\sum_{j=1}^N f_j(v)\,g_j(v),
		\]
		and the $L^2_v$ inner product
		\[
		\langle \pmb{f},\pmb{g}\rangle_{L^2_v}=\int_{\mathbb{R}^3}\pmb{f}(v)\cdot \pmb{g}(v)\,dv
		=\sum_{j=1}^N\int_{\mathbb{R}^3} f_j(v)\,g_j(v)\,dv.
		\]
		We define the projection operator $\pmb{P}_{\pmb{L}}$ onto $\mathrm{Ker}(\pmb{L})$ in $L^2_v$ by 
		\begin{align*}
			\pmb{P}_{\pmb{L}}(\pmb{f}) := \sum_{i=1}^{N+4} \langle \pmb{f}, \pmb{\phi}_i \rangle_{L^2_v}\pmb{\phi}_i.
		\end{align*}
		We state a reformulated version of Theorem 3.3 in \cite{BD2016}, expressed in terms of the orthonormal basis $\{\pmb{\phi}_i\}_{1\leq i \leq N+4}$.
		\begin{proposition} \cite{BD2016} \label{L coer1}
			There exists a positive constant $\lambda_L>0$ such that 
			\begin{align*}
				\langle \pmb{f} , \pmb{L}(\pmb{f}) \rangle_{L^2_v}\leq - \lambda_L \Vert (\pmb{I}-\pmb{P}_{\pmb{L}})(\pmb{f})\Vert_{L^2_v(\langle v \rangle^{\gamma/2})}^2.
			\end{align*}
		\end{proposition} 
		\begin{proof}
			In \cite{BD2016}, they used the following perturbation 
			\begin{align*}
				\pmb{F} = \pmb{\mu} + \pmb{g},
			\end{align*} 
			instead of $\pmb{F} = \pmb{\mu} + \sqrt{\pmb{\mu}} \pmb{f}$. In \cite{BD2016}, In \cite{BD2016}, the estimate is formulated in the weighted space $L^2_v(\pmb{\mu}^{-1/2})$, equipped with the norm
			\begin{align*}
				\Vert  \pmb{f} \Vert_{L^2_v(\pmb{\mu}^{-1/2})}^2:= \sum_{i=1}^N \Vert f_i \mu_i^{-1/2}\Vert_{L^2_v}^2. 
			\end{align*}
			The coercivity estimate in \cite{BD2016} reads
			\begin{align} \label{Briant_L}
				\langle \pmb{g} , \pmb{L'}(\pmb{g}) \rangle_{L^2_v(\pmb{\mu}^{-1/2})}\leq - \lambda_L \Vert (\pmb{I}-\pmb{P}_{\pmb{L'}})(\pmb{g})\Vert_{L^2_v(\langle v \rangle^{\gamma/2}\pmb{\mu}^{-1/2})}^2.
			\end{align}
			The operator $\pmb{L'}$ is given by
			\begin{align*}
				\pmb{L'}(\pmb{g}) = (L'_i(\pmb{g}))_{1\leq i \leq N} =\left(\sum_{j=1}^N L'_{ij}(\pmb{g})\right)_{1\leq i \leq N},
			\end{align*}
			where
			\begin{align*}
				L'_{ij}(\pmb{g}) = Q_{ij}(\mu_i , g_j) + Q_{ij}(g_i ,\mu_j).
			\end{align*}
			Moreover, the kernel of $\pmb{L'}$ is spanned by
			\begin{align*}
				\mathrm{Ker}(\pmb{L'})  = \mathrm{Span}\{\pmb{\phi'}_1 , \pmb{\phi'}_2, \dots, \pmb{\phi'}_{N+4}\} ,
			\end{align*}
			where 
			\begin{align*}
				\begin{cases}
					\pmb{\phi'}_i = \pmb{\phi'}_i(v)= \frac{1}{\sqrt{n_{\infty,i}}} {\mu_i(v)} \pmb{e}_i, \quad 1\leq i \leq N,\\
					\pmb{\phi'}_{i+N} = \pmb{\phi'}_{i+N}(v) =\frac{v_i}{\left(\sum_{k=1}^N m_k n_{\infty,k}\right)^{1/2}} \sum_{j=1}^N \left(m_j {\mu_j(v)}\right) \pmb{e}_j, \quad 1 \leq i \leq 3,\\
					\pmb{\phi'}_{N+4}  =  \pmb{\phi'}_{N+4} (v) = \frac{1}{\left(\sum_{k=1}^N  n_{\infty,k}\right)^{1/2}}  \sum_{j=1}^N \left(\frac{|v|^2 -3m_j^{-1}}{\sqrt{6}}m_j {\mu_j(v)}\right)\pmb{e}_j.
				\end{cases}
			\end{align*}
			To rewrite \eqref{Briant_L} in terms of $\pmb{f}$, we set $\pmb{g} = \sqrt{\pmb{\mu}} \pmb{f}=(\sqrt{\mu_i}f_i)_{1\leq i \leq N}$. Then the left-hand side in \eqref{Briant_L} becomes 
			\begin{align*}
				\langle \pmb{g} , \pmb{L'}(\pmb{g}) \rangle_{L^2_v(\pmb{\mu}^{-1/2})} = \sum_{i=1}^N \int_{\R^3} \frac{ g_i {L'_i}(\pmb{g})}{ {\mu_i}}dv = \sum_{i=1}^N \int_{\R^3} f_i L_i (\pmb{f}) dv = \langle \pmb{f}, \pmb{L}(\pmb{f}) \rangle_{L^2_v},
			\end{align*}
			where we have used the fact $L_i' (\pmb{g}) = \sqrt{\mu_i} L_i(\pmb{f})$. For the right-hand side in \eqref{Briant_L}, we have 
			\begin{align*}
				(\pmb{I}-\pmb{P}_{\pmb{L'}})(\pmb{g}) &= \sqrt{\pmb{{\mu}}}\pmb{f} -\pmb{P}_{\pmb{L'}}(\pmb{g})\\
				&= \sqrt{\pmb{{\mu}}}\pmb{f} - \sum_{i=1}^{N+4}  \langle \sqrt{\pmb{{\mu}}}\pmb{f} , \pmb{\phi'_i} \rangle_{L^2_v(\pmb{\mu}^{-1/2})} \pmb{\phi_i '} \\
				&= \sqrt{\pmb{{\mu}}}\pmb{f} - \sqrt{\pmb{{\mu}}}\sum_{i=1}^{N+4} \langle \pmb{f} , \pmb{\phi_i} \rangle_{L^2_v} \pmb{\phi_i }\\
				&= \sqrt{\pmb{{\mu}}} (\pmb{I} - \pmb{P}_{\pmb{L}})(\pmb{f}).
			\end{align*}
			This implies that 
			\begin{align*}
				\Vert (\pmb{I}-\pmb{P}_{\pmb{L'}})(\pmb{g})\Vert_{L^2_v(\langle v \rangle^{\gamma/2}\pmb{\mu}^{-1/2})}^2 = \Vert  (\pmb{I} - \pmb{P}_{\pmb{L}})(\pmb{f}) \Vert_{L^2_v(\langle v \rangle^{\gamma/2})} ^2.
			\end{align*}
		\end{proof}
		To derive the $L^2$ decay of \eqref{linear BE}, it is essential to control the macroscopic component $\pmb{P}_{\pmb{L}}(\pmb{f})$ in terms of the microscopic component $(\pmb{I}-\pmb{P}_{\pmb{L}})(\pmb{f})$. This estimate is proved in \cite{BD2016} and can be stated as follows.
		\begin{lemma} \cite{BD2016} \label{L coer2}
			Let $\pmb{f}_0(x,v)$ and $\pmb{q}(t,x,v)$ belong to $L^2_{x,v}$ and satisfy
			\begin{align*}
				\sum_{i=1}^{N+4} \left(\int_{\T^3 } \langle \pmb{f}_0, \pmb{\phi}_i \rangle_{L^2_v} dx\right) \pmb{\phi}_i =\sum_{i=1}^{N+4} \left(\int_{\T^3 } \langle \pmb{q}, \pmb{\phi}_i \rangle_{L^2_v}dx\right) \pmb{\phi}_i =0,
			\end{align*}
			where $\pmb{\phi}_i$ is defined by \eqref{phi}. 
			Suppose that $\pmb{f}=\pmb{f}(t,x,v) \in L^2_{x,v}$ is a solution to \eqref{linear BE} with initial data $\pmb{f}_0$. Then, there exist a positive constant $C_1>0$ and a function $N_{\pmb{f}}(t)$ such that $|N_{\pmb{f}}(t) | \leq C_1 \Vert \pmb{f}(t) \Vert_{L^2 }^2$ and 
			\begin{align*}
				\int_0^t \Vert \pmb{P}_{\pmb{L}} (\pmb{f})(s)\Vert_{L^2 (\langle v \rangle^{\gamma/2})}^2 ds &\leq N_{\pmb{f}} (t) - N_{\pmb{f}}(0) \cr
				&\quad +C_1 \int_0^t \Vert (\pmb{I}-\pmb{P}_{\pmb{L}})(\pmb{f})(s) \Vert_{L^2 (\langle v \rangle^{\gamma/2})}^2 ds +C_1 \int_0^t \Vert \pmb{q}(s) \Vert_{L^2 }^2ds, 
			\end{align*}
		\end{lemma}
		\noindent for all $t \geq 0$. 
		\begin{remark}
			Similar to Proposition \ref{L coer1}, Lemma 4.2 in \cite{BD2016} is formulated for perturbations of the form $\pmb{F}= \pmb{\mu} + \pmb{g}$ within the weighted space $L^2_v (\pmb{\mu}^{-1/2})$. In our framework, with the perturbation $\pmb{F} = \pmb{\mu} + \sqrt{\pmb{\mu}} \pmb{f}$, the above estimate can be derived by following the same argument as in the proof of Proposition \ref{L coer1}. Therefore, we omit the detailed proof of Lemma \ref{L coer2}. 
		\end{remark}
		\begin{theorem} \label{L2 decay}
			Let $\pmb{f}(t,x,v)$ be a solution to \eqref{linear BE} with initial data $\pmb{f}_0$ and source term $\pmb{q}$. Suppose that initial data $\pmb{f}_0$ and source term $\pmb{q}$ satisfy 
			\begin{align*}
				\sum_{i=1}^{N+4} \left(\int_{\T^3} \langle \pmb{f}_0, \pmb{\phi}_i \rangle_{L^2_v}dx\right) \pmb{\phi}_i =\sum_{i=1}^{N+4} \left(\int_{\T^3} \langle \pmb{q}(t), \pmb{\phi}_i \rangle_{L^2_v} dx\right) \pmb{\phi}_i =0, \quad \forall t \geq 0. 
			\end{align*}
			Then, there exist constants $\lambda>0$ and $C_2>0$ such that 
			\begin{align*}
				e^{2\lambda t}\Vert  \pmb{f}(t) \Vert_{L^2}^2  \leq C_2 \Vert \pmb{f}_0 \Vert_{L^2}^2 + C_2 \int_0^t e^{2\lambda s} \Vert \pmb{q} \Vert_{L^2 }^2ds,
			\end{align*}
			for all $t \geq 0$.
		\end{theorem}
		\begin{proof}
			From \eqref{linear BE}, for any $\alpha>0$, we get 
			\begin{align*}
				\p_t (e^{\alpha t} \pmb{f}) + v\cdot \nabla_x (e^{\alpha t} \pmb{f}) = e^{\alpha t}\pmb{L}(\pmb{f}) +e^{\lambda t} \pmb{q} + \alpha e^{\alpha t} \pmb{f}. 
			\end{align*}
			For convenience of notation, we denote $\pmb{\tilde{f}}:= e^{\alpha t}\pmb{f}$. By the $L^2$-energy method together with Proposition \ref{L coer1}, we obtain 
			\begin{align*}
				\frac{1}{2} \frac{d}{dt} \Vert \pmb{\tilde{f}}(t) \Vert_{L^2 }^2 &=  \int_{\T^3}  \langle \pmb{L}(\pmb{\tilde{f}}),\pmb{\tilde{f}} \rangle_{L^2_{v}} dx  + \int_{\T^3} e^{\alpha t} \langle \pmb{q} , \pmb{\tilde{f}}\rangle_{L^2_v} dx  +\alpha \Vert \pmb{\tilde{f}}(t)\Vert_{L^2 }^2\\
				&\leq -\lambda_L \Vert  (\pmb{I}-\pmb{P}_{\pmb{L}})(\pmb{\tilde{f}})(t) \Vert^2_{L^2(\langle v \rangle^{\gamma/2})}+\int_{\T^3} e^{\alpha t} \langle \pmb{q} , \pmb{\tilde{f}}\rangle_{L^2_v} dx +\alpha \Vert \pmb{\tilde{f}}(t) \Vert_{L^2}^2.
			\end{align*}
			By taking the integration with respect to time, one has 
			\begin{align} \label{energy est} 
				&\Vert \pmb{\tilde{f}}(t) \Vert_{L^2 }^2 + 2\lambda_L \int_0^t 		\Vert (\pmb{I}-\pmb{P}_{\pmb{L}})(\pmb{\tilde{f}})(s) \Vert^2_{L^2 (\langle v \rangle^{\gamma/2})}ds \nonumber \\
				&\leq \Vert \pmb{f}_0 \Vert_{L^2 }^2 + 2\int_0^te^{\alpha s} \langle \pmb{q},\pmb{\tilde{f}}\rangle_{L^2 }(s) ds + 2\alpha \int_0^t \Vert \pmb{\tilde{f}}(s) \Vert_{L^2 }^2ds .
			\end{align}
			Note that $\pmb{\tilde{f}}$ satisfies
			\begin{align*}
				\sum_{i=1}^{N+4} \left( \int_{\T^3} \langle \pmb{\tilde{f}}(t) , \pmb{\phi}_i\rangle_{L^2_v} dx \right) \pmb{\phi}_i =0, \quad \forall t\geq 0 
			\end{align*}
			and solves the following equation
			\begin{align*}
				\p_t \pmb{\tilde{f}} + v\cdot \nabla_x \pmb{\tilde{f}} = \pmb{L}(\pmb{\tilde{f}}) + \alpha \pmb{\tilde{f}} + e^{\alpha t} \pmb{q},
			\end{align*} 
			with initial data $\pmb{\tilde{f}}(0) = \pmb{f}_0$. Therefore, by applying Lemma \ref{L coer2} with source term $e^{\alpha t} \pmb{q} + \alpha \pmb{\tilde{f}}$ in \eqref{linear BE}, we obtain the following estimate 
			\begin{align} \label{L coer3}
				\int_0^t \Vert \pmb{P}_{\pmb{L}} (\pmb{\tilde{f}})(s)\Vert_{L^2 (\langle v \rangle^{\gamma/2})}^2 ds &\leq N_{\pmb{\tilde{f}}} (t) - N_{\pmb{\tilde{f}}}(0) +C_1 \int_0^t \Vert (\pmb{I}-\pmb{P}_{\pmb{L}})(\pmb{\tilde{f}})(s) \Vert_{L^2 (\langle v \rangle^{\gamma/2})}^2 ds \nonumber \\
				&\quad+C_1 \int_0^t e^{2\alpha s}\Vert \pmb{q}(s) \Vert_{L^2 }^2ds + C_1 \alpha^2 \int_0^t \Vert \pmb{\tilde{f}}(s) \Vert_{L^2 }^2ds, 
			\end{align}
			For a fixed $\varepsilon>0$, we multiply \eqref{L coer3} by $\varepsilon$ and add it to \eqref{energy est} to obtain the following estimate: 
			\begin{align*}
				&\left[\Vert \pmb{\tilde{f}}(t) \Vert_{L^2 }^2 - \varepsilon N_{\pmb{\tilde{f}}}(t)\right] + C_{\varepsilon} \int_0^t \left[\Vert \pmb{P}_{\pmb{L}} (\pmb{\tilde{f}})(s)\Vert_{L^2 (\langle v \rangle^{\gamma/2})}^2 + \Vert (\pmb{I}-\pmb{P}_{\pmb{L}})(\pmb{\tilde{f}})(s) \Vert_{L^2 (\langle v \rangle^{\gamma/2})}^2\right] ds  \\
				&= \left[\Vert \pmb{\tilde{f}}(t) \Vert_{L^2 }^2 - \varepsilon N_{\pmb{\tilde{f}}}(t)\right] + C_{\varepsilon} \int_0^t  \Vert \pmb{\tilde{f}}(s) \Vert_{L^2 (\langle v \rangle^{\gamma/2})}^2ds \\
				&\leq \left[\Vert \pmb{f}_0 \Vert_{L^2 }^2 -\varepsilon N_{\pmb{\tilde{f}}}(0) \right] +2\int_{0}^t e^{\alpha s}\langle \pmb{q} , \pmb{\tilde{f}}\rangle_{L^2 }(s) ds + C_1 \varepsilon \int_0^t e^{2\alpha s}\Vert \pmb{q}(s) \Vert _{L^2}^2 ds\\
				&\quad + (2\alpha +C_1 \alpha^2 \varepsilon) \int_0^t \Vert \pmb{\tilde{f}}(s) \Vert_{L^2 }^2 ds \\ 
				&\leq   \left[\Vert \pmb{f}_0 \Vert_{L^2 }^2 -\varepsilon N_{\pmb{\tilde{f}}}(0) \right] + \left(\frac{C_{\varepsilon}}{2} +2\alpha + C_1 \alpha^2 \varepsilon \right)\int_0^t \Vert \pmb{f}(s) \Vert_{L^2 }^2 ds + \left(\frac{1}{2C_{\varepsilon}} +C_1\varepsilon\right)\int_0^t \Vert \pmb{q}(s) \Vert_{L^2 }^2 ds,
			\end{align*} 
			where $C_{\varepsilon} = \min\{2\lambda_L- \varepsilon C_1, \varepsilon\}$. If we choose $\varepsilon$ and $\alpha$ sufficently small such that $\frac{C_{\varepsilon}}{2} -2\alpha -C_1 \alpha^2 \varepsilon>0$ and $C_1 \varepsilon <\frac{1}{2}$, then there exists a positive constant $C_2>0$ such that  
			\begin{align*}
				e^{2\alpha t}\Vert \pmb{f}(t) \Vert_{L^2 }^2  \leq C_2 \Vert \pmb{f}_0 \Vert_{L^2 }^2 +C_2\int_0^t e^{2\alpha s}\Vert \pmb{q}(s) \Vert_{L^2 }^2 ds .   
			\end{align*}
			Thus, we derive the $L^2 $ estimate for the equation \eqref{linear BE} if we choose $\lambda = \alpha$. 
		\end{proof}
		\subsection{Linear $L^\infty$ decay theory} 
		In this subsection, we consider the equation \eqref{linear BE} with $\pmb{q} \equiv 0$ in the $v$-weighted $L^\infty$ space. If we denote $\pmb{h} = \pmb{h}(t,x,v)=w(v) \pmb{f}(t,x,v)$, the linearized Boltzmann equation with respect to $\pmb{h}$ can be written as 
		\begin{align} \label{weight BE}
			\p_t \pmb{h} + v \cdot \nabla_x \pmb{h} + \pmb{\nu} \pmb{h} = K_w (\pmb{h}). 
		\end{align} 
		Equivalently, for each $i\in \{1,\cdots,N\}$, we have 
		\begin{align} \label{ith BE} 
			\p_t h_i + v\cdot \nabla_x h_i + \nu_i(v) h_i  =K_{w,i}(\pmb{h}),
		\end{align}
		where $h_i$ is the $i$-th component of $\pmb{h}$. Building on ideas from \cite{Guo10}, the $L^2$ decay theory established in the previous subsection is essential for obtaining the $v$-weighted $L^\infty$ estimate for solutions to \eqref{weight BE}. The following lemma provides the corresponding $v$-weighted $L^\infty$ decay estimate.
		\begin{lemma} \label{small L infty}
			Assume that $\pmb{f}_0$ satisfy the condition \eqref{conserv}. Let $\pmb{h}(t,x,v)$ be the solution to the $v$-weighted linearized Boltzmann equation \eqref{weight BE} with initial data $\pmb{h}_0 = w(v) \pmb{f}_0(x,v)$.  
			Then there exist positive constants $C_3>0$ and $\lambda_1>0$ satisfying
			\begin{align*}
				\Vert \pmb{h}(t) \Vert_{L^\infty} \leq C_3 e^{-\lambda_1 t} \Vert \pmb{h}_0\Vert_{L^\infty}, \quad \forall t \geq 0.
			\end{align*} 
		\end{lemma}
		\begin{proof}
			Fix $i \in \{1,2,\cdots, N\}$. We firstly set $\lambda_1 : = \min \left \{ \lambda, \frac{
				\nu_0}{2}\right\}$, where positive constants $\nu_0$ and $\lambda$ are defined in Lemma \ref{CFE} and Theorem \ref{L2 decay}, respectively. For fixed $(t,x,v) \in [0,T] \times \T^3 \times \R^3$, applying Duhamel's principle to \eqref{ith BE} and using the fact $2\lambda_1\le \nu_i(v)$ gives
			\begin{align} \label{small split} 
				|h_i(t,x,v)| &\leq e^{-2\lambda_1 t}|h_{0,i}(x-vt,v)| + \int_0^t e^{-2\lambda_1(t-s)} |{K}_{w,i} (\pmb{h})(s,x-(t-s)v,v)| ds \nonumber \\
				&:= I + II.
			\end{align}
			For $I$, we can directly show that 
			\begin{align} \label{small I} 
				I \leq e^{-2\lambda_1 t} \Vert \pmb{h}_{0} \Vert_{L^\infty}.
			\end{align} 
			For the term $II$, we apply the Duhamel principle once more to derive 
			\begin{align}\label{small II}
				II &\leq \sum_{j=1}^N \int_0^t e^{-2\lambda_1(t-s)} \int_{\R^3} |k^w_{ij}(v,v_*) h_j(s,x-(t-s)v,v_*)|dv_* ds \nonumber \\
				&\leq  \sum_{j=1}^N \int_0^t e^{-2\lambda_1(t-s)}\int_{\R^3} |k^w_{ij}(v,v_*)| e^{-2\lambda_1s} dv_* ds \Vert h_{0,j}\Vert_{L^\infty} \nonumber \\
				&\quad +  \sum_{j=1}^N \int_0^t e^{-2\lambda_1(t-s)} \int_{\R^3}  |k^w_{ij}(v,v_*)|  \int_0^s e^{-2\lambda_1(s-s')} |K_{w,j}(\pmb{h})| (s',\hat X(s'),v_*) ds' dv_* ds \nonumber \\
				&\leq e^{-\lambda_1 t} \Vert \pmb{h}_{0}\Vert_{L^\infty} \nonumber \\
				&\quad + \sum_{j,\ell=1}^N \int_0^t e^{-2\lambda_1 (t-s)} \int_{\R^3}  |k^w_{ij}(v,v_*)| \int_0^s e^{-2\lambda_1 (s-s')} \int_{\R^3} |k^w_{j\ell} (v_*,\eta)h_{\ell}(s',\hat X(s'),\eta)| d\eta ds'dv_*ds  \nonumber \\
				&:= II_1 + II_2,
			\end{align}
			where $\hat X(s') = x-(t-s) v- (s-s')v_*$ and we used Lemma \ref{K est}.
			Let us now estimate the term $II_2$. Let $M\geq 1$ be a large constant to be chosen later. We apply Lemma \ref{K est} and split the velocity domain into $\{|v_*|\geq M\}$ and $\{|v_*|\leq M\}$ as follows:		
			\begin{align*}
				II_2 &=  \sum_{j,\ell=1}^N \int_0^t e^{-2\lambda_1 (t-s)} \int_{|v_*| \geq M }  |k^w_{ij}(v,v_*)| \int_0^s e^{-2\lambda_1 (s-s')} \int_{\R^3} |k^w_{j\ell} (v_*,\eta)h_{\ell}(s',\hat X(s'),\eta)| d\eta ds'dv_*ds\\
				&\quad +  \sum_{j,\ell=1}^N \int_0^t e^{-2\lambda_1 (t-s)} \int_{|v_*| \leq M }  |k^w_{ij}(v,v_*)| \int_0^s e^{-2\lambda_1 (s-s')} \int_{\R^3} |k^w_{j\ell} (v_*,\eta)h_{\ell}(s',\hat X(s'),\eta)| d\eta ds'dv_*ds\\
				& \leq C\sup_{0\leq s \leq t} \left[e^{\lambda_1 s}\Vert \pmb{h}(s) \Vert_{L^\infty} \right]\sum_{j=1}^N \int_0^t e^{-2\lambda_1 t}e^{\lambda_1 s} \int_{|v_*|\geq M} \frac{ |k^w_{ij}(v,v_*)| }{1+|v_*|}dv_* ds \\
				&\quad + \sum_{j,\ell=1}^N \int_0^t e^{-2\lambda_1 (t-s)} \int_{|v_*| \leq M }  |k^w_{ij}(v,v_*)| \int_0^s e^{-2\lambda_1 (s-s')} \int_{\R^3} |k^w_{j\ell} (v_*,\eta)h_{\ell}(s',\hat X(s'),\eta)| d\eta ds'dv_*ds\\
				&\leq \frac{C}{M} e^{-\lambda_1 t}\sup_{0\leq s \leq t} \left[e^{\lambda_1 s}\Vert \pmb{h}(s) \Vert_{L^\infty} \right] \\
				& \quad + \sum_{j,\ell=1}^N \int_0^t e^{-2\lambda_1 (t-s)} \int_{|v_*| \leq M }  |k^w_{ij}(v,v_*)| \int_0^s e^{-2\lambda_1 (s-s')} \int_{\R^3} |k^w_{j\ell} (v_*,\eta)h_{\ell}(s',\hat X(s'),\eta)| d\eta ds'dv_*ds.
			\end{align*}
			To further estimate $II_2$, we divide the $\eta$-domain  into $\{|\eta| \geq 2M\}$ and $\{|\eta| \leq 2M\}$:
			\begin{align*}
				&\int_0^t e^{-2\lambda_1 (t-s)} \int_{|v_*| \leq M }  |k^w_{ij}(v,v_*)| \int_0^s e^{-2\lambda_1 (s-s')} \int_{\R^3} |k^w_{j\ell} (v_*,\eta)h_{\ell}(s',\hat X(s'),\eta)| d\eta ds'dv_*ds \\
				&=\int_0^t e^{-2\lambda_1 (t-s)} \int_{|v_*| \leq M }  |k^w_{ij}(v,v_*)| \int_0^s e^{-2\lambda_1 (s-s')} \int_{|\eta| \geq 2M} |k^w_{j\ell} (v_*,\eta)h_{\ell}(s',\hat X(s'),\eta)| d\eta ds'dv_*ds\\
				&\quad + \int_0^t e^{-2\lambda_1 (t-s)} \int_{|v_*| \leq M }  |k^w_{ij}(v,v_*)| \int_0^s e^{-2\lambda_1 (s-s')} \int_{|\eta| \leq 2M} |k^w_{j\ell} (v_*,\eta)h_{\ell}(s',\hat X(s'),\eta)| d\eta ds'dv_*ds\\
				&\leq \sup_{0\leq s \leq t} \left[e^{\lambda_1 s}\Vert \pmb{h}(s) \Vert_{L^\infty} \right] \\
				&\qquad \times \int_0^t e^{-2\lambda_1 (t-s)} \int_{|v_*| \leq M }  |k^w_{ij}(v,v_*)| \int_0^s e^{-2\lambda_1 s}e^{\lambda_1 s'} \int_{|\eta| \geq 2M} |k^w_{j\ell} (v_*,\eta)| e^{\frac{|v_*-\eta|^2}{32}} e^{-\frac{M^2}{32}}d\eta ds' dv_* ds\\
				&\quad +\int_0^t e^{-2\lambda_1 (t-s)} \int_{|v_*| \leq M }  |k^w_{ij}(v,v_*)| \int_{s-\delta}^{s} e^{-2\lambda_1 (s-s')} \int_{|\eta| \leq 2M} |k^w_{j\ell} (v_*,\eta)h_{\ell}(s',\hat X(s'),\eta)| d\eta ds'dv_*ds\\
				&\quad +\int_0^t e^{-2\lambda_1 (t-s)} \int_{|v_*| \leq M }  |k^w_{ij}(v,v_*)| \int_{0}^{s-\delta} e^{-2\lambda_1 (s-s')} \int_{|\eta| \leq 2M} |k^w_{j\ell} (v_*,\eta)h_{\ell}(s',\hat X(s'),\eta)| d\eta ds'dv_*ds\\
				&\leq C\left(\frac{1}{M}+\delta \right)e^{-\lambda_1 t}\sup_{0\leq s \leq t} \left[e^{\lambda_1 s}\Vert \pmb{h}(s) \Vert_{L^\infty} \right]\\
				&\quad + \int_0^t e^{-2\lambda_1(t-s)} \int_0^{s-\delta} e^{-2\lambda_1(s-s')} \left(\int_{|v_*|\leq M, |\eta| \leq 2M }|h_{\ell}(s',\hat X(s'),\eta)|^2 d\eta dv_*\right)^{1/2}ds'ds \\
				&\leq C\left(\frac{1}{M}+\delta \right)e^{-\lambda_1 t}\sup_{0\leq s \leq t} \left[e^{\lambda_1 s}\Vert \pmb{h}(s) \Vert_{L^\infty} \right]  \\
				&\quad + C_{q,M,\delta}\int_0^t e^{-2\lambda_1(t-s)} \int_0^{s-\delta} e^{-2\lambda_1 s}e^{\lambda_1 s'} \sup_{0\leq s' \leq s}\left[e^{\lambda_1 s'} \Vert f_\ell(s') \Vert_{L^2 } \right]ds'ds\\
				&\leq C\left(\frac{1}{M}+\delta \right)e^{-\lambda_1 t}\sup_{0\leq s \leq t} \left[e^{\lambda_1 s}\Vert \pmb{h}(s) \Vert_{L^\infty} \right]+ C_{q,M,\delta} e^{-\lambda_1 t}\sup_{0\leq s \leq t} \left[ e^{\lambda_1 s}\Vert \pmb{f} (s) \Vert_{L^2 }\right] ,
			\end{align*} 
			where we used Lemma \ref{K est} and applied the change of variables $v_* \mapsto y=\hat X(s')$ with $\left| \frac{dy}{dv_*}\right| = (s-s')^3 \geq \delta^3$. Therefore, we can bound the term $II_2$ by
			\begin{align} \label{small II2}
				II_2 &\leq C\left(\frac{1}{M}+\delta \right)e^{-\lambda_1 t}\sup_{0\leq s \leq t} \left[e^{\lambda_1 s}\Vert \pmb{h}(s) \Vert_{L^\infty} \right]+ C_{q,M,\delta} e^{-\lambda_1 t}\sup_{0\leq s \leq t} \left[ e^{\lambda_1 s}\Vert \pmb{f} (s) \Vert_{L^2 }\right].
			\end{align}
			Combining \eqref{small II} and \eqref{small II2} yields the bound for $II$:
			\begin{align} \label{small II final}
				II &\leq  e^{-\lambda_1 t} \Vert \pmb{h}_{0}\Vert_{L^\infty}   
				+C\left(\frac{1}{M}+\delta \right)e^{-\lambda_1 t}\sup_{0\leq s \leq t} \left[e^{\lambda_1 s}\Vert \pmb{h}(s) \Vert_{L^\infty} \right]+ C_{q,M,\delta} e^{-\lambda_1 t}\sup_{0\leq s \leq t} \left[ e^{\lambda_1 s}\Vert \pmb{f} (s) \Vert_{L^2 }\right].
			\end{align}
			Therefore, it follows from \eqref{small split}, \eqref{small I}, and \eqref{small II final} that
			\begin{align*}
				|h_i(t,x,v)| &\leq   e^{-\lambda_1 t} \Vert \pmb{h}_{0}\Vert_{L^\infty}   
				+C\left(\frac{1}{M}+\delta \right)e^{-\lambda_1 t}\sup_{0\leq s \leq t} \left[e^{\lambda_1 s}\Vert \pmb{h}(s) \Vert_{L^\infty} \right]+ C_{q,M,\delta} e^{-\lambda_1 t}\sup_{0\leq s \leq t} \left[ e^{\lambda_1 s}\Vert \pmb{f} (s) \Vert_{L^2 }\right].		
			\end{align*}
			Taking the summation over $i \in \{1,\cdots,N\}$, we obtain the following estimate 
			\begin{align*}
				\Vert \pmb{h}(t) \Vert_{L^\infty} &\leq Ce^{-\lambda_1t} \Vert \pmb{h}_0 \Vert_{L^2 }
				+C\left(\frac{1}{M}+\delta \right)e^{-\lambda_1 t}\sup_{0\leq s \leq t} \left[e^{\lambda_1 s}\Vert \pmb{h}(s) \Vert_{L^\infty} \right]+ C_{q,M,\delta} e^{-\lambda_1 t}\sup_{0\leq s \leq t} \left[ e^{\lambda_1 s}\Vert \pmb{f} (s) \Vert_{L^2 }\right].		
			\end{align*}
			From Theorem \ref{L2 decay} and definition of $\lambda_1$ we have 
			\begin{align*}
				\sup_{0\leq s \leq t} \left[e^{\lambda_1 s}  \Vert \pmb{f} (s) \Vert_{L^2 }\right] \leq C_2 \Vert \pmb{f}_0 \Vert_{L^2 }. 
			\end{align*}
			From a direct computation, we obtain
			\begin{align*}
				\Vert \pmb{f}_0 \Vert_{L^2 }^2 = \sum_{i=1}^N \Vert f_{0,i} \Vert_{L^2 }^2 &= \sum_{i=1}^N \left( \int_{\T^3 \times \R^3} |f_{0,i}(x,v)|^2 dvdx \right)\\
				&=\sum_{i=1}^N \left( \int_{\T^3 \times \R^3} w^{-2}(v)|h_{0,i}(x,v)|^2 dvdx \right)\\
				&\leq \sum_{i=1}^N \Vert h_{0,i} \Vert_{L^\infty}^2 \left( \int_{\T^3 \times \R^3} w^{-2}(v) dvdx\right)\\
				&\leq C_{q} \Vert \pmb{h}_0 \Vert_{L^\infty}^2.
			\end{align*}
			Consequently, if we choose $M$ sufficiently large and $\delta$ sufficiently small such that 
			\begin{align*}
				C\left(\frac{1}{M} +\delta\right) < \frac{1}{2},
			\end{align*}
			then there exists a positive constant $C_3=C_3(q)>0$ such that 
			\begin{align*}
				\Vert \pmb{h}(t) \Vert_{L^\infty} \leq C_3e^{-\lambda_1 t} \Vert \pmb{h}_0 \Vert_{L^\infty}, \quad \forall t\geq 0. 
			\end{align*}
		\end{proof}
		\subsection{The full nonlinear equation in a small amplitude regime } 
		Under a smallness assumption on the initial data, we prove the existence of a global-in-time solution to the full nonlinear Boltzmann equation in the $v$-weighted $L^\infty$ space.
		Moreover, we establish the uniqueness and positivity of the solution.
		These results correspond to Theorem \ref{thm:main2}.
		\begin{proof}[Proof of Theorem \ref{thm:main2}] 
			We set $\pmb{h} = w\pmb{f}$. 
			We first prove the existence of solutions to the multi-species Boltzmann equation. Let us define the iteration sequence starting with $\pmb{h}^0 \equiv \pmb{0}$. We consider the following iteration
			\begin{align*}
				\p_t \pmb{h}^{m+1} + v\cdot \nabla_x \pmb{h}^{m+1} +\pmb{\nu}\pmb{h}^{m+1} - K_{{w}} (\pmb{h}^{m+1}) = w \pmb{\Gamma} (\frac{\pmb{h}^m}{w}),
			\end{align*}
			with $\pmb{h}^{m+1}(0) = \pmb{h}_0$ to show that 
			\begin{align*}
				\sup_{m} \Vert \pmb{h}^m (t) \Vert_{L^\infty} \leq C  e^{-\lambda_1 t } \Vert \pmb{h}_0 \Vert_{L^\infty}.  
			\end{align*}
			For $m=1$, the estimate follows directly from Lemma \ref{small L infty}:
			\begin{align*}
				\Vert \pmb{h}^1 (t) \Vert_{L^\infty} \leq C e^{-\lambda_1 t }  \Vert \pmb{h}_0 \Vert_{L^\infty}.  
			\end{align*}
			For the inductive argument, we assume 
			\begin{align} \label{m assumption}
				\Vert \pmb{h}^m(t) \Vert_{L^\infty}\leq Ce^{-\lambda_1 t }\Vert \pmb{h}_0\Vert_{L^\infty}.
			\end{align}
			Following the idea of \cite{Guo10}, we split $\pmb{h}^{m+1} = \pmb{h}^{m+1}_g + \pmb{h}^{m+1}_{\Gamma}$, where $\pmb{h}_g^{m+1}$ satisfies the homogeneous linear weighted Boltzmann equation 
			\begin{align*}
				&\p_t \pmb{h}_g^{m+1}+ v\cdot \nabla_x \pmb{h}_g^{m+1} + \pmb{\nu} \pmb{h}_g^{m+1} -K_w (\pmb{h}_g^{m+1}) =0,\\
				&\pmb{h}_g^{m+1}(0,x,v) = \pmb{h}_0. 
			\end{align*}
			On the other hand, $\pmb{h}_{\Gamma}^{m+1}$ satisfies the inhomogeneous weighted Boltzmann equation 
			\begin{align*}
				&\p_t \pmb{h}_{\Gamma}^{m+1}+ v\cdot \nabla_x \pmb{h}_{\Gamma}^{m+1} + \pmb{\nu} \pmb{h}_{\Gamma}^{m+1} -K_w (\pmb{h}_{\Gamma}^{m+1}) =w\pmb{\Gamma}(\frac{\pmb{h}^m}{w}),\\
				&\pmb{h}_{\Gamma}^{m+1}(0,x,v) = \pmb{0}. 
			\end{align*} 
			It is directly deduced from Lemma \ref{small L infty} that 
			\begin{align*}
				\Vert \pmb{h}_g^{m+1} \Vert_{L^\infty} \leq Ce^{-\lambda_1 t }\Vert \pmb{h}_0 \Vert_{L^\infty}. 
			\end{align*}
			For $\pmb{h}_{\Gamma}^{m+1}$, fix $i \in \{1,\dots,N\}$. Using Duhamel's principle yields 
			\begin{align*}
				h^{m+1}_{\Gamma,i}(t,x,v) = \int_0^t e^{-\nu_i(v)(t-s)}  \left[K_{w,i} (\pmb{h}^{m+1}_{\Gamma})(s) +w\Gamma_i (\frac{\pmb{h}^m}{w})(s)\right]ds, 
			\end{align*} 
			where $h^{m+1}_{\Gamma,i}$ is the $i$-th component of $\pmb{h}^{m+1}_{\Gamma}$. This leads to 
			\begin{align*}
				\Vert h_{\Gamma,i}^{m+1} (t) \Vert_{L^\infty} &\leq  \int_0^t e^{-\nu_i(v) (t-s)} \left | w \Gamma_i (\frac{\pmb{h}^m}{w})(s)\right|ds \\
				&\quad + \sum_{j=1}^N \int_0^t e^{-\nu_i(v) (t-s)} \int_{\R^3} |k^w_{ij} (v,v_*) h_j^{m+1}(s,x-v(t-s),v_*)| dv_*ds\\
				&:= I+ II
			\end{align*}
			For $I$, we derive the estimate using Corollary \ref{nonlinear est} and the induction hypothesis \eqref{m assumption}:  
			\begin{align*}
				I&\leq C \int_0^t e^{-\nu_i(v)(t-s)} \langle v \rangle^{\gamma} \Vert \pmb{h}^m(s) \Vert_{L^\infty}^2ds \\
				&\leq C \Vert \pmb{h}_0 \Vert_{L^\infty}^2  \int_0^t e^{-\nu_i(v)(t-s)}e^{-2\lambda_1 s} \nu_i(v)  ds\\
				&\leq C e^{-\lambda_1 t} \Vert \pmb{h}_0 \Vert_{L^\infty}^2, 
			\end{align*}
			where we have used the fact $\langle v \rangle ^{\gamma} \sim\nu_i(v)$ from Lemma \ref{CFE} and the choice of $\lambda_1 = \min \{ \frac{\nu_0}{2}, \lambda \}$ in the proof of Lemma \ref{small L infty}. To handle the remaining term $II$, we substitute the mild formulation of $\pmb{h}^{m+1}_{\Gamma}$ into the integral. Then, we get
			\begin{align*}
				II &= \sum_{j=1}^N \int_0^t e^{-\nu_i(v) (t-s)} \int_{\R^3} |k^w_{ij} (v,v_*) h_{\Gamma,j}^{m+1}(s,x-v(t-s),v_*)| dv_*ds\\
				&\leq \sum_{j=1}^N \int_0^t e^{-\nu_i(v) (t-s)} \int_{\R^3} |k^w_{ij} (v,v_*)|\int_0^s e^{-\nu_j(v_*)(s-s')}|K_{w,j}(\pmb{h}^{m+1}_{\Gamma})(s')| ds'dv_*ds \\
				&\quad + \sum_{j=1}^N \int_0^t e^{-\nu_i(v) (t-s)} \int_{\R^3} |k^w_{ij} (v,v_*)|\int_0^s e^{-\nu_j(v_*)(s-s')} \left|w \Gamma_j (\frac{\pmb{h}^m}{w})(s')\right|ds'dv_*ds\\
				&:= II_1 + II_2. 
			\end{align*} 
			For $II_1$, we apply Duhamel's principle once more to $h_{\Gamma,j}^{m+1}$ and then repeat the velocity decomposition used in the proof of Lemma \ref{small L infty}. We fix a large constant $M\geq 1$. The estimates for the large velocity region and the short-time interval are identical to those in Lemma \ref{small L infty}. Hence, we only need to control the compact velocity contribution, which is reduced to the following $L^2$ term:		
			\begin{align*}
				\int_0^t e^{-\nu_i(v) (t-s)} \int_0^{s-\delta} e^{-\nu_{j}(v_*)(s-s')} \left( \int_{|v_*| \leq M, |\eta| \leq 2M} |h_{\Gamma,\ell}^{m+1} (s',\hat X(s'),\eta)|^2 d\eta dv_*\right)^{1/2} ds'ds.
			\end{align*}
			Applying Corollary \ref{nonlinear est} and Theorem \ref{L2 decay} with $\pmb{f}_0=\pmb{0}$ and $\pmb{q} = \pmb{\Gamma}({\pmb{f}^m})$, one obtains 
			\begin{align*}
				&\int_0^t e^{-\nu_i(v) (t-s)} \int_0^{s-\delta} e^{-\nu_j(v_*)(s-s')} \left( \int_{|v_*| \leq M, |\eta| \leq 2M} |h_{\Gamma,\ell}^{m+1} (s',\hat X(s'),\eta)|^2 d\eta dv_*\right)^{1/2} ds'ds\\
				&\leq  C_{q,M} 	\int_0^t e^{-\nu_i(v) (t-s)} \int_0^{s-\delta} e^{-\nu_j(v_*)(s-s')} \left( \int_{|v_*| \leq M, |\eta| \leq 2M} |f_{\Gamma,\ell}^{m+1} (s',\hat X(s'),\eta)|^2 d\eta dv_*\right)^{1/2} ds'ds\\
				&\leq C_{q,M,\delta} \int_0^t e^{-\nu_i(v) (t-s)} \int_0^{s-\delta} e^{-\nu_j(v_*)(s-s')} \Vert f^{m+1}_{\Gamma,\ell} (s') \Vert_{L^2 } ds'ds \\
				&\leq C_{q,M,\delta} \int_0^t e^{-\nu_i(v) (t-s)} \int_0^{s-\delta} e^{-\nu_j(v_*)(s-s')} \left(\int_0^{s'} e^{-2\lambda (s'-\tau)} \Vert \pmb{\Gamma}(\pmb{f}^m)(\tau) \Vert_{L^2 }^2 d\tau\right)^{1/2}ds'ds \\
				&\leq C_{q,M,\delta} \int_0^t e^{-\nu_i(v) (t-s)} \int_0^{s-\delta} e^{-\nu_j(v_*)(s-s')} \left(\int_0^{s'} e^{-2\lambda (s'-\tau)}  \Vert w^{-1} w\pmb{\Gamma}(\pmb{f}^m)(\tau) \Vert_{L^2 }^2 d\tau \right)^{1/2} ds'ds \\
				&\leq C_{q,M,\delta} \int_0^t e^{-\nu_i(v) (t-s)} \int_0^{s-\delta} e^{-\nu_j(v_*)(s-s')} \\
				&\quad \times \left(\int_0^{s'} e^{-2\lambda (s'-\tau)} \Vert \pmb{h}^m (\tau)\Vert_{L^\infty}^4 \left(\int_{\T^3 \times \R^3} w^{-2}(v) \langle v \rangle ^{2\gamma} dvdx\right) d\tau \right)^{1/2}ds'ds\\
				&\leq C_{q,M,\delta}\int_0^t e^{-2\lambda_1 (t-s)} \int_0^{s-\delta} e^{-2\lambda_1(s-s')} \left(\int_0^{s'} e^{-2\lambda_1 (s'-\tau)} e^{-4\lambda_1 \tau}\Vert \pmb{h}_0 \Vert^4_{L^\infty} d\tau \right)^{1/2} ds' ds \\
				&\leq C_{q,M,\delta} e^{-\lambda_1 t}\Vert \pmb{h}_0 \Vert^2_{L^\infty}.  
			\end{align*}
			For the remaining term $II_2$, we use Lemma \ref{K est} and Corollary \ref{nonlinear est} to derive the estimate 
			\begin{align*}
				II_2 &\leq  \sum_{j=1}^N \int_0^t e^{-\nu_i(v) (t-s)} \int_{\R^3} |k^w_{ij} (v,v_*)|\int_0^s e^{-\nu_j(v_*)(s-s')} \left|w \Gamma_j (\frac{\pmb{h}^m}{w})(s')\right|ds'dv_*ds\\
				&\leq C\sum_{j=1}^N \int_0^t e^{-\nu_i(v) (t-s)} \int_{\R^3} |k^w_{ij} (v,v_*)|\int_0^s e^{-\nu_j(v_*)(s-s')} \langle v_* \rangle^{\gamma} \Vert \pmb{h}^m(s') \Vert_{L^\infty}^2ds'dv_*ds\\
				&\leq C\sum_{j=1}^N \int_0^t e^{-2\lambda_1 (t-s)} \int_{\R^3} |k^w_{ij} (v,v_*)|\int_0^s e^{-\frac{\nu_j(v_*)}{2}(s-s')} \langle v_* \rangle^{\gamma} e^{-\lambda_1(s-s')}e^{-2\lambda_1s'} \Vert \pmb{h}_0 \Vert_{L^\infty}^2 ds'dv_*ds\\
				&\leq Ce^{-\lambda_1 t} \Vert \pmb{h}_0\Vert_{L^\infty}^2,
			\end{align*}
			where we have used the fact $\langle v_* \rangle^{\gamma} \sim \nu_j(v_*)$. Therefore, if $\|\pmb{h}_0\|_{L^\infty} \le \kappa$ for some sufficiently small $\kappa>0$, then		\begin{align*}
				\Vert \pmb{h}^{m+1}(t)\Vert_{L^\infty}\leq  \Vert \pmb{h}^{m+1}_g(t)\Vert_{L^\infty}+\Vert \pmb{h}^{m+1}_{\Gamma}(t)\Vert_{L^\infty} \leq C e^{-\lambda_1 t} \Vert \pmb{h}_0 \Vert_{L^\infty}.
			\end{align*}
			By induction argument, we conclude that 
			\begin{align*} 
				\sup_m \Vert \pmb{h}^{m}(t)\Vert_{L^\infty}  \leq C e^{-\lambda_1 t} \Vert \pmb{h}_0 \Vert_{L^\infty}.
			\end{align*}
			We next show that $\{\pmb{h}^m\}_{m\ge0}$ is a Cauchy sequence. In order to prove this, we consider the equation for $\pmb{h}^{m+1} - \pmb{h}^m$: 
			\begin{align*}
				(\p_t + v\cdot \nabla_x + \pmb{\nu}(v) -K_w)[\pmb{h}^{m+1}-\pmb{h}^{m}] = w\left\{\pmb{\Gamma}\left(\frac{\pmb{h}^{m}}{w}\right) -\pmb{\Gamma}\left(\frac{\pmb{h}^{m-1}}{w}\right)  \right\},
			\end{align*}
			with zero initial data. 
			Fix $i \in \{1, \dots, N\}$. For the $i$-th component, the above equation becomes 
			\begin{align*}
				(\p_t + v\cdot \nabla_x + \nu_i(v))[h_i^{m+1} - h_i^m] = K_{w,i} (\pmb{h}^{m+1} - \pmb{h}^m) +w\left(\Gamma_i \left(\frac{\pmb{h}^{m}}{w}\right) -\Gamma_i\left(\frac{\pmb{h}^{m-1}}{w}\right)\right).
			\end{align*}
			Note that 
			\begin{align*}
				\Gamma_i \left(\frac{\pmb{h}^{m}}{w}\right) -\Gamma_i\left(\frac{\pmb{h}^{m-1}}{w}\right) &=  \sum_{j=1}^N \left[\Gamma \left( \frac{h_i^{m}}{w}, \frac{h_j^{m}}{w}\right)-\Gamma \left( \frac{h_i^{m-1}}{w}, \frac{h_j^{m-1}}{w}\right)\right]\\
				&= \sum_{j=1}^N \left[\Gamma \left( \frac{h_i^{m}-h_i^{m-1}}{w}, \frac{h_j^{m}}{w}\right)+\Gamma \left( \frac{h_i^{m-1}}{w}, \frac{h_j^{m}-h_j^{m-1}}{w}\right)\right].
			\end{align*}
			By Corollary \ref{nonlinear est}, one has 
			\begin{align*}
				\left|w\left(\Gamma_i \left(\frac{\pmb{h}^{m}}{w}\right) -\Gamma_i\left(\frac{\pmb{h}^{m-1}}{w}\right)\right)\right| &\leq C\langle v \rangle^{\gamma} 
				\left(\sum_{j=1}^N  \Vert h_j^m \Vert_{L^\infty} \Vert h_i^m - h_i^{m-1} \Vert_{L^\infty} + \Vert h_i^{m-1} \Vert_{L^\infty}\sum_{j=1}^N \Vert h_j^{m} -h_j^{m-1} \Vert_{L^\infty} \right)\\
				&\leq C\langle v \rangle^{\gamma} \left(\Vert \pmb{h}^{m-1} \Vert_{L^\infty}+\Vert \pmb{h}^m \Vert_{L^\infty} \right) \Vert \pmb{h}^m - \pmb{h}^{m-1} \Vert_{L^\infty}. 
			\end{align*}
			Combining the method used to treat the ${h}^{m+1}_{\Gamma,i}$ term with the above estimate, we derive the following estimate
			\begin{align*}
				\Vert (h^{m+1}_i - h^{m}_i)(t) \Vert_{L^\infty} \leq C \sup_{0\leq s \leq t}  \left[e^{\lambda_1 s } (\Vert \pmb{h}^{m-1}(s) \Vert_{L^\infty}+\Vert \pmb{h}^{m}(s) \Vert_{L^\infty})\right] \sup_{0\leq s \leq t} \Vert (\pmb{h}^m - \pmb{h}^{m-1}) (s)\Vert_{L^\infty}.
			\end{align*}
			Summing over $i$, we obtain 
			\begin{align*}
				\sup_{0\leq s \leq t}\Vert (\pmb{h}^{m+1} - \pmb{h}^{m})(s) \Vert_{L^\infty} &\leq C \sup_{0\leq s \leq t}  \left[e^{\lambda_1 s } (\Vert \pmb{h}^{m-1}(s) \Vert_{L^\infty}+\Vert \pmb{h}^{m}(s) \Vert_{L^\infty})\right] \sup_{0\leq s \leq t} \Vert (\pmb{h}^m - \pmb{h}^{m-1}) (s)\Vert_{L^\infty}\\
				&\leq 2C\Vert \pmb{h}_0 \Vert_{L^\infty} \sup_{0\leq s \leq t} \Vert (\pmb{h}^m - \pmb{h}^{m-1}) (s)\Vert_{L^\infty}.
			\end{align*}
			Hence, by taking $\kappa$ sufficiently small such that 
			\begin{align*}
				2C \Vert \pmb{h}_0\Vert_{L^\infty}<\frac{1}{2},
			\end{align*}
			we deduce that $\pmb{h}^m$ is a Cauchy sequence and the limit is a solution to the full Boltzmann equation. \\ 
			\indent We now prove uniqueness. Suppose that $\tilde{\pmb{h}}$ is another solution to the full Boltzmann equation with the same initial data $\pmb{h}_0$, satisfying the same a priori bound 
			\begin{align*}
				\sup_{0\leq t \leq T_0} \Vert \pmb{\tilde{h}}(t) \Vert_{L^{\infty} } \leq \kappa.
			\end{align*}
			Then, the difference $\pmb{h}-\pmb{\tilde{h}}$ satisfies the equation 
			\begin{align*}
				(\p_t + v\cdot \nabla_x + \pmb{\nu}(v) -K_w)[\pmb{h}-\pmb{\tilde{h}}] = w\left\{\pmb{\Gamma}\left(\frac{\pmb{h}}{w}\right) -\pmb{\Gamma}\left(\frac{\pmb{\tilde{h}}}{w}\right)\right\},
			\end{align*} 
			with zero initial data. 
			Using the bilinear estimate for $\pmb{\Gamma}$ as in the proof of existence, we obtain for $0\leq t \leq T_0$, 
			\begin{align*}
				\Vert (\pmb{h} - \pmb{\tilde{h}})(t)\Vert_{L^\infty} \leq C_{T_0} \sup_{0\leq t \leq T_0} \left( \Vert \pmb{h}(t)\Vert_{L^\infty} + \Vert \pmb{\tilde{h}}(t) \Vert_{L^\infty} \right)  \sup_{0\leq t \leq T_0} \Vert (\pmb{h}-\pmb{\tilde{h}})(t) \Vert_{L^\infty}. 
			\end{align*}
			If $\kappa>0$ is sufficiently small such that 
			\begin{align*}
				C_{T_0} \sup_{0\leq t \leq T_0} \left( \Vert \pmb{h}(t)\Vert_{L^\infty} + \Vert \pmb{\tilde{h}}(t) \Vert_{L^\infty} \right) <1,
			\end{align*}
			then 
			\begin{align*}
				\sup_{0\leq t \leq T_0} \Vert (\pmb{h}-\pmb{\tilde{h}})(t) \Vert_{L^\infty} = 0. 
			\end{align*}
			Finally, it remains to show that $F_i = \mu_i + \sqrt{\mu_i} f_i \geq 0$ for all $i \in \{1,\cdots, N\}$. To prove positivity of the solution, we consider another equation 
			\begin{align*}
				\p_t F_i^{m+1} + v\cdot \nabla_x F_i^{m+1} +\nu_i(\pmb{f}^m) F_i^{m+1} = \sum_{j=1}^N Q_{ij}^+ (F_i^m,F_j^m),
			\end{align*}   
			with $F_i^{m+1}(0,x,v) = F_{0,i} =\mu_i +\sqrt{\mu_i} f_{0,i}$, starting from $F_i^0 \equiv \mu_i(v)$. Here, we define $\nu_i$ as 
			\begin{align*}
				\nu_i (\pmb{f}^m) = \sum_{j=1}^n \int_{\R^3\times\S^2} B_{ij}(v-v_*,\sigma) F_j^m (v_*) d\sigma dv_*. 
			\end{align*}
			Applying Duhamel's principle gives 
			\begin{align*}
				F_i^{m+1} (t,x,v) &= \exp \left[-\int_0^t \nu_i (\pmb{f}^m) (s,x-v(t-s),v) ds\right] F_{0,i}(x-tv,v)\\
				&\quad + \sum_{j=1}^N\int_0^t \exp \left[-\int_s^t \nu_i(\pmb{f}^m)(s',x-(t-s')v, v)ds'\right] Q_{ij}^+ (F_i^m,F_j^m) (s,x-v(t-s),v)ds. 
			\end{align*}
			From induction argument, the right-hand side terms are positive by $F_i^m \geq 0$ for all $i \in \{1,\cdots,N\}$. By taking the limit, we have the positivity of the solution to the full Boltzmann equation. 
		\end{proof}
		
		\section{Large-amplitude perturbation problem}\label{sec:LAPP}
		In this section, we develop the analytical tools required for the large-amplitude regime, where the smallness assumption $\|w\pmb{f}_0\|_{L^\infty} \le \kappa$ imposed in Section~\ref{sec:small_regime} is no longer available. We first establish a uniform positive lower bound for the nonlinear collision frequency $\mathcal{R}_i(\pmb{f})$ (Proposition~\ref{Rf est}). Building on this lower bound, we then derive a weighted $L^\infty$ Gr\"onwall-type estimate which leads the solution into the small-amplitude regime.
		
		To derive the $L^\infty$ estimate, the estimate for the gain term in Lemma \ref{gain est_new} plays a crucial role in using the smallness of the initial relative entropy. Thus, we consider the mild formulation of the following equation: 
		\begin{align} \label{multi pert}
			\p_t f_i + v\cdot \nabla_x f_i +\mathcal{R}_i(\pmb{f}) f_i = K_i (\pmb{f}) + \sum_{j=1}^N \Gamma^+(f_i,f_j),
		\end{align}
		where 
		\begin{align*}
			\mathcal{R}_i(\pmb{f}) = \sum_{j=1}^N \int_{\R^3 \times \S^2} B_{ij}(v-v_*,\sigma) [\mu_j(v_*)+\sqrt{\mu_j(v_*)} f_j(v_*)] d\sigma dv_*.
		\end{align*}
\subsection{$\mathcal{R}_i(\pmb{f})$ estimate}
A crucial step in the global $L^\infty$ estimate is to ensure that the nonlinear collision frequency $\mathcal{R}_i(\pmb{f})$ provides sufficient damping. The loss-term contribution to $\mathcal{R}_i(\pmb{f})$, 
\begin{align*}
\sum_j \int_{\mathbb{R}^3 \times \mathbb{S}^2} B_{ij}(v-v_*,\sigma)\sqrt{\mu_j(v_*)} f_j(v_*)\, d\sigma dv_*,
\end{align*} 
may weaken the damping effect of the collision frequency $\nu_i(v)$. Therefore, the main difficulty is to verify whether the dissipative structure induced by $\nu_i(v)$ remains dominant under the smallness of the initial relative entropy. Our argument tracks the propagation of this smallness through both the linear operator $K_{w,j}$ and the full nonlinear operator $w\Gamma_j$, by combining the relative entropy estimate of Lemma \ref{RE} with the pointwise gain estimate of Lemma \ref{gain est_new}.
\begin{proof}[Proof of Proposition~\ref{Rf est}]
	Recall that $\mathcal{R}_i(\pmb{f})$ is defined by
	\begin{align*}
		\mathcal{R}_i(\pmb{f}) = \sum_{j=1}^N \mathcal{R}_{ij}(\pmb{f}):= \sum_{j=1}^N \int_{\R^3 \times \S^2} B_{ij}(v-v_*,\sigma) [\mu_j(v_*)+\sqrt{\mu_j(v_*)}f_j(v_*)] d\sigma dv_* .
	\end{align*}
	Without loss of generality, we focus on a fixed species $j$, since the contributions from other species can be treated similarly. The component $\mathcal{R}_{ij}(\pmb{f})$ can be decomposed as 
	\begin{align*}
		\mathcal{R}_{ij}(\pmb{f}) &=\int_{\R^3 \times \S^2} B_{ij}(v-v_*,\sigma) [\mu_j(v_*)+\sqrt{\mu_j(v_*)}f_j(v_*)] d\sigma dv_* \cr
		&= \nu_{ij}(v)+\int_{\R^3 \times \S^2} B_{ij}(v-v_*,\sigma)\sqrt{\mu_j(v_*)}f_j(v_*) d\sigma dv_*,
	\end{align*}
	where $\nu_{ij}(v)$ is the linear collision frequency defined in \eqref{nu_def}.
	By Lemma \ref{CFE}, the collision frequency satisfies $\nu_{ij}(v) \simeq (1+|v|)^\gamma$, and  we use the assumption \eqref{kernel} on the collision kernel:
	\begin{align*}
		B_{ij}(v-v_*,\sigma)\leq C^\Phi_{ij}|v-v_*|^{\gamma}b_{ij}(\cos\theta) \leq C(1+|v|)^\gamma(1+|v_*|)^\gamma\leq C_{\nu}\nu_{ij}(v)\nu_{ij}(v_*).
	\end{align*}
	If we denote $h_j(t,x,v)=w(v)f_j(t,x,v)$, then we obtain a lower bound for $\mathcal{R}_{ij}(\pmb{f})$:
	\begin{align*}
		\mathcal{R}_{ij}(\pmb{f}) &\geq \nu_{ij}(v)\left[1-C_{\nu} \int_{\R^3}\nu_{ij}(v_*)\sqrt{\mu_j(v_*)}|f_j(v_*)|  dv_*\right]\cr
		& \geq \nu_{ij}(v)\left[1-C_{\nu} \int_{\R^3}e^{-\frac{m_j|v_*|^2}{8}}|h_j(v_*)|  dv_*\right].
	\end{align*}
	Hence, our goal is to find the transition time $\tilde{t} > 0$ such that for all $t \geq \tilde{t}$, the following estimate holds:
	\begin{align} \label{R}
	  	\int_{\mathbb{R}^3} e^{-\frac{m_j|v|^2}{8}} |{h}_j(t,x,v)| dv \leq 		\frac{1}{2C_{\nu}}.
	\end{align}
	To verify the estimate \eqref{R}, the governing equation for the weighted perturbation $h_j(t,x,v)$ is given by 
	\begin{align*} 
		\partial_t h_j + v \cdot \nabla_x h_j + \nu_j(v) h_j = K_{w,j}(\pmb{h}) + w \Gamma_{j}(\pmb{f}),
	\end{align*}
where the collision frequency $\nu_j$ is defined in \eqref{nu_def}. The weighted linear operator $K_{w,j}(\pmb{h})$ and the nonlinear term $w\Gamma_{j}(\pmb{f})$ are defined as:
\begin{align*}
    K_{w,j}(\pmb{h}) &=\sum_{k=1}^N \int_{\mathbb{R}^3} k^w_{jk}(v, v_*) h_k(v_*) dv_*:= \sum_{k=1}^N\int_{\R^3} k_{jk}(v,v_*) \frac{w(v)}{w(v_*)} h_k(v_*) dv_*, \\
    w\Gamma_{j}(\pmb{f}) &= w(v) \sum_{k=1}^N \int_{\R^3 \times \S^2} B_{jk}(v-v_*, \sigma) \sqrt{\mu_k(v_*)}\left[ f_j(v') f_k(v_*')-f_j(v) f_k(v_*) \right]d\sigma dv_*.
\end{align*}
		Along the trajectory $X(s):= x-(t-s)v$, we can express the mild form of $h_j$ by using Duhamel's principle		\begin{align*}
			h_j(t, x, v) &= e^{-\nu_{j}(v) t} h_{0,j}(x - tv, v) 
			 +\int_{0}^{t} e^{-\nu_{j}(v) (t - s)}[K_{w,j}(\pmb{h})+w\Gamma_{j}(\pmb{f})](s,X(s),v)  ds.
		\end{align*}
		Using the mild formulation, we decompose the left-hand side of \eqref{R} as follows: 		
		\begin{align}\label{h intv}
				&\int_{\R^3}e^{-\frac{m_j|v|^2}{8}}|h_j(t,x,v)|dv \cr
				& \quad\leq \int_{\R^3}e^{-\frac{m_j|v|^2}{8}}e^{-\nu_{j}(v) t}|h_{0,j}(x - tv, v)|dv\cr
				& \qquad+\int_{t-\delta}^{t} \int_{\R^3} e^{-\frac{m_j|v|^2}{8}} e^{-\nu_{j}(v) (t - s)}  |[K_{w,j}(\pmb{h})+w(v)\Gamma_{j}(\pmb{f})]|(s,X(s),v) dv ds\cr
				& \qquad+\int_{0}^{t-\delta}  \int_{\R^3}e^{-\frac{m_j|v|^2}{8}}e^{-\nu_{j}(v) (t - s)}   |[K_{w,j}(\pmb{h})+w(v)\Gamma_{j}(\pmb{f})]|(s,X(s),v) dv ds,
		\end{align}
		where $0< \delta \leq 1$ is a small positive constant to be determined later. To establish the desired result, we estimate each term on the right-hand side of \eqref{h intv}. We begin with the first term, which represents the contribution from the initial data. Using the uniform lower bound of the collision frequency $\nu_j(v) \geq \nu_0 > 0$ from Lemma \ref{CFE}, a direct calculation yields
		\begin{align}\label{rf1}
			\int_{\R^3}e^{-\frac{m_j|v|^2}{8}}e^{-\nu_{j}(v) t}|h_{0,k}(x - tv, v)|dv\leq Ce^{-\nu_0t}\|h_{0,j}\|_{L^\infty}.
		\end{align}
		By Lemma \ref{K est}, we have
\[
\int_{\mathbb{R}^3} |k^w_{jk}(v,v_*)| \, dv_* \le \frac{C}{1+|v|}. 
\]
Hence,
\begin{align*}
|K_{w,j}(\pmb{h})(s,X(s),v)| &\le \sum_{k=1}^N \int_{\mathbb{R}^3} |k^w_{jk}(v,v_*)| |h_k(s,X(s),v_*)| \, dv_* \\
&\le \sum_{k=1}^N \left(\int_{\mathbb{R}^3} |k^w_{jk}(v,v_*)| \, dv_* \right)\|h_k(s)\|_{L^\infty} \\
&\le C \|\pmb{h}(s)\|_{L^\infty}.
\end{align*}
Combining this with the nonlinear estimate in Corollary \ref{nonlinear est}, we obtain
\begin{align*}
    e^{-\nu_{j}(v) (t - s)} |[K_{w,j}(\pmb{h})+w\Gamma_{j}(\pmb{f})]|(s,X(s),v)\leq C e^{-\nu_0(t-s)} \left( \|\pmb{h}(s)\|_{L^\infty} + \|\pmb{h}(s)\|_{L^\infty}^2 \right).
\end{align*}
		Thus, we have 
		\begin{align}\label{rf2}
					&\int_{t-\delta}^{t}  \int_{\R^3} e^{-\frac{m_j|v|^2}{8}}e^{-\nu_{j}(v) (t - s)} |[K_{w,j}(\pmb{h})+w\Gamma_{j}(\pmb{f})]|(s,X(s),v)dv ds\cr
				&\leq  C\sup_{0\leq s\leq t} \left\{\|\pmb{h}(s)\|_{L^\infty}\left(1+ \|\pmb{h}(s)\|_{L^\infty}\right)\right\} \int_{t-\delta}^{t}e^{-\nu_0 (t - s)}ds\int_{\R^3}e^{-\frac{m_j|v|^2}{8}}dv \cr
				&\leq  C\delta\sup_{0\leq s\leq t} \left\{\|\pmb{h}(s)\|_{L^\infty}\left(1+ \|\pmb{h}(s)\|_{L^\infty}\right)\right\} .
		\end{align}
		Next, we estimate the remaining term
		\begin{align*}
			\int_{0}^{t-\delta} \int_{\R^3}e^{-\frac{m_j|v|^2}{8}}e^{-\nu_{j}(v) (t - s)}   |[K_{w,j}(\pmb{h})+w\Gamma_{j}(\pmb{f})]|(s,X(s),v) dv ds.
		\end{align*}
	We first focus on the term involving the linear operator $K_{w,j}$. Using the kernel representation $k^w_{jk}$, we can further bound 
	\begin{align*}
  &\int_{0}^{t-\delta}\int_{\R^3}e^{-\frac{m_j|v|^2}{8}} e^{-\nu_{j}(v) (t - s)} |K_{w,j}(\pmb{h})(s,X(s),v)| dvds \\
  &\leq \sum_{k=1}^{N} \int_{0}^{t-\delta}\int_{\R^3}e^{-\frac{m_j|v|^2}{8}} e^{-\nu_{j}(v) (t - s)}\int_{\mathbb{R}^3} |k^w_{jk}(v, v_*) h_k(s, X(s), v_*)| dv_*dvds.
\end{align*}
To handle the integration over the velocity variables $(v, v_*)$, we fix a sufficiently large number $L\geq 1$ and decompose the velocity space into three disjoint subsets $\Omega_1, \Omega_2,$ and $\Omega_3$ defined as follows:
\begin{align*}
    \Omega_1 &:= \{ (v, v_*) \in \mathbb{R}^3 \times \mathbb{R}^3 : |v| \geq L \}, \\
    \Omega_2 &:= \{ (v, v_*) \in \mathbb{R}^3 \times \mathbb{R}^3 : |v| \leq L, \, |v_*| \geq 2L \}, \\
    \Omega_3 &:= \{ (v, v_*) \in \mathbb{R}^3 \times \mathbb{R}^3 : |v| \leq L, \, |v_*| \leq 2L \}.
\end{align*}
We first estimate the contribution from $\Omega_1$. By Lemma \ref{K est}, we have
\begin{align} \label{K(1)}
        &\sum_{k=1}^{N} \int_{0}^{t-\delta} \int_{|v| \geq L} e^{-\frac{m_j|v|^2}{8}} e^{-\nu_{j}(v) (t - s)}  \int_{\mathbb{R}^3} |k^w_{jk}(v,v_*)| |h_k(s, X(s), v_*)| dv_* dv ds \cr
        &\leq \sup_{0 \leq s \leq t} \|\pmb{h}(s)\|_{L^\infty} \sum_{k=1}^{N}\int_0^{t-\delta}  \int_{|v| \geq L} e^{-\frac{m_j|v|^2}{8}} e^{-\nu_{j}(v) (t - s)} \int_{\mathbb{R}^3} |k^w_{jk}(v,v_*)| dv_*  dvds  \cr
        &\leq \frac{C}{L}  \sup_{0 \leq s \leq t} \|\pmb{h}(s)\|_{L^\infty}.
\end{align}
For the region $\Omega_2$, we observe that $|v-v_*| \geq L$. Using Lemma \ref{CFE} and \ref{K est}, we obtain
\begin{align} \label{K(2)}
     &\sum_{k=1}^{N} \int_{0}^{t-\delta}  \int_{\Omega_2} e^{-\frac{m_j|v|^2}{8}} e^{-\nu_{j}(v) (t - s)} |k^w_{jk}(v,v_*)| |h_k(s, X(s), v_*)| dv_* dv ds\cr
    &\leq C \sup_{0 \leq s \leq t} \|\pmb{h}(s)\|_{L^\infty} \sum_{k=1}^{N} \int_{\Omega_2} e^{-\frac{m_j|v|^2}{8}} \frac{e^{\frac{1}{64} |v-v_*|^2}}{e^{\frac{L^2}{64} }} |k^w_{jk}(v,v_*)| dv_* dv \cr
    &\leq \frac{C}{L} \sup_{0 \leq s \leq t} \|\pmb{h}(s)\|_{L^\infty} \sum_{k=1}^{N} \int_{|v| \leq L} e^{-\frac{m_j|v|^2}{8}} dv \cr
    &\leq \frac{C}{L} \sup_{0 \leq s \leq t} \|\pmb{h}(s)\|_{L^\infty}.
\end{align}
For $\Omega_3$, we use the $L^2_v$-boundedness of the kernel $k^w_{jk}$: 
\begin{align} \label{k_square}
\int_{\R^3} \left|k^w_{jk}(v,v_*)\right|^2 dv_* \leq C.
\end{align}
This property can be established by combining Lemma \ref{K est} with the argument used in the single-species case (see \cite{Guo10}). Since the extension to the multi-species setting follows analogously, we omit the details. Then, we have
		\begin{align}\label{k3}
				&\sum_{k=1}^{N} \int_{0}^{t-\delta}  \int_{\Omega_3} e^{-\frac{m_j|v|^2}{8}}e^{-\nu_{j}(v) (t - s)}  |k^w_{jk}(v,v_*)| |h_k(s, X(s), v_*)| dv_* dvds \cr
				& \leq \sum_{k=1}^{N}\int_{0}^{t-\delta}e^{-\nu_0(t - s)}\left(\int_{ |v|\leq L,\ |v_*|\leq 2L}e^{-\frac{m_j|v|^2}{4}}  \left|k^w_{jk}(v,v_*)\right|^2dv_*dv\right)^\frac{1}{2}\cr
				&\quad\times\left(	\int_{ |v|\leq L,\ |v_*|\leq 2L}e^{-\frac{m_j|v|^2}{4}}\left|h_k(s, x - (t - s)v, v_*) \right|^2dv_*dv\right)^\frac{1}{2}ds\cr
				& \leq C	\sum_{k=1}^N\int_{0}^{t-\delta}e^{-\nu_0(t - s)}\left(\int_{ |v|\leq L,\ |v_*|\leq 2L}e^{-\frac{m_j|v|^2}{4}}\left|h_k(s, x - (t - s)v, v_*) \right|^2dvdv_*\right)^\frac{1}{2}ds.
		\end{align}
		To estimate the last term, we apply the change of variables $y:=x-(t-s)v$. Then,
		\begin{align}\label{cv}
			\begin{split} 
			&\int_{ |v|\leq L,\ |v_*|\leq 2L}e^{-\frac{m_j|v|^2}{4}}\left|h_k(s, x - (t - s)v, v_*) \right|^2dvdv_* \\
			&\leq C_{L}\frac{1+(t-s)^3}{(t-s)^3}\int_{\mathbb{T}^3}\int_{ |v_*|\leq 2L}\left|h_k(s,y, v_*) \right|^2dv_*dy.
			\end{split}
		\end{align}
		Thus, we can estimate \eqref{k3} as follows:
		\begin{align}\label{111}
				&\sum_{k=1}^{N} \int_{0}^{t-\delta} \int_{\Omega_3} e^{-\frac{m_j|v|^2}{8}}e^{-\nu_{j}(v) (t - s)}  |k^w_{jk}(v,v_*)| |h_k(s, X(s), v_*)| dv_* dvds \cr				
				&\leq C_{L}\sum_{k=1}^N\int_{0}^{t-\delta}e^{-\nu_0 (t-s)} \left(\frac{1+(t-s)^3}{(t-s)^3}\right)^{1/2}\left(\int_{\T^3}\int_{|v_*|\leq 2L} \left|h_k(s, y, v_*) \right|^2dv_*dy\right)^{\frac{1}{2}}ds\cr
				&\leq C_{L,\delta} \sum_{k=1}^N\int_{0}^{t-\delta}e^{-\nu_0 (t - s)}\left(\int_{\T^3}\int_{ |v_*|\leq 2L}\left|f_k(s, y, v_*) \right|^2dv_*dy\right)^\frac{1}{2}ds.
		\end{align} 
		It follows from Lemma \ref{RE} that 
		\begin{align}\label{222}
				&\int_{\T^3}\int_{ |v_*|\leq 2L}\left|f_k(s, y, v_*) \right|^2 dv_* dy \cr 
				&=\int_{\T^3}\int_{ |v_*|\leq 2L}\left|f_k(s, y, v_*) \right|^2\left(\chi_{\left\{\left|f_k\right|\leq\sqrt{\mu_k}\right\}}+\chi_{\left\{\left|f_k\right|\geq\sqrt{\mu_k(v_*)}\right\}}\right)dv_*dy\cr
				&\leq \mathcal{E}(\pmb{F}_0)+\int_{\T^3}\int_{ |v_*|\leq 2L}\left|f_k(s, y, v_*) \right|^2\chi_{\left\{\left|f_k\right|\geq\sqrt{\mu_k}\right\}}dv_*dy\cr
				&\leq\mathcal{E}(\pmb{F}_0)+\sup_{0\leq s\leq t}\|h_k(s)\|_{L^\infty}\int_{\T^3}\int_{|v_*|\leq 2L}\mu_k(v_*)^{-1/2}\sqrt{\mu_k(v_*)}\left|f_k(s,y, v_*) \right|\chi_{\left\{\left|f_k\right|\geq\sqrt{\mu_k}\right\}}dv_*dy\cr
				&\leq C_{L}\mathcal{E}(\pmb{F}_0)(1+\sup_{0\leq s\leq t}\|\pmb{h}(s)\|_{L^\infty}),
		\end{align}
		where we used $\mu_k(v_*)^{-1/2} \leq C_L$ for $|v_*| \leq 2L$. 
		Thus, combining \eqref{K(1)}, \eqref{K(2)}, \eqref{111}, and \eqref{222}, we estimate the contribution of the operator $K$ appearing in the last line of \eqref{h intv} as follows:
		\begin{align}\label{Klast}
				 &\int_{0}^{t-\delta}\int_{\R^3}e^{-\frac{m_j|v|^2}{8}} e^{-\nu_{j}(v) (t - s)} |K_{w,j}(\pmb{h})(s,X(s),v)| dvds \cr					& \leq  \frac{C}{L}\sup_{0\leq s\leq t}\|\pmb{h}(s)\|_{L^\infty} +C_{L,\delta}\sqrt{\mathcal{E}(\pmb{F}_0)}\left( 1+\sup_{0\leq s\leq t}\|\pmb{h}(s)\|_{L^\infty}\right)^\frac{1}{2}.
		\end{align}
		To estimate the nonlinear term in the last line of \eqref{h intv}
		\begin{align*}
			&\int_{0}^{t-\delta} \int_{\R^3} e^{-\frac{m_j|v|^2}{8}} e^{-\nu_{j}(v) (t - s)} |w\Gamma_{j}(\pmb{f})(s, X(s), v)| dv ds,
		\end{align*}
		we first consider the region $\{|v| \geq L\}$. Using Corollary \ref{nonlinear est}, we have
\begin{align} \label{nonlinear(1)} 
    &\int_{0}^{t-\delta} \int_{|v| \geq L} e^{-\frac{m_j|v|^2}{8}} e^{-\nu_{j}(v) (t - s)} |w\Gamma_{j}(\pmb{f})(s, X(s), v)| dv ds \cr
    &\leq C \int_{0}^{t-\delta} e^{-\nu_0 (t-s)}  \int_{|v| \geq L} e^{-\frac{m_j|v|^2}{8}} \langle v \rangle^\gamma \| \pmb{h}(s) \|_{L^\infty}^2 dvds \cr
    &\leq \frac{C}{L} \sup_{0 \leq s \leq t} \| \pmb{h}(s) \|_{L^\infty}^2.
\end{align}
We now focus on the compact domain $\{|v| \leq L\}$. In this region, we decompose the nonlinear operator into its gain and loss components:
\begin{align*}
    &\int_{0}^{t-\delta} \int_{|v| \leq L} e^{-\frac{m_j|v|^2}{8}} e^{-\nu_{j}(v) (t - s)} |w\Gamma_{j}(\pmb{f})(s, X(s), v)| dv ds \\
    & \leq \int_{0}^{t-\delta} \int_{|v| \leq L} e^{-\frac{m_j|v|^2}{8}} e^{-\nu_{j}(v) (t - s)} \left[ |w\Gamma^+_{j}(\pmb{f})| + |w\Gamma^-_{j}(\pmb{f})| \right](s, X(s), v) dv ds.
\end{align*}
		For the gain term $w\Gamma^+_{j}$, we use Lemma \ref{gain est_new} to obtain the following estimate
		\begin{align} \label{gamma gain_split}
		&\int_{0}^{t-\delta} \int_{|v| \leq L} e^{-\frac{m_j|v|^2}{8}} e^{-\nu_{j}(v) (t - s)} |w\Gamma^+_{j}(\pmb{f})(s, X(s), v)| dv ds\cr
			&=\int_{0}^{t-\delta}\int_{|v|\leq L}e^{-\frac{m_j|v|^2}{8}} e^{-\nu_{j}(v) (t - s)} \left|\sum_{k=1}^{N} w(v) \Gamma^+(f_j, f_k)(s, x - (t - s)v, v)\right|dvds\cr
			&\leq C_{*,1}\int_{0}^{t-\delta}\int_{ |v|\leq L}e^{-\nu_j(v)(t-s)}e^{-\frac{m_j|v|^2}{8}}\left[\frac{\|h_j(s)\|_{L^{\infty}}}{(1+|v|)^{\frac{6}{5}-\gamma}}\sum_{k=1}^N\left\{\int_{\R^3}\left(1+|\eta|\right)^{-5q+10}|h_k(s,x - (t - s)v,\eta)|^5d\eta\right\}^\frac{1}{5} \right.\cr
			&\left.\qquad +\sum_{k=1}^N\frac{\|h_k(s)\|_{L^{\infty}}}{1+|v|}\left\{\int_{\R^3}\left(1+|\eta|\right)^{-2q+4}|h_j(s,x - (t - s)v,\eta)|^2d\eta\right\}^\frac{1}{2}\right]dvds\cr
			&:= R_1 + R_2.
		\end{align}
		The prior part $R_1$ can be estimated as follows:
		\begin{align} \label{R1_1} 
			R_1 &\leq C_{*,1}\sup_{0\leq s\leq t}\|h_j(s)\|_{L^{\infty}} \cr
			&\quad \times \int_{0}^{t-\delta}e^{-\nu_0(t-s)}\int_{|v|\leq L}e^{-\frac{m_j|v|^2}{8}}\sum_{k=1}^N\left\{\int_{\R^3}\left(1+|\eta|\right)^{-5q+10}|h_k(s,x - (t - s)v,\eta)|^5d\eta\right\}^\frac{1}{5} dvds.
		\end{align}
		For the above integral, we split the integration into two parts: $\{|\eta|\leq 2L\}$ and $\{|\eta|\geq2L\}$. For $\{|\eta|\geq2L\}$, we have 
		\begin{align}\label{R1_2}   
			&\int_{|v|\leq L}e^{-\frac{m_j|v|^2}{8}}\sum_{k=1}^N\left\{\int_{|\eta|\geq2L}\left(1+|\eta|\right)^{-5q+10}|h_k(s,x - (t - s)v,\eta)|^5d\eta\right\}^\frac{1}{5} dv\cr
			&\leq \sum_{k=1}^N\sup_{0\leq s\leq t}\|h_k(s)\|_{L^\infty}\left\{\int_{|\eta|\geq2L}\left(1+|\eta|\right)^{-5q+10}d\eta\right\}^\frac{1}{5}\cr
			&\leq \frac{C}{L^{1/5}} \sup_{0\leq s \leq t} \Vert \pmb{h}(s) \Vert \left(\int_{|\eta|\geq2L}(1+|\eta|)^{-5q +11} d\eta\right)^\frac{1}{5} \cr
			&\leq \frac{C_q}{L^{1/5}} \sup_{0\leq s \leq t} \Vert \pmb{h}(s) \Vert, 
		\end{align}
		where we used the condition $q>4$ in \eqref{v_weight}. 
		For $\{|\eta|\leq2L\}$, it follows from H\"{o}lder's inequality that 
		\begin{align}\label{333}
			\begin{split}
				&\int_{|v|\leq L}e^{-\frac{m_j|v|^2}{8}}\sum_{k=1}^N\left\{\int_{|\eta|\leq2L}\left(1+|\eta|\right)^{-5q+10}|h_k(s,x - (t - s)v,\eta)|^5d\eta\right\}^\frac{1}{5} dv\cr
				&\leq \sum_{k=1}^N\left(\int_{|v|\leq L}e^{-\frac{5m_j|v|^2}{32}}dv\right)^\frac{4}{5}\left(\int_{|v|\leq L,|\eta|\leq 2L}\left(1+|\eta|\right)^{-5q+10}|h_k(s,x - (t - s)v,\eta)|^5d\eta dv\right)^\frac{1}{5}\cr
				&  \leq C\sum_{k=1}^N \sup_{0\leq s \leq t} \Vert h_k(s) \Vert_{L^\infty}^{3/5}\left(\int_{|v|\leq L,|\eta|\leq 2L}\left(1+|\eta|\right)^{-5q+10}|h_k(s,x - (t - s)v,\eta)|^2 dvd\eta\right)^\frac{1}{5}\cr
				&= C \sum_{k=1}^N\sup_{0\leq s \leq t}\|h_k(s)\|^\frac{3}{5}_{L^\infty}\left(\int_{|v|\leq L,|\eta|\leq 2L}(1+|\eta|)^{-3q+10} |f_k(s,x - (t - s)v,\eta)|^2dvd\eta\right)^\frac{1}{5} \cr
				&  \leq C_L\sum_{k=1}^N\sup_{0\leq s \leq t}\|h_k(s)\|^\frac{3}{5}_{L^\infty}\left(\int_{|v|\leq L,|\eta|\leq 2L}|f_k(s,x - (t - s)v,\eta)|^2dvd\eta\right)^\frac{1}{5}.
			\end{split}
		\end{align}
		Using the change of variables $y:=x-(t-s)v$ and following the same argument as in \eqref{cv} and \eqref{222}, we get following that
		\begin{align}\label{torus f}
				\int_{|v|\leq L,|\eta|\leq 2L}|f_k(s,x - (t - s)v,\eta)|^2dvd\eta \leq C_{L}\frac{1+(t-s)^3}{(t-s)^3} \mathcal{E}(\pmb{F}_0) (1+ \sup_{0\leq s \leq t} \Vert \pmb{h}(s) \Vert_{L^\infty}). 
		\end{align}
		Thus, the term in the first line of \eqref{333} can be further estimated as
		\begin{align} \label{R1_3}
			\begin{split}
			&\int_{|v|\leq L}e^{-\frac{m_j|v|^2}{8}}\sum_{k=1}^N\left\{\int_{|\eta|\leq2L}\left(1+|\eta|\right)^{-5q+10}|h_k(s,x - (t - s)v,\eta)|^5d\eta\right\}^\frac{1}{5} dv\cr
			& \leq C_{L}\left( \frac{1+(t-s)^3}{(t-s)^3} \right)^{\frac{1}{5}} \mathcal{E}(\pmb{F}_0)^{\frac{1}{5}} \sum_{k=1}^N \sup_{0\leq s \leq t} \Vert h_k(s) \Vert_{L^\infty}^{\frac{3}{5}} (1+ \sup_{0\leq s \leq t} \Vert \pmb{h}(s) \Vert_{L^\infty})^{\frac{1}{5}} \cr
			&\leq C_{L} \left( \frac{1+(t-s)^3}{(t-s)^3} \right)^{\frac{1}{5}} \mathcal{E}(\pmb{F}_0)^{\frac{1}{5}}   \left[ \sup_{0\leq s \leq t}\Vert \pmb{h}(s)\Vert_{L^\infty}^{\frac{3}{5}} \times (1+ \sup_{0\leq s \leq t}\Vert \pmb{h}(s) \Vert_{L^\infty})^{\frac{1}{5}} \right].
		\end{split}
		\end{align}
		Hence, combining \eqref{R1_1},\eqref{R1_2}, and \eqref{R1_3} yields the estimate for $R_1$: 
		\begin{align} \label{R1}
					R_1 \leq \frac{C_{q,\gamma}}{L^{1/5}}  \sup_{0\leq s \leq t} \Vert \pmb{h}(s) \Vert_{L^\infty}^2 +C_{q,\gamma,L,\delta} \mathcal{E}(\pmb{F}_0)^{\frac{1}{5}}   \left[ \sup_{0\leq s \leq t}\Vert \pmb{h}(s)\Vert_{L^\infty}^{\frac{3}{5}} \times (1+ \sup_{0\leq s \leq t}\Vert \pmb{h}(s) \Vert_{L^\infty})^{\frac{1}{5}} \right].
		\end{align}
		For $R_2$, we obtain the following estimate by using a similar argument as in the estimate of $R_1$
		\begin{align} \label{R2}
			R_2&=C_{*,1} \int_0^{t-\delta} \int_{|v|\leq L} e^{-\nu_j(v) (t-s)} e^{-\frac{m_j |v|^2}{8}} \cr
			&\quad \times \sum_{k=1}^{N}\frac{\|h_k(s)\|_{L^{\infty}}}{1+|v|}\left\{\int_{\R^3}\left(1+|\eta|\right)^{-2q+4}|h_j(s,x - (t - s)v,\eta)|^2d\eta\right\}^\frac{1}{2}dvds\cr
			& \leq \frac{C_{q,\gamma}}{L^{1/2}} \sup_{0\leq s \leq t} \Vert \pmb{h}(s) \Vert_{L^\infty}^2 +C_{q,\gamma,L,\delta} \sqrt{\mathcal{E}(\pmb{F}_0)}\left(1+ \sup_{0\leq s \leq t} \Vert \pmb{h}(s) \Vert_{L^\infty}\right)^{\frac{1}{2}}\sup_{0\leq s \leq t} \Vert \pmb{h}(s) \Vert_{L^\infty},
		\end{align} 
		where we used the condition $q>4$ in \eqref{v_weight}. 
		Hence, for the gain term $w\Gamma^+_j$, it follows from \eqref{gamma gain_split},\eqref{R1}, and \eqref{R2} that 
		\begin{align} \label{total_gain_sum}
    &\int_0^{t-\delta} \int_{|v| \leq L} e^{-\frac{m_j|v|^2}{8}} e^{-\nu_j(v)(t-s)} |w\Gamma^+_{j}(\pmb{f})(s,X(s),v)| dv ds \cr
    &\leq \frac{C_{q,\gamma}}{L^{1/5}}  \sup_{0\leq s \leq t} \Vert \pmb{h}(s) \Vert_{L^\infty}^2 + C_{q,\gamma,L,\delta} \sqrt{\mathcal{E}(\pmb{F}_0)}\left(1+ \sup_{0\leq s \leq t} \Vert \pmb{h}(s) \Vert_{L^\infty}\right)^{\frac{1}{2}}\sup_{0\leq s \leq t} \Vert \pmb{h}(s) \Vert_{L^\infty} \cr
    &\quad + C_{q,\gamma,L,\delta} \mathcal{E}(\pmb{F}_0)^{\frac{1}{5}}   \left[ \sup_{0\leq s \leq t}\Vert \pmb{h}(s)\Vert_{L^\infty}^{\frac{3}{5}} \times (1+ \sup_{0\leq s \leq t}\Vert \pmb{h}(s) \Vert_{L^\infty})^{\frac{1}{5}}\right]. 
\end{align}
		Lastly, we consider the loss part of nonlinear term $w\Gamma^-_j$. We have 
		\begin{align*}
			&\sum_{k=1}^N|w(v)\Gamma^-(f_j,f_k)(s,X(s),v)|\\
			&\leq \Vert h_j(s) \Vert_{L^\infty} \sum_{k=1}^N \int_{\R^3} |v-v_*|^\gamma \sqrt{\mu_k(v_*)} |f_k(s,X(s),v_*)| dv_* \\
			&\leq \Vert h_j(s) \Vert_{L^\infty} \sum_{k=1}^N \left(\int_{\R^3} |v-v_*|^{2\gamma} \mu_k(v_*) dv_*\right)^{1/2} \left(\int_{\R^3} |f_k(s,X(s),v_*)|^2 dv_*\right)^{1/2} \\
			&\leq C \langle v \rangle^\gamma \sum_{k=1}^N \left(\int_{\R^3} |f_k(s,X(s),v_*)|^2 dv_*\right)^{1/2},
		\end{align*}
		where the second inequality comes from H\"{o}lder's inequality. Hence, one obtains that 
			\begin{align} \label{total_loss_sum}
			&\int_0^{t-\delta} \int_{|v| \leq L} e^{-\frac{m_j|v|^2}{8}} e^{-\nu_j(v)(t-s)} |w\Gamma^-_{j}(\pmb{f})(s,X(s),v)| dv ds \cr
				&=\int_{0}^{t-\delta}\int_{|v|\leq L}e^{-\frac{m_j|v|^2}{8}} e^{-\nu_{j}(v) (t - s)} \left|\sum_{k=1}^{N} w(v) \Gamma^-(f_j, f_k)(s, X(s), v)\right|dvds\cr
				& \le C\int_{0}^{t-\delta}\int_{|v|\leq L}e^{-\frac{m_j|v|^2}{8}} e^{-\nu_{j}(v) (t - s)} \langle v \rangle^\gamma \|h_j(s)\|_{L^\infty}\sum_{k=1}^{N}\left( \int_{\R^3} |f_k(s,X(s),v_*)|^2 	dv_* \right)^{1/2}dvds\cr
				&\leq C \sup_{0\leq s \leq t} \Vert h_j(s) \Vert_{L^\infty} \cr
				&\quad \times \sum_{k=1}^N \int_0^{t-\delta} e^{-\nu_0(t-s)}\left(\int_{\R^3} e^{-\frac{m_j|v|^2}{4}} \langle v \rangle^{2\gamma} dv \right)^{1/2}\left( \int_{\{|v|\leq L\}\times \R^3} |f_k(s,X(s),v_*)|^2 dv_* dv \right)^{1/2}ds   \cr
				& \leq C\sup_{0\leq s \leq t} \Vert h_j(s) \Vert_{L^\infty} \sum_{k=1}^N \int_0^{t-\delta} e^{-\nu_0(t-s)} \cr
				&\quad \times \left[\left(\int_{|v|\leq L, |v_*| \geq 2L} |f_k(s,X(s),v_*)|^2 dv_*dv\right)^{1/2}+\left(\int_{|v|\leq L, |v_*| \leq 2L} |f_k(s,X(s),v_*)|^2 dv_*dv\right)^{1/2}  \right]ds\cr
				&\leq \frac{C_q}{L}\sup_{0\leq s \leq t} \Vert \pmb{h}(s) \Vert_{L^\infty}^2 + C_{q,L,\delta} \sqrt{\mathcal{E}(\pmb{F}_0)}   (1+\sup_{0\leq s \leq t} \Vert \pmb{h}(s) \Vert_{L^\infty})^{1/2}\sup_{0\leq s \leq t} \Vert \pmb{h}(s) \Vert_{L^\infty}.
			\end{align}
			where the last inequality comes from \eqref{torus f}. 
			Therefore, we obtain the following estimate for the nonlinear term in the last line of \eqref{h intv} by combining \eqref{nonlinear(1)}, \eqref{total_gain_sum}, and \eqref{total_loss_sum}:
	\begin{align} \label{total_nonlinear}
		&\int_0^{t-\delta} \int_{\R^3} e^{-\frac{m_j|v|^2}{8}} e^{-\nu_j(v)(t-s)} |w\Gamma_{j}(\pmb{f})(s,X(s),v)| dv ds \cr
		&\leq \frac{C_{q,\gamma}}{L^{1/5}}  \sup_{0\leq s \leq t} \Vert \pmb{h}(s) \Vert_{L^\infty}^2 + C_{q,\gamma,L,\delta} \sqrt{\mathcal{E}(\pmb{F}_0)}(1+ \sup_{0\leq s \leq t} \Vert \pmb{h}(s) \Vert_{L^\infty})^{1/2}\sup_{0\leq s \leq t} \Vert \pmb{h}(s) \Vert_{L^\infty} \cr
    &\quad + C_{q,\gamma,L,\delta} \mathcal{E}(\pmb{F}_0)^{\frac{1}{5}}   \left[ \sup_{0\leq s \leq t}\Vert \pmb{h}(s)\Vert_{L^\infty}^{\frac{3}{5}} \times (1+ \sup_{0\leq s \leq t}\Vert \pmb{h}(s) \Vert_{L^\infty})^{\frac{1}{5}}\right]. 
	\end{align}	
	By combining \eqref{rf1}, \eqref{rf2}, \eqref{Klast}, and \eqref{total_nonlinear}, we obtain the following bound for the velocity-weighted integral of the perturbation $h_j$:
	\begin{align*}
   & \int_{\mathbb{R}^3} e^{-\frac{m_j|v|^2}{8}} |{h}_j(t,x,v)| dv \\
    &\leq C e^{-\nu_0 t} \|\pmb{h}_{0}\|_{L^\infty} +C_{q,\gamma}\left(\delta+\frac{1}{L^{1/5}}\right) \left[\sup_{0\leq s \leq t} \Vert \pmb{h}(s)\Vert_{L^\infty}+\sup_{0\leq s \leq t} \Vert \pmb{h}(s)\Vert_{L^\infty}^2\right] \cr 
      &\quad +C_{q,\gamma,L,\delta}\sqrt{\mathcal{E}(\pmb{F}_0)}\left( 1+\sup_{0\leq s\leq t}\|\pmb{h}(s)\|_{L^\infty}\right)^\frac{1}{2}\left[1+ \sup_{0\leq s \leq t} \Vert \pmb{h}(s)\Vert_{L^\infty} \right] \\
		    &\quad + C_{q,\gamma,L,\delta} \mathcal{E}(\pmb{F}_0)^{\frac{1}{5}}   \left[ \sup_{0\leq s \leq t}\Vert \pmb{h}(s)\Vert_{L^\infty}^{\frac{3}{5}} \times (1+ \sup_{0\leq s \leq t}\Vert \pmb{h}(s) \Vert_{L^\infty})^{\frac{1}{5}}\right] \\
    &:= C e^{-\nu_0 t} \|\pmb{h}_{0}\|_{L^\infty}  + E.
    \end{align*}
    To achieve the bound \eqref{R}, we first choose $L>0$ sufficiently large and $\delta>0$ sufficiently small so that
\[
C_{q,\gamma}\left(\delta+\frac{1}{L^{1/5}}\right) \left[\sup_{0\leq s \leq t} \Vert \pmb{h}(s)\Vert_{L^\infty}+\sup_{0\leq s \leq t} \Vert \pmb{h}(s)\Vert_{L^\infty}^2\right] \leq \frac{1}{8C_\nu}.
\]
Next, we choose $\varepsilon_1=\varepsilon_1(\bar M,L,\delta)>0$ sufficiently small so that, whenever
\[
\mathcal{E}(\pmb{F}_0)\le \varepsilon_1,
\]
it follows that
\begin{align*} 
    E \leq \frac{1}{4C_\nu}.
\end{align*}
Since $\|\pmb{h}_0\|_{L^\infty}$ is not assumed to be small, we define
\begin{align} \label{tilde_t}
\tilde t:=\frac{1}{\nu_0}\log\bigl(4C_{\nu,1}\|\pmb{h}_0\|_{L^\infty}\bigr),
\end{align}
where $C_{\nu,1}:= C C_{\nu}$. 
Then for every $t\ge \tilde t$,
\[
C e^{-\nu_0 t}\|\pmb{h}_0\|_{L^\infty}\le \frac{1}{4C_\nu}.
\]
For all $j \in \{1, \cdots, N \}$, we obtain the desired result 
\begin{align*} 
    \int_{\mathbb{R}^3} e^{-\frac{m_j|v|^2}{8}} |{h}_j(t,x,v)| dv \leq \frac{1}{4C_{\nu}} +\frac{1}{4C_{\nu}}  = \frac{1}{2C_\nu}, \quad \forall t \geq \tilde{t}. 
\end{align*} 
	\end{proof}
		\subsection{$L^\infty$ estimate} 
		Under large-amplitude perturbations, $\mathcal{R}_i(f)$ has a positive lower bound for $t \geq \tilde{t}$ for each $i \in \{1,\cdots,N\}$. Therefore, in this subsection, we consider the following reformulated Boltzmann equation \eqref{multi pert} for each $i\in \{1,\cdots, N\}$.
				Multiplying the equation by the velocity weight function $w(v)$ and denoting $h_i=wf_i$, we can represent $h_i$ via Duhamel's principle as follows:
		\begin{align} \label{mild_form}
			 h_i(t,x,v) &= e^{-\int_0^t \mathcal{R}_i(\pmb{f})(s,x-(t-s)v,v)ds}h_{0,j}(x-tv,v) \nonumber \\
			 &\quad + \int_0^t e^{-\int_s^t \mathcal{R}_i(\pmb{f}) (\tau,x-(t-\tau)v,v)d\tau}\left [K_{w,i} (\pmb{h}) +w\Gamma^+_{i}(\pmb{f})\right ](s,x-(t-s)v,v)ds,
		\end{align}
		where the operator ${K}_{w,i}$ is given by 
		\begin{align*}
			K_{w,i} (\pmb{h})= \sum_{j=1}^N\int_{\R^3} k_{ij}(v,v_*) \frac{w(v)}{w(v_*)} h_j(v_*) dv_*= \sum_{j=1}^N \int_{\R^3} k^w_{ij}(v,v_*)h_j(v_*)dv_*.
		\end{align*}
		 If we define the solution operator by
		\begin{align}\label{G}
		 	G_{v,i}(t,s) := e^{-\int_s^t \mathcal{R}_i(\pmb{f})(\tau,x-(t-\tau)v,v)d\tau},
		 \end{align}
		 then, under the assumption of Proposition \ref{Rf est}, the solution operator $G_{v,i}$ satisfies the following bound:
		\begin{align} \label{Gv}
			G_{v,i}(t,s) &= G_{v,i}(t,s) \chi_{\{t \leq \tilde{t}\}} +G_{v,i}(t,s)\chi_{\{s \leq \tilde{t} \leq t\}} + G_{v,i}(t,s) \chi_{\{\tilde{t} \leq s \leq t \}} \cr
			&\leq e^{\frac{\nu_0}{2} \tilde{t}}e^{-\frac{\nu_0}{2} (t-s)} \chi_{\{t \leq \tilde{t}\}} +\exp \left\{-\int_{\tilde{t}}^t \frac{\nu_i(v)}{2}d\tau \right\}\chi_{\{s\leq \tilde{t}\leq t \}}+ \exp \left\{-\int_{s}^t \frac{\nu_i(v)}{2}d\tau \right\}\chi_{\{\tilde{t}\leq s \leq t \}}   \cr
			&\leq e^{\frac{\nu_0}{2} \tilde{t}}e^{-\frac{\nu_0}{2} (t-s)} \chi_{\{t \leq \tilde{t}\}} +e^{-\frac{\nu_0}{2} (t-\tilde{t})} \chi_{\{ s\leq \tilde{t} \leq t \}}   + e^{-\frac{\nu_0}{2} (t-s)}  \chi_{\{ s\leq \tilde{t} \leq t \}}  \cr
			&\leq e^{\lambda_0 \tilde{t}} e^{-\lambda_0 (t-s)},
		\end{align}
		where $\lambda_0 := \frac{\nu_0}{2}$ and $\nu_0$ is defined in Lemma \ref{CFE}. Based on the Duhamel representation for the weighted perturbation $h_i$, we now aim to establish a uniform $L^\infty$ bound. The main difficulty in the large-amplitude regime is that the damping effect of the collision frequency does not immediately dominate. However, as shown in \eqref{Gv}, $G_{v,i}$ decays exponentially for $t \ge \tilde{t}$.		
		
		By integrating the terms $K_{w,i} (\pmb{h})$ and $w\Gamma^+_{i}(\pmb{f})$ along the characteristic trajectories and utilizing the pointwise estimates established in the previous lemmas, we close the $L^\infty$ estimates. The following theorem provides a precise bound for the weighted perturbation, showing that the $L^\infty$ norm of the solution can be controlled by the initial data and the contributions of the linear and nonlinear terms.
		\begin{theorem}\label{Linfty}
			Let $h_i=h_i(t,x,v)$ be a solution to 
			\begin{align*}
				\p_t h_i + v\cdot \nabla_x h_i + \mathcal{R}_i(\pmb{f}) h_i = {K}_{w,i} (\pmb{h}) + w\Gamma^+_{i}(\pmb{f}). 
			\end{align*}
			Let $T>0$ and $t\in [0,T]$. Under the a priori assumption 
			\begin{align*}
			\sup_{0\leq t \leq T} \Vert \pmb{h}(t) \Vert_{L^\infty} \leq \bar{M}< \infty, 
			\end{align*}
			 there exists a constant $C_{*,2}=C_{*,2}(q,\gamma) > 0$ such that 
			\begin{align} \label{h_final_bound}
				\|\pmb{h}(t)\|_{L^\infty} \leq C_{*,2} e^{2\lambda_0 \tilde{t}}e^{-\frac{\lambda_0}{2} t} \| \pmb{h}_{0} \|_{L^\infty}\left(1+\int_{0}^{t}\|\pmb{h}(s)\|_{L^\infty}ds\right)+D, 
			\end{align}
			where $\lambda_0$ is defined in \eqref{Gv}, and $D$ denotes a remainder term to be specified in the proof.		\end{theorem}
		\begin{proof}
		Fix $(t,x,v) \in [0,T]\times\T^3\times\R^3$. Using the mild formulation \eqref{mild_form} of $h_i$, we have 
		\begin{align} \label{split_h}
			|h_i(t,x,v)| &\leq G_{v,i}(t,0) |h_{0,i}(x-vt,v)| \cr
			&\quad + \int_0^t G_{v,i}(t,s) |{K}_{w,i}(\pmb{h})(s,x-(t-s)v,v)|ds \cr 
			&\quad + \int_0^t G_{v,i}(t,s) |w\Gamma^+_{i}(\pmb{f})(s,x-(t-s)v,v)| ds \cr 
			&:= I + II + III. 
		\end{align}
		where $G_{v,i}$ is defined in \eqref{G}. For $I$ in \eqref{split_h}, we use \eqref{Gv} to obtain
		\begin{align} \label{I_final} 
			I \leq e^{\lambda_0 \tilde{t}}e^{-\lambda_0 t} \Vert \pmb{h}_{0} \Vert_{L^\infty} .
		\end{align} 
		For the term $II$ in \eqref{split_h}, we apply Duhamel's principle once more to obtain
		\begin{align}\label{II}
			II &\leq \sum_{j=1}^N \int_0^t G_{v,i}(t,s) \int_{\R^3} |k^w_{ij}(v,v_*) h_j(s,x-(t-s)v,v_*)|dv_* ds \nonumber \\
			&\leq  \sum_{j=1}^N \int_0^t G_{v,i}(t,s) \int_{\R^3} |k^w_{ij}(v,v_*)| G_{v_*,j}(s,0) dv_* ds \Vert h_{0,j}\Vert_{L^\infty} \nonumber \\
			&\quad +  \sum_{j=1}^N \int_0^t G_{v,i}(t,s) \int_{\R^3}  |k^w_{ij}(v,v_*)|  \int_0^s G_{v_*,j} (s,s') |K_{w,j}(\pmb{h}) (s',\hat X(s'),v_*)| ds' dv_* ds \nonumber \\
			&\quad + \sum_{j=1}^N \int_0^t G_{v,i}(t,s) \int_{\R^3}  |k^w_{ij}(v,v_*)|  \int_0^s G_{v_*,j} (s,s') |w\Gamma^+_{j}(\pmb{f}) (s',\hat X(s'),v_*)| ds' dv_* ds \nonumber \\
			&\leq e^{2\lambda_0 \tilde{t}} e^{-\frac{\lambda_0}{2}t} \Vert \pmb{h}_{0}\Vert_{L^\infty} \nonumber \\
			&\quad + e^{2\lambda_0 \tilde{t}}\sum_{j,\ell=1}^N \int_0^t e^{-\lambda_0 (t-s)} \int_{\R^3}  |k^w_{ij}(v,v_*)| \int_0^s e^{-\lambda_0 (s-s')} \int_{\R^3} |k^w_{j\ell} (v_*,\eta)h_{\ell}(s',\hat X(s'),\eta)| d\eta ds'dv_*ds  \nonumber \\
			&\quad + C_{*,1}e^{2\lambda_0 \tilde{t}}\sum_{j,\ell=1}^N \int_0^t e^{-\lambda_0 (t-s)} \int_{\R^3}  |k^w_{ij}(v,v_*)| \int_0^s e^{-\lambda_0 (s-s')} \frac{\Vert h_\ell(s') \Vert_{L^\infty}}{1+|v_*|} \nonumber \\
			&\quad \quad \times \left(\int_{\R^3} (1+|\eta|)^{-2q+4} |h_{j}(s',\hat X(s'),\eta)|^2d\eta\right)^{1/2} ds'dv_* ds \nonumber \\
			&\quad + C_{*,1}e^{2\lambda_0 \tilde{t}}\sum_{j,\ell=1}^N \int_0^t e^{-\lambda_0 (t-s)} \int_{\R^3}  |k^w_{ij}(v,v_*)| \int_0^s e^{-\lambda_0 (s-s')} \frac{\Vert h_j(s')\Vert_{L^\infty}}{(1+|v_*|)^{\frac{6}{5}-\gamma}}  \nonumber\\
			&\quad \quad \times \left( \int_{\R^3} (1+|\eta|)^{-5q+10} |h_{\ell}(s',\hat X(s'),\eta)|^5 d\eta \right)^{1/5}ds' dv_* ds\nonumber\\
			&:= II_1 + II_2 + II_3 + II_4,
		\end{align}
		where we have used Lemma \ref{K est}, Lemma \ref{gain est_new}, and \eqref{Gv}. Here, we denote 
		\begin{align*}
			\hat X(s') := x-(t-s) v- (s-s')v_*.
		\end{align*}
		Let us now estimate the term $II_2$. Let $M\geq 1$ be a large constant to be chosen later. We split the velocity domain into $\{|v_*| \geq M \}$ and $\{|v_*| \leq M \}$, and apply Lemma \ref{K est} to obtain
		\begin{align*}
			&II_2 \\
			&=  e^{2\lambda_0 \tilde{t}}\sum_{j,\ell=1}^N \int_0^t e^{-\lambda_0 (t-s)} \int_{|v_*| \geq M }  |k^w_{ij}(v,v_*)| \int_0^s e^{-\lambda_0 (s-s')} \int_{\R^3} |k^w_{j\ell} (v_*,\eta)h_{\ell}(s',\hat X(s'),\eta)| d\eta ds'dv_*ds\\
			&\quad +  e^{2\lambda_0 \tilde{t}}\sum_{j,\ell=1}^N \int_0^t e^{-\lambda_0 (t-s)} \int_{|v_*| \leq M }  |k^w_{ij}(v,v_*)| \int_0^s e^{-\lambda_0 (s-s')} \int_{\R^3} |k^w_{j\ell} (v_*,\eta)h_{\ell}(s',\hat X(s'),\eta)| d\eta ds'dv_*ds\\
			& \leq C e^{2\lambda_0 \tilde{t}} \sup_{0\leq s \leq t} \Vert \pmb{h}(s) \Vert_{L^\infty} \sum_{j=1}^N \int_0^t e^{-\lambda_0 (t-s)} \int_{|v_*|\geq M} \frac{ |k^w_{ij}(v,v_*)| }{1+|v_*|}dv_* ds \\
			&\quad + e^{2\lambda_0 \tilde{t}}\sum_{j,\ell=1}^N \int_0^t e^{-\lambda_0 (t-s)} \int_{|v_*| \leq M }  |k^w_{ij}(v,v_*)| \int_0^s e^{-\lambda_0 (s-s')} \int_{\R^3} |k^w_{j\ell} (v_*,\eta)h_{\ell}(s',\hat X(s'),\eta)| d\eta ds'dv_*ds\\
			&\leq \frac{C}{M} e^{2\lambda_0 \tilde{t}} \sup_{0\leq s \leq t} \Vert \pmb{h}(s) \Vert_{L^\infty} \\
			& \quad + e^{2\lambda_0 \tilde{t}}\sum_{j,\ell=1}^N \int_0^t e^{-\lambda_0 (t-s)} \int_{|v_*| \leq M }  |k^w_{ij}(v,v_*)| \int_0^s e^{-\lambda_0 (s-s')} \int_{\R^3} |k^w_{j\ell} (v_*,\eta)h_{\ell}(s',\hat X(s'),\eta)| d\eta ds'dv_*ds.
		\end{align*}
		To further estimate $II_2$, we split the $\eta$-domain into $\{|\eta|\ge 2M\}$ and $\{|\eta|\le 2M\}$:
		\begin{align*}
			&\int_0^t e^{-\lambda_0 (t-s)} \int_{|v_*| \leq M }  |k^w_{ij}(v,v_*)| \int_0^s e^{-\lambda_0 (s-s')} \int_{\R^3} |k^w_{j\ell} (v_*,\eta)h_{\ell}(s',\hat X(s'),\eta)| d\eta ds'dv_*ds \\
			=&\int_0^t e^{-\lambda_0 (t-s)} \int_{|v_*| \leq M }  |k^w_{ij}(v,v_*)| \int_0^s e^{-\lambda_0 (s-s')} \int_{|\eta| \geq 2M} |k^w_{j\ell} (v_*,\eta)h_{\ell}(s',\hat X(s'),\eta)| d\eta ds'dv_*ds\\
			&\quad + \int_0^t e^{-\lambda_0 (t-s)} \int_{|v_*| \leq M }  |k^w_{ij}(v,v_*)| \int_0^s e^{-\lambda_0 (s-s')} \int_{|\eta| \leq 2M} |k^w_{j\ell} (v_*,\eta)h_{\ell}(s',\hat X(s'),\eta)| d\eta ds'dv_*ds\\
			\leq& \sup_{0\leq s \leq t} \Vert h_{\ell} (s) \Vert_{L^\infty}\\
			&\times\bigg( \int_0^t e^{-\lambda_0 (t-s)} \int_{|v_*| \leq M }  |k^w_{ij}(v,v_*)| \int_0^s e^{-\lambda_0 (s-s')} \int_{|\eta| \geq 2M} |k^w_{j\ell} (v_*,\eta)| e^{\frac{|v_*-\eta|^2}{32}} e^{-\frac{M^2}{32}}d\eta ds' dv_* ds\bigg)\\
			&\quad +\int_0^t e^{-\lambda_0 (t-s)} \int_{|v_*| \leq M }  |k^w_{ij}(v,v_*)| \int_{s-\delta}^{s} e^{-\lambda_0 (s-s')} \int_{|\eta| \leq 2M} |k^w_{j\ell} (v_*,\eta)h_{\ell}(s',\hat X(s'),\eta)| d\eta ds'dv_*ds\\
			&\quad +\int_0^t e^{-\lambda_0 (t-s)} \int_{|v_*| \leq M }  |k^w_{ij}(v,v_*)| \int_{0}^{s-\delta} e^{-\lambda_0 (s-s')} \int_{|\eta| \leq 2M} |k^w_{j\ell} (v_*,\eta)h_{\ell}(s',\hat X(s'),\eta)| d\eta ds'dv_*ds\\
			&\leq C\left(\frac{1}{M}+\delta\right) \sup_{0\leq s \leq t} \Vert \pmb{h} (s) \Vert_{L^\infty}\\ 
			&\quad + \int_0^t e^{-\lambda_0(t-s)} \int_0^{s-\delta} e^{-\lambda_0(s-s')} \left(\int_{|v_*|\leq M, |\eta| \leq 2M}|h_{\ell}(s',\hat X(s'),\eta)|^2 d\eta dv_*\right)^{1/2}ds'ds \\
			&\leq C\left(\frac{1}{M}+\delta\right) \sup_{0\leq s \leq t} \Vert \pmb{h} (s) \Vert_{L^\infty}  + C_{M,\delta}\sqrt{\mathcal{E}(\pmb{F}_0)}\left(1+\sup_{0\leq s\leq t}\|\pmb{h}(s)\|_{L^\infty}\right)^{1/2},
		\end{align*} 
		where we use Lemma \ref{K est} and the last inequality comes from \eqref{111} and \eqref{222} in the proof of Proposition \ref{Rf est}. Therefore, we derive the estimate for the term $II_2$ in \eqref{II} as follows: 
		\begin{align} \label{II2}
			II_2 &\leq C_{\tilde{t}} \left(\frac{1}{M}+\delta\right) \sup_{0\leq s \leq t} \Vert \pmb{h} (s) \Vert_{L^\infty}  + C_{M,\delta,\tilde{t}} \sqrt{\mathcal{E}(\pmb{F}_0)}\left(1+\sup_{0\leq s\leq t}\|\pmb{h}(s)\|_{L^\infty}\right)^{1/2}.
		\end{align}
		Similarly, for the term $II_3$ in \eqref{II}, we split the velocity domain into $\{|v_*|\le M\}$ and $\{|v_*|\ge M\}$.
		\begin{align*}
			&\int_0^t e^{-\lambda_0 (t-s)} \int_{\R^3}  |k^w_{ij}(v,v_*)| \\
			&\quad \times \int_0^s e^{-\lambda_0 (s-s')} \frac{\Vert h_\ell(s') \Vert_{L^\infty}}{1+|v_*|}\left(\int_{\R^3} (1+|\eta|)^{-2q+4} |h_{j}(s',\hat X(s'),\eta)|^2d\eta\right)^{1/2} ds'dv_* ds\\
			 &= \int_0^t e^{-\lambda_0 (t-s)} \int_{|v_*| \geq M}  |k^w_{ij}(v,v_*)| \\
			 &\qquad \times \int_0^s e^{-\lambda_0 (s-s')} \frac{\Vert h_\ell(s') \Vert_{L^\infty}}{1+|v_*|}\left(\int_{\R^3} (1+|\eta|)^{-2q+4} |h_{j}(s',\hat X(s'),\eta)|^2d\eta\right)^{1/2} ds'dv_* ds\\
			&\quad +\int_0^t e^{-\lambda_0 (t-s)} \int_{|v_*|\leq M}  |k^w_{ij}(v,v_*)| \\
			&\qquad \times \int_0^s e^{-\lambda_0 (s-s')} \frac{\Vert h_\ell(s') \Vert_{L^\infty}}{1+|v_*|}\left(\int_{\R^3} (1+|\eta|)^{-2q+4} |h_{j}(s',\hat X(s'),\eta)|^2d\eta\right)^{1/2} ds'dv_* ds\\
			&\leq \frac{C}{M} \sup_{0\leq s \leq t} \Vert \pmb{h}(s) \Vert_{L^\infty}^2  \\
			&\quad +\sup_{0\leq s \leq t} \Vert \pmb{h}(s) \Vert_{L^\infty}\int_0^t e^{-\lambda_0 (t-s)} \int_{|v_*|\leq M}  |k^w_{ij}(v,v_*)|\\
			&\qquad \times \int_0^s e^{-\lambda_0 (s-s')}
			 \left(\int_{|\eta| \geq 2M} (1+|\eta|)^{-2q+4} |h_{j}(s',\hat X(s'),\eta)|^2d\eta\right)^{1/2} ds'dv_* ds\\
			 &\quad + \sup_{0\leq s \leq t} \Vert \pmb{h}(s) \Vert_{L^\infty}\int_0^t e^{-\lambda_0 (t-s)} \int_{|v_*|\leq M}  |k^w_{ij}(v,v_*)|\\
			 &\qquad \times\int_0^s e^{-\lambda_0 (s-s')}  \left(\int_{|\eta| \leq 2M} (1+|\eta|)^{-2q+4} |h_{j}(s',\hat X(s'),\eta)|^2d\eta\right)^{1/2} ds'dv_* ds \\ 
			&\leq  \frac{C}{\sqrt{M}} \sup_{0\leq s \leq t} \Vert \pmb{h}(s) \Vert_{L^\infty}^2 + \sup_{0\leq s \leq t} \Vert \pmb{h}(s) \Vert_{L^\infty}\int_0^t e^{-\lambda_0 (t-s)} \int_{|v_*|\leq M}  |k^w_{ij}(v,v_*)|\\
			&\quad\times  \Bigg[\int_{s-\delta}^{s}e^{-\lambda_0 (s-s')}\left(\int_{|\eta| \leq 2M} (1+|\eta|)^{-2q+4}|h_j(s',\hat X(s'),\eta)|^2 d\eta\right)^{1/2}ds'dv_*ds  \\
			&\qquad + \int_{0}^{s-\delta}e^{-\lambda_0 (s-s')}\left(\int_{|\eta| \leq 2M} (1+|\eta|)^{-2q+4}|h_j(s',\hat X(s'),\eta)|^2 d\eta\right)^{1/2}ds'dv_*ds\bigg]\\
			&\leq C\left(\frac{1}{\sqrt{M}} +\delta \right)\sup_{0\leq s \leq t} \Vert \pmb{h}(s) \Vert_{L^\infty}^2 \\
			&\quad +\sup_{0\leq s \leq t} \Vert \pmb{h}(s) \Vert_{L^\infty}\int_0^t e^{-\lambda_0 (t-s)} \int_{0}^{s-\delta}e^{-\lambda_0 (s-s')} \left(\int_{|v_*|\leq M, |\eta| \leq 2M} |h_j(s',\hat X(s'),\eta)|^2 d\eta dv_* \right)^{1/2}ds'ds\\
			&\leq C\left(\frac{1}{\sqrt{M}} +\delta \right)\sup_{0\leq s \leq t} \Vert \pmb{h}(s) \Vert_{L^\infty}^2  + C_{M,\delta}\sqrt{\mathcal{E}(\pmb{F}_0)}\sup_{0\leq s \leq t} \Vert \pmb{h}(s) \Vert_{L^\infty} \left(1+\sup_{0\leq s\leq t}\|\pmb{h}(s)\|_{L^\infty}\right)^{1/2},
		\end{align*}
		where we use the condition $q>4$ in \eqref{v_weight}, and apply H\"{o}lder's inequality together with \eqref{k_square}. Moreover, the last inequality follows from \eqref{111} and \eqref{222} in the proof of Proposition \ref{Rf est}. Hence, the term $II_3$ in \eqref{II} can be further bounded by 
		\begin{align} \label{II3} 
			II_3 &\leq C_{q,\gamma,\tilde{t}}\left(\frac{1}{\sqrt{M}} +\delta \right)\sup_{0\leq s \leq t} \Vert \pmb{h}(s) \Vert_{L^\infty}^2 + C_{q,\gamma,M,\delta,\tilde{t}}\sqrt{\mathcal{E}(\pmb{F}_0)}\sup_{0\leq s \leq t} \Vert \pmb{h}(s) \Vert_{L^\infty} \left(1+\sup_{0\leq s\leq t}\|\pmb{h}(s)\|_{L^\infty}\right)^{1/2}. 
		\end{align}
		Since $II_4$ can be treated similarly to $II_3$, we omit the details. Following the same argument, we obtain the following bound for $II_4$:
		\begin{align}\label{II4} 
			II_4 &\leq C_{q,\gamma,\tilde{t}} \left(\frac{1}{M^{\frac{1}{5}}} +\delta \right)\sup_{0\leq s \leq t} \Vert \pmb{h}(s) \Vert_{L^\infty}^2 + C_{q,\gamma,M,\delta,\tilde{t}}\mathcal{E}(\pmb{F}_0)^{\frac{1}{5}} \sup_{0\leq s \leq t}\Vert \pmb{h}(s) \Vert_{L^\infty}^{\frac{8}{5}}\left(1+\sup_{0\leq s\leq t}\|\pmb{h}(s)\|_{L^\infty}\right)^{\frac{1}{5}}.
		\end{align}
		Combining \eqref{II}, \eqref{II2}, \eqref{II3}, and \eqref{II4} yields the bound for $II$ in \eqref{split_h}:
		\begin{align} \label{II_final}
			II &\leq e^{2\lambda_0 \tilde{t}} e^{-\frac{\lambda_0}{2}t} \Vert \pmb{h}_{0}\Vert_{L^\infty}   \nonumber \\
			&\quad +C_{q,\gamma,\tilde{t}} \left(\frac{1}{M^{\frac{1}{5}}} + \delta \right)\left[\sup_{0\leq s \leq t} \Vert \pmb{h}(s) \Vert_{L^\infty}+\sup_{0\leq s \leq t} \Vert \pmb{h}(s) \Vert_{L^\infty}^2 \right]\nonumber \\
			&\quad + C_{q,\gamma,M,\delta,\tilde{t}}\left(\sqrt{\mathcal{E}(\pmb{F}_0)} +\mathcal{E}(\pmb{F}_0)^{\frac{1}{5}}\right) \left[1+ \sup_{0\leq s \leq t} \Vert \pmb{h}(s) \Vert_{L^\infty}^{\frac{9}{5}} \right].
		\end{align}
		From now on, we consider the remaining part $III$ in \eqref{split_h}. We apply Lemma \ref{gain est_new} to estimate $III$
\begin{align} \label{III_split}
    III &= \int_0^t G_{v,i}(t,s) |w\Gamma^+_{i}(\pmb{f})(s, x-(t-s)v, v)| \, ds \cr
    &\leq C_{*,1}\sum_{j=1}^{N} \int_0^t G_{v,i}(t,s) \frac{\|h_j(s)\|_{L^\infty}}{1+|v|} \left( \int_{\mathbb{R}^3} (1+|v_*|)^{-2q+4} |h_i(s, X(s), v_*)|^2 \, dv_* \right)^{1/2} ds \cr
    &\quad +C_{*,1} \sum_{j=1}^{N}  \int_0^t G_{v,i}(t,s) \frac{\|h_i(s)\|_{L^\infty}}{(1+|v|)^{\frac{6}{5}-\gamma}} \left( \int_{\mathbb{R}^3} (1+|v_*|)^{-5q+10} |h_j(s, X(s), v_*)|^5 \, dv_* \right)^{1/5} ds \cr
    &:= III_1 + III_2,
\end{align}
where $X(s) = x-(t-s)v$. We focus on the estimate of $III_2$, since $III_1$ can be treated similarly. We decompose the velocity domain of $v_*$ into $\{|v_*| \geq M\}$ and $\{|v_*| \leq M\}$.
	 For the first case $\{|v_*| \geq M\}$, we get the bound by direct computation
	\begin{align}\label{III2_small1}
		&C_{*,1} \sum_{j=1}^{N}  \int_0^t G_{v,i}(t,s) \frac{\|h_i(s)\|_{L^\infty}}{(1+|v|)^{\frac{6}{5}-\gamma}} \left( \int_{\mathbb{R}^3} (1+|v_*|)^{-5q+10} |h_j(s, X(s), v_*)|^5 \, dv_* \right)^{1/5} ds\cr
		&\leq C_{*,1}e^{\lambda_0 \tilde{t}} \sum_{j=1}^{N}\sup_{0\leq s \leq t }\| h_i(s)\|_{L^\infty} \sup_{0\leq s \leq t }\| h_j(s)\|_{L^\infty} \int_0^te^{-\lambda_0(t-s) } \left( \int_{|v_*|\geq M} (1+|v_*|)^{-5q + 10}  dv_*\right)^\frac{1}{5}ds\cr
		& \leq \frac{C_{q,\gamma,\tilde{t}}}{M^\frac{1}{5}}\sup_{0\leq s \leq t }\| \pmb{h}(s)\|_{L^\infty}^2.
	\end{align}
	For the second case $\{|v_*| \leq M\}$, we apply Duhamel's principle once more to $h_j$. Recall that $\hat{X}(\tau) = X(s) - (s-\tau)v_*$. Using the triangle inequality and \eqref{Gv}, we obtain the bound	
	\begin{align} \label{III2_split}
	III_2 \leq III_{2,1} + III_{2,2} + III_{2,3},
	\end{align}
	where the terms are defined by
	\begin{align*}
		III_{2,1} &:= C_{q,\gamma}e^{2\lambda_0 \tilde{t}} \sum_{j=1}^N \int_0^t e^{-\lambda_0 (t-s)} \Vert \pmb{h}(s) \Vert _{L^\infty} \left( \int_{|v_*|\leq M} (1+|v_*|)^{-5q + 10 }  e^{-5\lambda_0 s}|h_{0,j}(\hat X(0),v_*)|^5 dv_*\right)^\frac{1}{5}ds,\cr
		III_{2,2} &:= C_{q,\gamma} e^{2\lambda_0 \tilde{t}} \int_0^t e^{-\lambda_0(t-s)} \Vert \pmb{h}(s) \Vert_{L^\infty} \\
		&\quad \times \left( \int_{|v_*|\leq M} (1+|v_*|)^{-5q + 10 }  \left(\int_{s-\delta}^{s}e^{-\lambda_0 (s-\tau)}\left|[K_{w,j}(\pmb{h})+w\Gamma^+_{j}(\pmb{f})](\tau,\hat X(\tau),v_*)\right|d\tau\right)^5 dv_*\right)^\frac{1}{5}ds,\cr
		III_{2,3} &:=C_{q,\gamma} e^{2\lambda_0 \tilde{t}} \int_0^t e^{-\lambda_0(t-s)} \Vert \pmb{h}(s) \Vert_{L^\infty} \\
		&\quad \times \left( \int_{|v_*|\leq M} (1+|v_*|)^{-5q + 10 }  \left(\int_{0}^{s-\delta}e^{-\lambda_0 (s-\tau)}\left|[K_{w,j}(\pmb{h})+w\Gamma^+_{j}(\pmb{f})](\tau,\hat X(\tau),v_*)\right|d\tau\right)^5 dv_*\right)^\frac{1}{5}ds.
	\end{align*}
	The terms $III_{2,1}$ and $III_{2,2}$ in \eqref{III2_split} can be estimated in the same way as \eqref{rf1} and \eqref{rf2} in the proof of Proposition \ref{Rf est}. In particular, we obtain
		\begin{align} \label{III2,1}
		III_{2,1}\leq C_{q,\gamma} e^{2\lambda_0 \tilde{t}}e^{-\lambda_0 t}\|\pmb{h}_{0}\|_{L^\infty} \int_0^t \Vert  \pmb{h}(s)\Vert_{L^\infty} ds
	\end{align}
	and 
	\begin{align} \label{III2,2}
		III_{2,2}&\leq C_{q,\gamma,\tilde{t}}\delta\sup_{0\leq s\leq t} \|\pmb{h}(s)\|_{L^\infty}^2\left(1+ \sup_{0\leq s \leq t}\|\pmb{h}(s)\|_{L^\infty}\right).
	\end{align}
	It remains to estimate the term $III_{2,3}$.  To apply Lemma \ref{RE}, we split the fifth power into cubic and quadratic parts by using Lemma \ref{K est} and Corollary \ref{nonlinear est}. Specifically,
		\begin{align} \label{III2,3_split}
		III_{2,3}&\leq C_{q,\gamma,\tilde{t}}\sup_{0\leq s \leq t} \Vert \pmb{h}(s)\Vert_{L^\infty} \left[\sup_{0\leq s\leq t}\|\pmb{h}(s)\|_{L^\infty}  + \sup_{0\leq s \leq t} \|\pmb{h}(s)\|_{L^\infty}^2 \right]^{3/5} \int_0^t e^{-\lambda_0(t-s)}ds \cr 
		&\quad \times 
		\left( \int_{|v_*|\leq M} (1+|v_*|)^{-5q + 10 }  \left(\int_{0}^{s-\delta}e^{-\lambda_0 (s-\tau)}\left|[K_{w,j}(\pmb{h})+w\Gamma^+_{j}(\pmb{f})](\tau,\hat X(\tau),v_*)\right|d\tau\right)^2 dv_*\right)^\frac{1}{5}\cr
		&\leq C_{q,\gamma,\tilde{t}}\sup_{0\leq s \leq t} \Vert \pmb{h}(s)\Vert_{L^\infty} \left[\sup_{0\leq s\leq t}\|\pmb{h}(s)\|_{L^\infty}  + \sup_{0\leq s \leq t} \|\pmb{h}(s)\|_{L^\infty}^2 \right]^{3/5} \int_0^t e^{-\lambda_0(t-s)}ds\cr
		&\quad  \times \left(III_{2,3,1} + III_{2,3,2}\right), 
	\end{align}
	where 
	\begin{align*}
		III_{2,3,1}&=\left( \int_{|v_*|\leq M} (1+|v_*|)^{-5q + 10 }  \left(\int_{0}^{s-\delta}e^{-\lambda_0 (s-\tau)}\left|K_{w,j}(\pmb{h})(\tau,\hat X(\tau),v_*)\right|d\tau\right)^2 dv_*\right)^\frac{1}{5},\cr
		III_{2,3,2}&=\left( \int_{|v_*|\leq M} (1+|v_*|)^{-5q + 10 }  \left(\int_{0}^{s-\delta}e^{-\lambda_0 (s-\tau)}\left|w\Gamma^+_{j}(\pmb{f})(\tau,\hat X(\tau),v_*)\right|d\tau\right)^2 dv_*\right)^\frac{1}{5}.			
	\end{align*}
	For the term $III_{2,3,1}$, we proceed with the following estimate 
	\begin{align*}
		II&I_{2,3,1}\\
		\leq& \sum_{k=1}^{N}\left( \int_{|v_*|\leq M} (1+|v_*|)^{-5q + 10 }  \left(\int_{0}^{s-\delta}\int_{\mathbb{R}^3}e^{-\lambda_0 (s-\tau)}\left|k^w_{jk}(v_*,\eta)h_k(\tau,\hat X(\tau),\eta)\right|d\eta d\tau\right)^2 dv_*\right)^\frac{1}{5}\cr
		=&\sum_{k=1}^{N}\bigg[ \int_{|v_*|\leq M} (1+|v_*|)^{-5q + 10 }\\
		&\qquad\times  \left(\int_{0}^{s-\delta}\left(\int_{|\eta|\geq 2M}+\int_{|\eta|\leq 2M}\right)e^{-\lambda_0 (s-\tau)}\left|k^w_{jk}(v_*,\eta)h_k(\tau,\hat X(\tau),\eta)\right|d\eta d\tau\right)^2 dv_*\bigg]^\frac{1}{5}\cr
		\leq& \frac{C_{q}}{M^\frac{2}{5}}\sup_{0\leq s \leq t }\|\pmb{h}(s)\|^{\frac{2}{5}}_{L^\infty}\cr
		&\quad +\sum_{k=1}^{N}\left( \int_{|v_*|\leq M} (1+|v_*|)^{-5q + 10 }  \left(\int_{0}^{s-\delta}\int_{|\eta|\leq 2M}e^{-\lambda_0 (s-\tau)}\left|k^w_{jk}(v_*,\eta)h_k(\tau,\hat X(\tau),\eta)\right|d\eta d\tau\right)^2 dv_*\right)^\frac{1}{5}.
	\end{align*}
	Here, the last inequality follows from the decay estimate of the weighted kernel $k^w_{jk}$ in Lemma \ref{K est}, which yields a decay factor of order $M^{-2/5}$ for $|v_*|\leq M$ and $|\eta|\geq 2M$. For the remaining part of $III_{2,3,1}$, we apply Hölder’s inequality combined with \eqref{k_square} to obtain 
	\begin{align*}
		& \left(\int_{0}^{s-\delta}\int_{|\eta|\leq 2M}e^{-\lambda_0 (s-\tau)}\left|k^w_{jk}(v_*,\eta)h_k(\tau,\hat X(\tau),\eta)\right|d\eta d\tau\right)^2 \cr
		&\leq \left( \int_{0}^{s-\delta}e^{-\lambda_0 (s-\tau)}   \int_{|\eta|\leq 2M}\left|k^w_{jk}(v_*,\eta*)\right|^2d\eta d\tau \right) \left(\int_0^{s-\delta} e^{-\lambda_0 (s-\tau)} \int_{|\eta|\leq 2M}\left|h_k(\tau,\hat X(\tau),\eta)\right|^2d\eta d\tau \right)\\
		& \leq C\int_0^{s-\delta} e^{-\lambda_0 (s-\tau)} \int_{|\eta|\leq 2M}\left|h_k(\tau,\hat X(\tau),\eta)\right|^2d\eta d\tau. 
	\end{align*}
	By a similar argument as in \eqref{torus f}, we obtain	
	\begin{align*}
		&\left( \int_{|v_*|\leq M} (1+|v_*|)^{-5q + 10 }  \left(\int_{0}^{s-\delta}\int_{|\eta|\leq 2M}e^{-\lambda_0 (s-\tau)}\left|k^w_{jk}(v_*,\eta)h_k(\tau,\hat X(\tau),\eta)\right|d\eta d\tau\right)^2 dv_*\right)^\frac{1}{5} \\ 
		&\leq C_{q,M}\left( \int_{0}^{s-\delta}e^{-\lambda_0 (s-\tau)}\int_{|v_*|\leq M, |\eta|\leq 2M}\left|f_k(\tau,\hat X(\tau),\eta)\right|^2d\eta dv_* d\tau \right)^\frac{1}{5}\\
		&\leq C_{q,M,\delta} \mathcal{E}(\pmb{F}_0)^{\frac{1}{5}} \left(1+ \sup_{0\leq s \leq t} \Vert \pmb{h} (s) \Vert_{L^\infty}\right)^{\frac{1}{5}}, 
	\end{align*}
	which implies the following bound for $III_{2,3,1}$
	\begin{align}\label{III2,3,1}
		III_{2,3,1} \leq  \frac{C_{q}}{M^\frac{2}{5}}\sup_{0\leq s \leq t }\|\pmb{h}(s)\|^{\frac{2}{5}}_{L^\infty} + C_{q,M,\delta} \mathcal{E}(\pmb{F}_0)^{\frac{1}{5}} \left(1+ \sup_{0\leq s \leq t} \Vert \pmb{h} (s) \Vert_{L^\infty}\right)^{\frac{1}{5}}. 
	\end{align}
	Next, let us consider $III_{2,3,2}$. It follows from Lemma \ref{gain est_new} that
	\begin{align} \label{split_III2,3,2}
		&III_{2,3,2}\cr
		&=\left( \int_{|v_*|\leq M} (1+|v_*|)^{-5q + 10 }  \left(\int_{0}^{s-\delta}e^{-\lambda_0 (s-\tau)}\left|w\Gamma^+_{j}(\pmb{f})(\tau,\hat X(\tau),v_*)\right|d\tau\right)^2 dv_*\right)^\frac{1}{5}\cr
		&\leq C_{*,1}\left( \int_{|v_*|\leq M} (1+|v_*|)^{-5q + 10 }  \left(\int_{0}^{s-\delta}e^{-\lambda_0 (s-\tau)}\sum_{k=1}^{N}\frac{\|h_k(\tau)\|_{L^\infty}}{1+|v_*|}\left(\int_{\R^3} \frac{|h_j(\tau,\hat X(\tau),\eta)|^2}{(1+|\eta|)^{2q-4}}d\eta\right)^\frac{1}{2}d\tau\right)^2 dv_*\right)^\frac{1}{5}\cr
		&+ C_{*,1}\left( \int_{|v_*|\leq M} (1+|v_*|)^{-5q + 10 }  \left(\int_{0}^{s-\delta}e^{-\lambda_0 (s-\tau)}\sum_{k=1}^{N}\frac{\Vert h_j(\tau)\Vert_{L^\infty}}{(1+|v_*|)^{\frac{6}{5}-\gamma}}\left( \int_{\R^3} \frac{|h_k(\tau,\hat X(\tau),\eta)|^5}{(1+|\eta|)^{5q - 10 } } d\eta\right)^\frac{1}{5}d\tau\right)^2 dv_*\right)^\frac{1}{5}\cr
		&:= III_{2,3,2,1}+III_{2,3,2,2}.
	\end{align}
	Since both terms can be treated similarly, we only provide the estimate for $III_{2,3,2,2}$ and omit the details for $III_{2,3,2,1}$. To estimate $III_{2,3,2,2}$, we split the $\eta$-domain into $|\eta| \geq 2M$ and $|\eta| \leq 2M$. For the region $|\eta| \geq 2M$, we obtain
	\begin{align} \label{III2,3,2,2,1}
    &\int_{|v_*|\leq M} (1+|v_*|)^{-5q+10} \left( \int_{0}^{s-\delta} e^{-\lambda_0 (s-\tau)}  \sum_{k=1}^{N} \frac{\|h_j(\tau)\|_{L^\infty}}{(1+|v_*|)^{\frac{6}{5}-\gamma}} \left( \int_{|\eta|\geq 2M} \frac{|{h}_k(\tau, \hat{X}(\tau), \eta)|^5}{(1+|\eta|)^{5q - 10}} d\eta \right)^{1/5} d\tau \right)^2 dv_* \cr
    &\leq C\int_{|v_*|\leq M} (1+|v_*|)^{-5q+10} \left( \int_{0}^{s-\delta} e^{-\lambda_0 (s-\tau)}  \sum_{k=1}^{N} \frac{\|\pmb{h}(\tau)\|_{L^\infty}}{(1+|v_*|)^{\frac{6}{5}-\gamma}} \left( \int_{|\eta|\geq 2M} \frac{|{h}_k(\tau, \hat{X}(\tau), \eta)|^5}{(1+|\eta|)^{5q - 11}M} d\eta \right)^{1/5}  d\tau \right)^2 dv_* \cr
    &\leq  \frac{C_{q}}{M^{2/5}} \sup_{0\leq s \leq t} \Vert \pmb{h}(s) \Vert_{L^\infty}^{4} \int_{|v_*|\leq M} (1+|v_*|)^{-5q+10} \left( \int_{0}^{s-\delta} e^{-\lambda_0 (s-\tau)}  d\tau \right)^2 dv_* \cr
    &\leq \frac{C_{q}}{M^{2/5}} \sup_{0\leq s \leq t} \Vert \pmb{h}(s) \Vert_{L^\infty}^{4}. 
\end{align}
	For the region $|\eta| \leq 2M$, it follows from H\"{o}lder's inequality and \eqref{torus f} that 
	\begin{align}\label{III2,3,2,2,2}
       &\int_{|v_*|\leq M} (1+|v_*|)^{-5q+10} \left( \int_{0}^{s-\delta} e^{-\lambda_0 (s-\tau)}  \sum_{k=1}^{N}  \frac{\|h_j(\tau)\|_{L^\infty}}{(1+|v_*|)^{\frac{6}{5}-\gamma}} \left( \int_{|\eta|\leq 2M} \frac{|{h}_k(\tau, \hat{X}(\tau), \eta)|^5}{(1+|\eta|)^{5q - 10}} d\eta \right)^{1/5} d\tau \right)^2 dv_* \cr
               &\leq C \sup_{0\leq s \leq t} \|\pmb{h}(s)\|^2_{L^\infty} \int_{0}^{s-\delta} e^{-\lambda_0(s-\tau)} \int_{|v_*|\leq M} \left( \int_{|\eta|\leq 2M} |{h}_k(\tau, \hat{X}(\tau), \eta)|^5 d\eta \right)^{2/5} dv_* d\tau \cr
        &\leq C_M \sup_{0\leq s \leq t} \|\pmb{h}(s)\|^{\frac{16}{5}}_{L^\infty} \int_{0}^{s-\delta} e^{-\lambda_0(s-\tau)} \left( \int_{|v_*|\leq M, |\eta|\leq 2M} |{f}_k(\tau, \hat{X}(\tau), \eta)|^2 d\eta dv_* \right)^{2/5} d\tau \cr
        &\leq C_{M,\delta} \mathcal{E}(\pmb{F}_0)^{\frac{2}{5}} \sup_{0\leq s \leq t } \Vert \pmb{h}(s)\Vert_{L^\infty}^{\frac{16}{5}}\left(1+\sup_{0\leq s \leq t} \Vert \pmb{h}(s) \Vert_{L^\infty}\right)^{\frac{2}{5}}.
\end{align}
	Thus, combining \eqref{split_III2,3,2}, \eqref{III2,3,2,2,1}, and \eqref{III2,3,2,2,2} yields the estimate for $III_{2,3,2,2}$  
	\begin{align}\label{III2,3,2,2}
		III_{2,3,2,2} \leq \frac{C_{q,\gamma}}{M^{2/25}} \sup_{0\leq s \leq t} \Vert \pmb{h}(s) \Vert_{L^\infty}^{\frac{4}{5}} + C_{q,\gamma, M,\delta} \mathcal{E}(\pmb{F}_0)^{\frac{2}{25}} \sup_{0\leq s \leq t} \Vert \pmb{h}(s) \Vert_{L^\infty}^{\frac{16}{25}} \left(1+\sup_{0\leq s \leq t} \Vert \pmb{h}(s) \Vert_{L^\infty} \right)^{\frac{2}{25}}. 
	\end{align}
	For $III_{2,3,2,1}$, the estimate can be derived by a similar argument; hence, we only state the result:
	\begin{align}  \label{III2,3,2,1}
		III_{2,3,2,1} \leq  \frac{C_{q,\gamma}}{M^\frac{1}{5}} 	\sup_{0\leq s \leq t}\Vert \pmb{h}(s)\Vert_{L^\infty}^{\frac{4}{5}} 	+ C_{q,\gamma,M,\delta} \mathcal{E}(\pmb{F}_0)^{\frac{1}{5}}\sup_{0\leq s \leq t} \Vert \pmb{h}(s) \Vert_{L^\infty}^{\frac{2}{5}} \left(1+\sup_{0\leq s \leq t} \Vert \pmb{h}(s) \Vert_{L^\infty}\right)^{\frac{1}{5}}. 
	\end{align}
	From \eqref{split_III2,3,2}, \eqref{III2,3,2,2} and \eqref{III2,3,2,1},  we obtain the following bound for $III_{2,3,2}$ 
	\begin{align} \label{III2,3,2}
		III_{2,3,2} &\leq \frac{C_{q,\gamma}}{M^\frac{1}{5}} 	\sup_{0\leq s \leq t}\Vert \pmb{h}(s)\Vert_{L^\infty}^{\frac{4}{5}}  + C_{q,\gamma,M,\delta} \left[\mathcal{E}(\pmb{F}_0)^{\frac{2}{25}}+\mathcal{E}(\pmb{F}_0)^{\frac{1}{5}}\right]\left(1+\sup_{0\leq s \leq t} \Vert \pmb{h}(s)\Vert_{L^\infty}\right)^{\frac{18}{25}}. 
			\end{align}
	Combining \eqref{Gv},\eqref{III2_small1},\eqref{III2_split}\eqref{III2,1},\eqref{III2,2},\eqref{III2,3_split},\eqref{III2,3,1}, and \eqref{III2,3,2}, we have completed the estimate for $III_2$ in \eqref{III_split} as follows:
	\begin{align} \label{III2}
		\begin{split}
    III_2 &\leq C_{q,\gamma} e^{2\lambda_0 \tilde{t}} e^{-\lambda_0 t} \Vert \pmb{h}_0\Vert_{L^\infty} \int_0^t\Vert \pmb{h}(s) \Vert_{L^\infty} ds\cr
    &\quad + C_{q,\gamma,\tilde{t}}\left(\frac{1}{M^{\frac{1}{5}}} +\delta\right)\left(1+ \sup_{0\leq s \leq t} \Vert \pmb{h}(s) \Vert_{L^\infty} \right)^3 \\
    &\quad + C_{q,\gamma,M,\delta,\tilde{t}} \left[\mathcal{E}(\pmb{F}_0)^{\frac{2}{25}}+\mathcal{E}(\pmb{F}_0)^{\frac{1}{5}}\right]\left(1+\sup_{0\leq s \leq t} \Vert \pmb{h}(s)\Vert_{L^\infty}\right)^3.
    \end{split}
    \end{align}	
The term $III_1$ can be handled similarly, and we obtain	
\begin{align} \label{III1}
		\begin{split}
		III_1 &\leq C_{q,\gamma}e^{2\lambda_0 \tilde{t}}e^{-\lambda_0 t}\|\pmb{h}_{0}\|_{L^\infty}\int_0^t \Vert \pmb{h}(s) \Vert_{L^\infty} ds \cr
		&\quad + C_{q,\gamma,\tilde{t}} \left( \frac{1}{\sqrt{M}} +\delta\right) \left(1+ \sup_{0 \leq s \leq t} \Vert \pmb{h}(s) \Vert_{L^\infty} \right)^3 \\
		&\quad + C_{q,\gamma,M,\delta,\tilde{t}} \left[ \sqrt{\mathcal{E}(\pmb{F}_0)} + \mathcal{E}(\pmb{F}_0)^{\frac{1}{5}} \right] \left(1+ \sup_{0\leq s \leq t} \Vert \pmb{h}(s) \Vert_{L^\infty} \right)^3. 
		\end{split}
	\end{align}
	Using the estimates \eqref{III_split},\eqref{III2}, and \eqref{III1}, we obtain the following estimate for $III$:
	\begin{align}\label{III_final}
		\begin{split}
		III &\leq C_{q,\gamma}e^{2\lambda_0 \tilde{t}}e^{-\lambda_0 t}\|\pmb{h}_{0}\|_{L^\infty}\int_0^t \Vert \pmb{h}(s) \Vert_{L^{\infty}} ds \cr
		&\quad + C_{q,\gamma,\tilde{t}} \left(\frac{1}{M^{\frac{1}{5}}} +\delta\right)\left(1+ \sup_{0\leq s \leq t} \Vert \pmb{h}(s) \Vert_{L^\infty} \right)^3 \\
		&\quad + C_{q,\gamma,M,\delta,\tilde{t}} \left[\mathcal{E}(\pmb{F}_0)^{\frac{2}{25}} +\sqrt{\mathcal{E}(\pmb{F}_0)}\right] \left(1+ \sup_{0\leq s \leq t} \Vert \pmb{h}(s) \Vert_{L^\infty} \right)^3.
		\end{split}			
		\end{align}
	 For each $i \in \{1,\cdots, N\}$, combining \eqref{I_final}, \eqref{II_final}, and \eqref{III_final}, we have
	\begin{align*}
		|h_i(t,x,v)|&\leq C_{q,\gamma} e^{2\lambda_0 \tilde{t}}e^{-\frac{\lambda_0}{2} t} \Vert \pmb{h}_{0} \Vert_{L^\infty} \left(1+ \int_0^t \Vert \pmb{h}(s) \Vert_{L^\infty} ds \right) \\
					&\quad + C_{q,\gamma,\tilde{t}} \left(\frac{1}{M^{\frac{1}{5}}} +\delta\right)\left(1+ \sup_{0\leq s \leq t} \Vert \pmb{h}(s) \Vert_{L^\infty} \right)^3\\
					&\quad + C_{q,\gamma,M,\delta,\tilde{t}} \left[\mathcal{E}(\pmb{F}_0)^{\frac{2}{25}} +\sqrt{\mathcal{E}(\pmb{F}_0)}\right] \left(1+ \sup_{0\leq s \leq t} \Vert \pmb{h}(s) \Vert_{L^\infty} \right)^3.
	\end{align*}
	By taking the summation over $i \in \{1,\cdots, N\}$, we derive the estimate
	\begin{align*}
		\|\pmb{h}(t)\|_{L^\infty}&\leq C_{q,\gamma}e^{2\lambda_0 \tilde{t}}e^{-\frac{\lambda_0}{2} t} \| \pmb{h}_{0} \|_{L^\infty}\left(1+\int_{0}^{t}\|\pmb{h}(s)\|_{L^\infty}ds\right)+D,
	\end{align*}
	where 
	\begin{align} \label{D}
		D:=D(M,\delta,\mathcal{E}(\pmb{F}_0),\Vert \pmb{h}_0 \Vert_{L^\infty}, \bar{M}) &=  C_{q,\gamma,\tilde{t}} \left(\frac{1}{M^{\frac{1}{5}}} +\delta\right)\left(1+ \sup_{0\leq s \leq t} \Vert \pmb{h}(s) \Vert_{L^\infty} \right)^3 \cr
		&\quad + C_{q,\gamma,M,\delta,\tilde{t}} \left[\mathcal{E}(\pmb{F}_0)^{\frac{2}{25}} +\sqrt{\mathcal{E}(\pmb{F}_0)}\right] \left(1+ \sup_{0\leq s \leq t} \Vert \pmb{h}(s) \Vert_{L^\infty} \right)^3.
	\end{align}
		\end{proof}

		\section{Proof of theorem \ref{thm:main}}\label{sec:large_regime}
		In this section, we prove Theorem~\ref{thm:main} based on the estimates obtained in the previous sections.
		\subsection{Local in time existence in $L^\infty$}
		We now state the local-in-time existence theorem for the multi-species Boltzmann equation \eqref{mixed bol}.
		\begin{lemma}\label{local in time}
			Let $q>4$ be fixed as in the velocity weight function \eqref{v_weight}. If $F_{0,i}(x,v)=\mu_i(v)+\sqrt{\mu_i(v)}f_{0,i}(x,v)\geq 0$ and $ \|w\pmb{f}_{0}\|_{L^\infty}\leq M_0$ for some constant $M_0>0$, then there exists a time $\hat t= \hat{t}(M_0) > 0$ such that the initial value problem \eqref{mixed bol} with initial data $F_{0,i}(x,v)$ has a unique non-negative solution $F_i(t,x,v)=\mu_i(v)+\sqrt{\mu_i(v)}f_i(t,x,v)$ for all $1\leq i\leq N$ and $t\in[0,\hat t]$, satisfying
			\begin{align*}
				\sup_{0\leq t\leq \hat t}\|w\pmb{f}(t)\|_{L^\infty}\leq 2\|w\pmb{f}_0\|_{L^\infty}.
			\end{align*}
		\end{lemma}
		\begin{proof}
			We construct a sequence $\{F_i^n\}_{n\ge 0}$ via the following iteration scheme:
			\begin{align}\label{itera}
				\left\{
				\begin{aligned}
					&\partial_t F^{n+1}_i + v\cdot \nabla_x F^{n+1}_i
					+ F^{n+1}_i\sum_{j=1}^N\int_{\mathbb{R}^3\times\mathbb{S}^2} B_{ij}(v-v_*,\sigma)\,F^{n}_j(v_*)\,d\sigma\,dv_*
					=\sum_{j=1}^NQ_{ij}^+(F^{n}_i,F^{n}_j),\\
					&F_i^{n+1}(0,x,v)=F_{0,i}(x,v)\ge 0,\qquad F^{0}_i(t,x,v)=\mu_i(v).
				\end{aligned}
				\right.
			\end{align}
			First, we prove the non-negativity of $\{F_i^n\}_{n\ge 0}$.
			Recall the definition of the collision frequency $\nu_i(v)$
			\begin{align*}
				\nu_i(v)
				= \sum_{j=1}^N \int_{\mathbb{R}^3 \times \mathbb{S}^2}
				B_{ij}(v-v_*,\sigma)\,\mu_j(v_*)\,d\sigma\,dv_* .
			\end{align*}
			Then the first iterate $F_i^{1}$ admits the representation
			\begin{align*}
				F_i^{1}(t,x,v)
				&= e^{-\nu_i(v)t}\,F_{0,i}(x-tv,v) \\
				&\quad + \sum_{j=1}^N \int_{0}^{t} e^{-\nu_i(v)(t-s)}
				\int_{\mathbb{R}^3\times \mathbb{S}^2}
				B_{ij}(v-v_*,\sigma)\,\mu_j(v_*)\,d\sigma\,dv_* \,\mu_i(v)\,ds
				\;\ge 0 .
			\end{align*}
			Since $F_{0,i}\ge 0$ and $\mu_i\ge 0$, it follows that $F_i^{1}\ge 0$. Assume that $F_i^{n}\ge 0$ for all $i=1,\dots,N$.
			By Duhamel's principle, we obtain
			\begin{align*}
				F_i^{n+1}(t,x,v)
				&= e^{-\int_{0}^{t} I_i^n(\tau,X(\tau),v)\,d\tau}\,F_{0,i}(x-tv,v) \\
				&\quad + \sum_{j=1}^N \int_{0}^{t}
				e^{-\int_{s}^{t} I_i^n(\tau,X(\tau),v)\,d\tau}\,
				Q_{ij}^+(F_i^{n},F_j^{n})(s,X(s),v)\,ds.
			\end{align*}
			Here $I_i^n(t,x,v)$ is defined by 
			\begin{align*}
				I_i^n(t,x,v)
				:= \sum_{j=1}^N \int_{\mathbb{R}^3\times \mathbb{S}^2}
				B_{ij}(v-v_*,\sigma)\,F_j^{n}(t,x,v_*)\,d\sigma\,dv_* .
			\end{align*}
			Since $F_{0,i}\ge 0$ and $Q_{+,ij}(F_i^{n},F_j^{n})\ge 0$, it follows that $F_i^{n+1}\ge 0$.
			With the notation $h_i^n=h_i^n(t,x,v) = w(v)f_i^n(t,x,v)$, we can rewrite \eqref{itera} as follows:
			\begin{align}\label{iter2}
				\begin{split}
					\left\{
					\begin{aligned}
						\partial_th^{n+1}_i+ v\cdot \nabla_xh^{n+1}_i+ \nu_ih^{n+1}_i\,
						&=K_{w,i} (\pmb{h}^n) + \sum_{j=1}^N w(v)\Gamma^+_{ij}(f_i^{n},f_j^{n})
- \sum_{j=1}^N w(v)\Gamma^-_{ij}(f_i^{n+1},f_j^{n}), \\ 
						h^{n+1}_i(0,x,v) &= h_{0,i}(x,v), \qquad h^{0}\equiv0 .
					\end{aligned}
					\right.
				\end{split}
			\end{align}
			We establish the existence of a positive time $\hat{t}$ such that the solution to
			\eqref{iter2} exists on $[0,\hat{t}]$ and obeys the uniform estimate
			\begin{align}\label{local h}
				\sup_{0\leq t\leq \hat t}\|\pmb{h}^{n}(t)\|_{L^\infty}=\sup_{0\le t\le \hat{t}}\sum_{i=1}^N\|h^n_i(t)\|_{L^\infty}
				\le 2\|\pmb{h}_{0}\|_{L^\infty}.
			\end{align}
			For the step $n=1$, the representation $h^1(t,x,v)=e^{-\nu_i(v)t} h_{0,i}(x-tv,v)$
			immediately yields \eqref{local h}.
			Assume that \eqref{local h} holds for $n$.
			Applying Duhamel’s formula and using the estimates in Lemma \ref{K est} and Corollary \ref{nonlinear est} for $K_w$ and $w\Gamma$, we obtain  
			\begin{align*}
				\|h^{n+1}_i(t)\|_{L^\infty}
				&\le \|h_{0,i}\|_{L^\infty} \\
				&\quad + \int_{0}^{t}
				\Big(
				\|K_{w,i} (\pmb{h}^n)(s)\|_{L^\infty}
				+ \sum_{j=1}^N \|w\Gamma^+(f_i^{n},f_j^{n})(s)\|_{L^\infty}
				+\sum_{j=1}^N \|w\Gamma^-(f_i^{n+1},f_j^{n})(s)\|_{L^\infty}
				\Big)\,ds \\[1mm]
				&\le \|\pmb{h}_{0}\|_{L^\infty}
				+ C\hat{t} \Big(
				\sup_{0\le s\le \hat{t}}\|\pmb{h}^{n}(s)\|_{L^\infty} +\sup_{0\le s\le \hat{t}}\|\pmb{h}^{n}(s)\|_{L^\infty}^2 
				+ \sup_{0\le s\le \hat{t}}\|\pmb{h}^{n}(s)\|_{L^\infty}
				\sup_{0\le s\le \hat{t}}\|\pmb{h}^{n+1}(s)\|_{L^\infty}
				\Big).
			\end{align*}
			Summing over $i$ and taking the supremum in time, we obtain
			\begin{align*}
				\sup_{0\le t\le \hat{t}}\|\pmb{h}^{n+1}(t)\|_{L^\infty}\leq& \|\pmb{h}_{0}\|_{L^\infty}\\
				&+C\hat{t}\times\bigg(
				\sup_{0\le s\le \hat{t}}\|\pmb{h}^{n}(s)\|_{L^\infty} +\sup_{0\le s\le \hat{t}}\|\pmb{h}^{n}(s)\|_{L^\infty}^2 
				+ \sup_{0\le s\le \hat{t}}\|\pmb{h}^{n}(s)\|_{L^\infty}
				\sup_{0\le s\le \hat{t}}\|\pmb{h}^{n+1}(s)\|_{L^\infty} \bigg) \cr
				\leq& \|\pmb{h}_{0}\|_{L^\infty}+C\hat{t}\left(	2\|\pmb{h}_{0}\|_{L^\infty}+4\|\pmb{h}_{0}\|^2_{L^\infty}+2\|\pmb{h}_{0}\|_{L^\infty}\sup_{0\leq t\leq \hat t}\|\pmb{h}^{n+1}(t)\|_{L^\infty}\right)
			\end{align*}
			Choosing $\hat{t}$ sufficiently small such that 
			\[
			\hat t \le \min\left\{\frac{2}{3C\bigl(2\|\pmb{h}_0\|_{L^\infty}+1\bigr)},\ 
			\frac{1}{6C\|\pmb{h}_0\|_{L^\infty}}\right\},
			\]
			we can absorb the last term to obtain 
			\begin{align*}
				\frac{2}{3}\sup_{0\le s\le \hat{t}}\|\pmb{h}^{n+1}(s)\|_{L^\infty}
				\le \frac{4}{3}\|\pmb{h}_0\|_{L^\infty}.
			\end{align*}
			Hence,
			\begin{align}\label{local hk}
				\sup_{0\le s\le \hat{t}}\|\pmb{h}^{n+1}(s)\|_{L^\infty}
				\le 2\|\pmb{h}_0\|_{L^\infty}.
			\end{align}
			By applying the same argument and using the following property for the nonlinear term $\Gamma$
			\[
			\Gamma(g-h,f)=\Gamma(g,f)-\Gamma(h,f),
			\]
			we obtain
			\begin{align*}
				\sup_{0\le s\le \hat{t}}\|\pmb{h}^{n+1}-\pmb{h}^{n}\|_{L^\infty}\leq \frac{1}{2}\sup_{0\le s\le{\hat{t}}}\|\pmb{h}^{n}-\pmb{h}^{n-1}\|_{L^\infty}.
			\end{align*}
This implies that $\{\pmb{h}^n\}_{n\ge1}$ is a Cauchy sequence.
			Thus, it converges to a limit $\pmb{h}$ in $L^\infty$.	Moreover, $F^n_i \to F_i$ as $n \to \infty$ for each $i\in \{1,\cdots N\}$. 
			Since each $F^n_i$ is non-negative, the limit function $F_i$ is also non-negative on 
			$t \in [0,\hat{t}]$. 
			Therefore, passing to the limit in \eqref{local hk} as $n \to \infty$, we obtain 
			\begin{align*}
				\sup_{0\leq t\leq \hat{t}}\|\pmb{h}(t)\|_{L^\infty}\leq 2\|\pmb{h}_0\|_{L^\infty}.
			\end{align*}
			
					\end{proof}
		\subsection{Global-in-Time Solution in $L^\infty$}
		Having established the local-in-time existence, we extend the solution globally in time via a continuation argument based on uniform $L^\infty$ estimates.
\begin{proof}[Proof of Theorem \ref{thm:main}] We first recall the estimate established in Theorem \ref{Linfty}.
By \eqref{h_final_bound}, we have
		\begin{align}\label{main revised}
			\Vert \pmb{h}(t) \Vert_{L^{\infty}}
			\leq
			C_{q,\gamma} e^{2\lambda_0 \tilde{t}} e^{-\frac{\lambda_0}{2} t}
			\Vert \pmb{h}_0 \Vert_{L^\infty}
			\left(1 + \int_{0}^{t} \Vert \pmb{h}(s) \Vert_{L^\infty} \, ds\right)
			+ D,
		\end{align}
		where $D$ is given in \eqref{D}. Let 
		\begin{align*}
		Z(t) := 1 + \int_0^t \Vert \pmb{h}(s) \Vert_{L^{\infty}} \, ds.
		\end{align*}
		By differentiating $Z(t)$ with respect to time and using Theorem \ref{Linfty}, together with the assumption $\|\pmb{h}_0\|_{L^\infty} \le M_0$, we obtain
		\begin{align*}
			Z'(t) \leq C_{q,\gamma} M_0 e^{2\lambda_0 \tilde{t}} e^{-\frac{\lambda_0}{2} t} Z(t) + D.
		\end{align*}
		Multiplying both sides by the integrating factor
		 \begin{align*}
		 \exp \left\{ \frac{2 C_{q,\gamma} M_0}{\lambda_0} e^{2\lambda_0 \tilde{t}} (e^{-\frac{\lambda_0}{2} t} - 1) \right\},
		 \end{align*}
		  we deduce
		\begin{align*}
			\left( Z(t) \exp \left\{ \frac{2 C_{q,\gamma} M_0}{\lambda_0} e^{2\lambda_0 \tilde{t}} (e^{-\frac{\lambda_0}{2} t} - 1) \right\} \right)' 
			\leq D \exp \left\{ \frac{2 C_{q,\gamma} M_0}{\lambda_0} e^{2\lambda_0 \tilde{t}} (e^{-\frac{\lambda_0}{2} t} - 1) \right\} \leq D.
		\end{align*}
		Integrating over $[0, t]$ yields
		\begin{align}\label{proof est 1}
			Z(t) \leq (1 + Dt) \exp \left\{ \frac{2 C_{q,\gamma} M_0}{\lambda_0} e^{2\lambda_0 \tilde{t}} (1 - e^{-\frac{\lambda_0}{2} t}) \right\} 
			\leq (1 + Dt) \exp \left\{ \frac{2 C_{q,\gamma}M_0}{\lambda_0} e^{2\lambda_0 \tilde{t}} \right\}.
		\end{align}
		Substituting estimate \eqref{proof est 1} into \eqref{main revised} gives  
		\begin{equation}\label{proof est 2}
			\Vert \pmb{h}(t) \Vert_{L^\infty}
			\le C_{q,\gamma} M_0 e^{2\lambda_0 \tilde{t} + \frac{2 C_{q,\gamma} M_0}{\lambda_0} e^{2\lambda_0 \tilde{t}}} (1 + Dt) e^{-\frac{\lambda_0}{2} t} + D.
		\end{equation}
	We now define 
	\begin{equation}\label{Mbar}
		\bar{M} := 4 C_{q,\gamma} M_0 e^{2\lambda_0 \tilde{t} + \frac{2 C_{q,\gamma} M_0}{\lambda_0} e^{2\lambda_0 \tilde{t}}}.
	\end{equation}
	We also set
	\begin{equation}\label{T0}
		T_0 := \frac{4}{\lambda_0} \ln(\bar{M}/\kappa),
	\end{equation}
		where $\kappa>0$ is the constant given in Theorem \ref{thm:main2}, representing the smallness condition on the initial data. For $0 \le t \le T_0$, the estimate \eqref{proof est 2} together with \eqref{Mbar} yields
		\begin{align*}
			\Vert \pmb{h}(t) \Vert_{L^\infty} 
			\le \frac{1}{4} \bar{M} (1 + Dt) e^{-\frac{\lambda_0}{2} t} + D 
			\le \frac{1}{4} \bar{M} (1 + 2D) e^{-\frac{\lambda_0}{4} t} + D,
		\end{align*}
		where we used the fact that $t e^{-\frac{\lambda_0}{4} t}$ is uniformly bounded for $t \ge 0$. Recall from \eqref{D} that
\begin{align*}
D= C_{q,\gamma,\tilde{t}}
\left(\frac{1}{M^{\frac{1}{5}}}+\delta\right)
\left(1+\sup_{0\le s\le t}\|\pmb{h}(s)\|_{L^\infty}\right)^3 + C_{q,\gamma,M,\delta,\tilde{t}}\left[\mathcal{E}(\pmb{F}_0)^{\frac{2}{25}}
+\mathcal{E}(\pmb{F}_0)^{\frac{1}{5}}\right]\left(1+\sup_{0\le s\le t}\|\pmb{h}(s)\|_{L^\infty}\right)^3.
\end{align*}
Using the a priori bound 
\[
\sup_{0\le s\le t}\|\pmb{h}(s)\|_{L^\infty}\le \bar M,
\]
we first choose $M\geq 1$ sufficiently large and $\delta>0$ sufficiently small, depending on $q,\gamma,\tilde t$ and $\bar M$, so that
\[
C_{q,\gamma,\tilde{t}}
\left(\frac{1}{M^{\frac{1}{5}}}+\delta\right)(1+\bar M)^3
\le \frac12 \min\left\{\frac{\bar M}{8},\frac{\kappa}{4},\frac14\right\}.
\]
Then, for these fixed $M$ and $\delta$, we choose
\[
\e_2=\e_2(M,\delta,\bar M,\tilde t,\kappa)>0
\]
sufficiently small so that, whenever
\[
\mathcal{E}(\pmb{F}_0)\le \e_2,
\]
we have
\[
C_{q,\gamma,M,\delta,\tilde{t}}
\left[
\mathcal{E}(\pmb{F}_0)^{\frac{2}{25}}
+\mathcal{E}(\pmb{F}_0)^{\frac{1}{5}}
\right](1+\bar M)^3
\le \frac12 \min\left\{\frac{\bar M}{8},\frac{\kappa}{4},\frac14\right\}.
\]
Consequently,
\[
D \le \min\left\{\frac{\bar M}{8},\frac{\kappa}{4},\frac14\right\}.
\]
Therefore, from the previous estimate, we obtain the closed a priori bound
		\begin{equation}\label{proof est 3}
			\Vert \pmb{h}(t) \Vert_{L^\infty} \le \frac{3}{8} \bar{M} e^{-\frac{\lambda_0}{4} t} + \frac{\kappa}{4}  \le \frac{1}{2} \bar{M}, \qquad 0 \le t \le T_0.
		\end{equation}
		\indent To extend the solution up to time $T_0$, we use Lemma \ref{local in time}. 
		Since $\|\pmb{h}_0\|_{L^\infty} \le  M_0 \leq \frac{\bar{M}}{4}$, there exists a local existence time $\hat{t}_0= \hat{t}(M_0) > 0$ such that the solution exists on $[0, \hat{t}_0]$
		satisfying 
		\[
		\sup_{0 \le t \le \hat{t}_0} \|\pmb{h}(t)\|_{L^\infty} \le 2\|\pmb{h}_0\|_{L^\infty} \le \frac{\bar{M}}{2}.
		\]
		Assume that $\|\pmb{h}(t)\|_{L^\infty} \le \bar{M}$ holds for $t \in [0,\hat{t}_0]$. 
Then, by \eqref{proof est 3}, we actually have
\[
\|\pmb{h}(t)\|_{L^\infty} \le \frac{\bar{M}}{2}.
\]
This uniform bound allows us to restart the local existence argument at $t = \hat{t}_0$ with the same local time increment $\hat{t}_1:=\hat{t}(\frac{\bar{M}}{2})$. Specifically, since $\|\pmb{h}(\hat{t}_0)\|_{L^\infty} \le \bar{M}/2$, the solution can be extended to the interval $[\hat{t}_0, \hat{t}_0+\hat{t}_1]$ satisfying 
\begin{align*}
\|\pmb{h}(t)\|_{L^\infty} \le \bar{M}, \quad \forall t\in [\hat{t}_0, \hat{t}_0+\hat{t}_1].
\end{align*}		
By repeating this procedure finitely many times, we extend the solution to the entire interval $[0,T_0]$.  Throughout this process, the a priori estimate \eqref{proof est 3} ensures that the $L^\infty$ norm never exceeds $\bar{M}/2$, thereby preventing the solution from blowing up and guaranteeing that the local existence time does not shrink to zero. Hence, this establishes the existence, uniqueness, and the uniform bound \eqref{proof est 3} on the entire interval $[0, T_0]$.\\
		\indent By \eqref{proof est 2} with $t=T_0$ and \eqref{T0}, we obtain
		\[
		\Vert \pmb{h}(T_0)\Vert_{L^\infty}
		\le 
		\frac{3}{8}\kappa + \frac14 \kappa
		< \kappa,
		\]
		where we have used $D \le \kappa/4$.  
		Therefore, we can apply Theorem \ref{thm:main2} to obtain global existence and exponential decay:
		\begin{align} \label{proof est 4} 
		\Vert \pmb{h}(t)\Vert_{L^\infty}
		\le 
		C e^{-\lambda_1 (t-T_0)} 
		\Vert \pmb{h}(T_0)\Vert_{L^\infty} \leq C \kappa e^{-\lambda_1(t-T_0)} ,
		\qquad t \ge T_0,
		\end{align}
		where $\lambda_1$ is given in Theorem \ref{thm:main2}. Hence, by choosing $\vartheta= \min  \left \{ \lambda_1, \frac{\lambda_0}{4} \right \}$, we deduce from \eqref{tilde_t},\eqref{Mbar},\eqref{proof est 3}, and \eqref{proof est 4} that 
		\begin{align*}
			\Vert \pmb{h}(t) \Vert_{L^\infty} &\leq C\bar{M} e^{-\vartheta t} \\
			&\leq C_{q,\gamma} M_0 e^{2\lambda_0 \tilde{t} + \frac{2 C_{q,\gamma} M_0}{\lambda_0} e^{2\lambda_0 \tilde{t}}}e^{-\vartheta t}\\
			&\leq C_{q,\gamma} M_0^2 \exp \left\{ \frac{C_{q,\gamma,\nu}M_0^2}{\lambda_1}\right\} e^{-\vartheta t} , \quad t\geq 0,
		\end{align*}
		where $C_{q,\gamma,\nu}:= C_{q,\gamma} C_{\nu,1}>0$, and we used $\lambda_1 \leq \lambda_0$. 
		\end{proof}
		\noindent{\bf Data availability:} No data was used for the research described in the article.
\newline

\noindent{\bf Conflict of interest:} The authors declare that they have no conflict of interest.\newline

\noindent{\bf Acknowledgement}
G.Ko is supported by the National Natural Science Foundation of China (No. 12288201). M.-S. Lee was supported by Basic Science Research Programs through the National Research Foundation of Korea (NRF) funded by the Ministry of Education (RS-2023-00244475 and RS-2024-00462755). S.-J. Son is
supported by the National Research Foundation of Korea(NRF) grant funded by the Korea government(MSIT
and MOE)(No.RS-2023-00219980 and No.RS-2025-25419038).

	\end{document}